\documentclass[12pt,reqno]{amsart}
\usepackage{amsmath}
\usepackage{amssymb}
\usepackage{amstext}
\usepackage{mathrsfs}
\usepackage{a4wide}
\usepackage{graphicx}
\usepackage{bbm}
\usepackage{verbatim}
\usepackage{bm}
\usepackage{graphicx}
\usepackage{tikz}
\usepackage{pgfplots}
\usepackage{titletoc}
\usepackage[T1]{fontenc}

\allowdisplaybreaks \numberwithin{equation}{section}
\usepackage{color}
\usepackage{cases}
\usepackage{hyperref}

\numberwithin{equation}{section}

\newtheorem{theorem}{Theorem}[section]
\newtheorem{proposition}[theorem]{Proposition}
\newtheorem{corollary}[theorem]{Corollary}
\newtheorem{lemma}[theorem]{Lemma}

\theoremstyle{definition}

\theoremstyle{remark}
\newtheorem{remark}[theorem]{Remark}

\begin{document}
		
	\title[Helical stationary states of Euler]{Piecewise smooth stationary  Euler flows with support in a neighborhood of a helix} 
	
	\author{Daniel Peralta-Salas, Jie Wan}

	\address{Instituto de Ciencias Matem\'aticas, 
		Consejo Superior de Investigaciones Cient\'{\i}ficas, 
		Madrid, Spain}
	\email{dperalta@icmat.es}

	\address{School of Mathematics and Statistics, Beijing Institute of Technology, Beijing,  P.R. China}
	\email{wanjie@bit.edu.cn}

\maketitle

\begin{abstract}
We construct stationary solutions of the three-dimensional incompressible Euler equations with helical symmetry and support in a neighborhood of a helix. The solutions are piecewise smooth and arise from a nonlinear overdetermined elliptic boundary value problem associated with a stream-function formulation. A distinguishing feature is that the vortex cross-sections are intrinsically anisotropic: after rescaling, the leading-order shape is elliptic rather than radial, and the boundary exhibits a nontrivial third Fourier mode reflecting helical effects absent in previous axisymmetric constructions. A key step in the proof is the analysis of a genuinely anisotropic overdetermined elliptic problem with prescribed Dirichlet and nonconstant Neumann conditions.  
	
\end{abstract}
\maketitle \small{\bf Keywords:} Incompressible Euler equations; Helical symmetry; Compact velocity; Overdetermined problem.   \\		 

\tableofcontents

\section{Introduction}

The problem of constructing stationary solutions to the three-dimensional incompressible Euler
equations 
\begin{equation}\label{eq:euler}
u\cdot \nabla u+\nabla p=0,\qquad \nabla\cdot u=0
\end{equation}
with compact support has received much attention in modern fluid dynamics.  In two  dimensions, compactly supported stationary solutions can be produced explicitly through stream function representations.

However, in three dimensions the situation is dramatically different: a combination of vorticity stretching, nonlocal pressure coupling, and geometric constraints lead to strong obstructions for  steady flows with compact support. On the one hand, classical rigidity results exclude large classes of smooth compactly supported configurations, including Beltrami fields and several symmetric ansatzs. Nadirashvili \cite{Nadirashvili} and Chae–Constantin \cite{ChaeConstantin} have shown that no compactly supported steady solutions of 3D Euler of Beltrami type exist, and that in fact there are not even any Beltrami fields with finite energy. Axisymmetric stationary Euler flows of compact support without swirl do not exist \cite{JiuXin} either. Note that the existence of $ C^{1, \alpha}  $ axisymmetric stationary Euler flows whose vorticity is compactly supported was proved in~\cite{FraenkelBerger}. On the other hand, convex integration produces highly irregular compactly supported $C^\alpha$ weak solutions to the stationary Euler equations~\cite{EPP25}  (see also~\cite{ChoffrutSzekelyhidi} for the construction of $ L^\infty $ solutions).

A central development in the smooth regime was the construction by Gavrilov \cite{Gavrilov} of compactly supported smooth stationary Euler flows. 
These solutions, which are actually axisymmetric with swirl, demonstrate that compact support is possible under an axisymmetric geometric framework. Subsequently, Constantin, La, and Vicol \cite{ConstantinLaVicol} reformulated and extended such a construction using the Grad–Shafranov equation, emphasizing the role of $ localizability $, namely the condition
\begin{equation*} 
u\cdot\nabla p=0,
\end{equation*}
which means that the pressure remains
constant along streamlines, and holds in those examples. Gavrilov also constructed \cite{GavrilovHelical} stationary Euler fields which are localizable and that are supported in a neighborhood of a helix.

All the aforementioned examples concern smooth (compactly supported) steady Euler flows that are localizable. A striking development was made by Dom\'inguez-V\'azquez, Enciso, and Peralta-Salas \cite{DominguezVazquezEncisoPeraltaSalas2021}, who established the existence of a compactly supported Euler flow that is non-localizable. Their construction is based on piecewise smooth axisymmetric stationary Euler flows supported on toroidal domains and relies on solving an overdetermined boundary value problem arising from the Grad--Shafranov formulation of the axisymmetric Euler equations, in which both Dirichlet data and nonconstant Neumann data are prescribed. The resulting velocity field is compactly supported in a thin tubular neighborhood of a circle, with cross-sections that are approximately disk-shaped. 
Beyond providing a concrete example, this construction reveals that compact support and non-localizability are compatible within the axisymmetric setting. This naturally raises the question of whether non-localizability is merely a consequence of axisymmetric geometry or instead reflects a more robust phenomenon in fluid dynamics. In particular, does there exist a piecewise smooth Euler flow that is non-localizable and admits a genuinely helical, rather than axisymmetric, geometric structure?

Our objective in this paper is to construct a class of stationary weak solutions that exhibit a     mechanism of non-localization, 
  based on helical geometry. The solutions are concentrated in a tubular neighborhood of a smooth
helical curve in $ \mathbb{R}^3 $ and are compactly supported in the two transverse directions. 
The solutions extend periodically along a helical direction
in space, while the transverse profile is governed by an elliptic balance law. Our solutions are   different from those obtained by Gavrilov \cite{GavrilovHelical}: they
are piecewise smooth stationary weak solutions with helical symmetry, and they are
non-localizable.  They  are also  different from those obtained by \cite{DominguezVazquezEncisoPeraltaSalas2021}: the support of the velocity field is  a small tubular neighborhood of a helix, and its cross-section is
approximately an ellipse rather than a disk.  Each stationary solution we construct is bounded  and supported on a helical symmetric  domain with a smooth boundary, and the flow is smooth in the interior of this domain (up to the boundary) and possibly discontinuous
across the boundary. 

A related fundamental problem is the desingularization of helical vortex filaments, namely, whether one can construct three-dimensional incompressible Euler flows whose vorticity concentrates near helices and whose leading-order dynamics are described by the vortex filament equations. In the time-dependent setting, \cite{DonatiLacaveMiot} established such a desingularization result over a suitable time interval. For steady flows, \cite{DDMW2022} first constructed smooth helically symmetric vortices, while \cite{CaoWan2023} and \cite{CaoWan2024} further obtained helically symmetric $C^k$ vortices with compactly supported cross-sections and vortex patch solutions in helical domains with bounded cross-section, respectively. Related constructions for other models, including the Ginzburg–Landau and Gross–Pitaevskii equations, were obtained in \cite{DDMR1} and \cite{DDMR2}. The solutions constructed here differ from these previous works in two essential ways. First, both the vorticity and velocity fields have compactly supported cross-sections, whereas previous Euler constructions only impose such localization on the vorticity. Second, our solutions possess a nontrivial swirl component, which substantially changes the structure of the vortex and introduces additional analytical challenges.
 
Our construction is based on proving the existence of nontrivial solutions to an elliptic boundary value problem with both Dirichlet and Neumann data, which is commonly called overdetermined. To see how overdetermined boundary problems appear in this context, let us start by introducing the Grad--Shafranov formulation of the helically symmetric Euler equation in three dimensions. Let $ h>0 $. As demonstrated in Section 2, a helically symmetric solution to the Euler equation (with pitch $ h $) in Cartesian coordinates $(x_1,x_2,x_3)$ (in terms of the canonical orthonormal basis $\{\mathbf{e}_1, \mathbf{e}_2, \mathbf{e}_3\}$) is
\begin{equation}\label{2-5-1}
\begin{split}
u=&\frac{1}{|\vec{\xi}|^2}\left(-x_1x_2\partial_{x_1}\psi+(h^2+x_1^2)\partial_{x_2}\psi+F x_2 \right)\mathbf{e}_1+ \frac{1}{|\vec{\xi}|^2}\left(-(h^2+x_2^2)\partial_{x_1}\psi+x_1x_2\partial_{x_2}\psi-F x_1 \right)\mathbf{e}_2\\
&+\frac{1}{|\vec{\xi}|^2}\left(-hx_1 \partial_{x_1}\psi-hx_2\partial_{x_2}\psi+F h \right)\mathbf{e}_3,
\end{split} 
\end{equation}
where $ |\vec{\xi}|^2=x_1^2+x_2^2+h^2 $, and the function $\psi(x_1,x_2)$ satisfies the equation
\begin{equation}\label{Grad Shaf eqs} 
\nabla\cdot(K(x_1,x_2)\nabla\psi)=-\frac{2h^2F(\psi)}{|\vec{\xi}|^4}-\frac{F'(\psi)F(\psi)}{|\vec{\xi}|^2}+H'(\psi) 
\end{equation}
for some functions $F$ and $H$. The coefficient matrix is explicitly defined as 
	\begin{equation*} 
K(x_1,x_2): =\frac{1}{h^2+x_1^2+x_2^2}
\begin{pmatrix}
h^2+x_2^2 & -x_1x_2 \\
-x_1x_2 &  h^2+x_1^2
\end{pmatrix}.
\end{equation*}
The pressure is then given by
\begin{equation}\label{eq:pressure} 
p = H(\psi)-\frac1{2|\vec{\xi}|^2}
\left[(h^2+x_2^2)(\partial_{x_1}\psi)^2+(h^2+x_1^2)(\partial_{x_2}\psi)^2
-2x_1x_2\partial_{x_1}\psi\,\partial_{x_2}\psi+F(\psi)^2\right].
\end{equation}
The functions $F$ and $H$ can be picked freely.

For $ \theta\in\mathbb{R} $,   the rotation by an angle $\theta$ around the $x_3$-axis is defined as $ Q_\theta:=\begin{pmatrix}
R_\theta & 0  \\
0 & 1
\end{pmatrix} $ with $ R_\theta:=\begin{pmatrix}
\cos\theta & \sin\theta  \\
-\sin\theta &\cos\theta
\end{pmatrix}. $
The first observation, which is similar to the axisymmetric case in \cite{DominguezVazquezEncisoPeraltaSalas2021}, is the following:
\begin{lemma}\label{lem:overdet}
	Let $\varOmega_0\subset\mathbb{R}^2$ be a bounded $C^2$ domain.  Suppose that $\psi\in C^1(\varOmega_0)$ solves \eqref{Grad Shaf eqs} and
	\begin{equation}\label{eq:dirichlet}
	\psi=0\quad\hbox{on }\partial\varOmega_0,
	\end{equation}
	together with the   Bernoulli-Neumann boundary condition
	\begin{equation}\label{eq:bernoulli-bc}
	\frac1{|\vec{\xi}|^2}
	\left[(h^2+x_2^2)(\partial_{x_1}\psi)^2+(h^2+x_1^2)(\partial_{x_2}\psi)^2
	-2x_1x_2\partial_{x_1}\psi\,\partial_{x_2}\psi+F(0)^2\right]=c 
	\quad\hbox{on }\partial\varOmega_0,
	\end{equation}
where $ c $ is a constant.	Let $u(x_1,x_2,x_3)=Q_{\frac{x_3}{h}}u^*\left (R_{-\frac{x_3}{h}}(x_1,x_2)\right )$ and $ p(x_1,x_2,x_3)=p^*\left (R_{-\frac{x_3}{h}}(x_1,x_2)\right )$, where
	\begin{equation*}
	u^*=\begin{cases}
	\hbox{the field in \eqref{2-5-1}},&x\in\varOmega_0,\\
	0,&x\notin\varOmega_0,
	\end{cases}
	\qquad
	p^*=\begin{cases}
	\hbox{the pressure in \eqref{eq:pressure}},&x\in\varOmega_0,\\
	H(0)-c/2,&x\notin\varOmega_0.
	\end{cases}
	\end{equation*}
	Then $ (u,p) $ is a weak helically symmetric solution of the stationary Euler equations in $\mathbb{R}^3$.
\end{lemma}

Therefore, the key step is to establish the existence of nontrivial solutions to a nonstandard overdetermined boundary value problem associated with a semilinear elliptic equation in two variables. The study of overdetermined elliptic problems originates in the celebrated work of Serrin \cite{Serrin}, which revealed the strong geometric rigidity imposed by the simultaneous prescription of Dirichlet and Neumann boundary data. Since then, overdetermined problems have evolved into a powerful tool for constructing solutions to nonlinear elliptic equations; see \cite{DelaySicbaldi,DominguezEncisoPeraltaSalas,DominguezVazquezEncisoPeraltaSalas2023,PacardSicbaldi}. 
In a different direction, overdetermined elliptic problems have recently emerged in the study of stationary Euler flows in two dimensions, where compactly supported solutions are closely related to Schiffer-type boundary value problems. This perspective has led to a number of striking developments concerning nonradial stationary flows and annular Schiffer domains; see, for instance, \cite{EncisoFernandezRuizSicbaldi2025,GomezSerranoParkShi2025}.

The overdetermined problem arising from \eqref{Grad Shaf eqs} together with \eqref{eq:dirichlet} and \eqref{eq:bernoulli-bc} exhibits several genuinely new features that place it beyond the scope of the existing theory. The first one is the presence of anisotropy. Unlike the setting of \cite{DominguezVazquezEncisoPeraltaSalas2021}, where the leading-order profile is radially symmetric, the matrix $K$ introduces a preferred set of directions and breaks rotational invariance. As a consequence, the natural approximate solutions are no longer close to Euclidean balls but rather to ellipsoidal configurations whose geometry reflects the helical twisting encoded in $K$. A second, equally distinctive phenomenon concerns the asymptotic geometry of the free boundary. In \cite{DominguezVazquezEncisoPeraltaSalas2021}, the boundary deformation converges to zero at leading order. In contrast, for the present problem the deformation of the cross-section $B_\varepsilon(\theta)$ converges to a nontrivial profile
$ B(\theta)=C^*\cos 3\theta +t^* $ for suitable constants $C^*$ and $t^*$. The appearance of this persistent zeroth and third Fourier mode reflects a genuine geometric resonance induced by the helical structure and forces a much more delicate analysis. To the best of our knowledge, neither this anisotropic mechanism nor the emergence of a nonvanishing limiting deformation has an analogue in the previously known existence results for overdetermined semilinear elliptic problems. Actually, we have the following flexible, very explicit existence theorem: 
\begin{theorem}\label{thm:main}
	Let $h>0$. 
	Consider $\widehat{F},H\in C^s((-1,0])$ satisfying
	\begin{align*}
	\widehat{F}(0)=  \widehat{F}'(0)=0,
	\qquad H'(0)>0,
	\end{align*}
	where $s>2$ is not an integer. Then:
	\begin{enumerate}
		\item For each small $\varepsilon>0$ and any $R>0$, there exists a nontrivial piecewise $ C^s $, helically symmetric compactly supported stationary Euler flow $u$ (with pitch $ h $) of the form described in Lemma \ref{lem:overdet} for  a suitable $ C^{s+1} $  planar domain $\varOmega_{R,\varepsilon}$.
		\item The functions defining the solution are
		\begin{align}\label{eq:main-F}
		F(\psi):= \left( \varepsilon^2F_R+  \widehat{F}(\psi)\right) ^{\frac{1}{2}} , 
		\end{align}
		and  $ H(\psi) $,	where $F_R $ is a positive constant depending on $R$.
		\item 
		The domain $\varOmega_{R,\varepsilon}$ is a small deformation of an ellipse centered   at $(R,0)$ of  the $x_1$--$x_2$ plane,
		\begin{equation*} 
		\varOmega_{R,\varepsilon}
		=\left\{(x_1,x_2):
		x_1=R+\varepsilon\sigma_1\rho\cos\theta,
		\quad x_2=\varepsilon\sigma_2\rho\sin\theta,
		\quad 0\le \rho<1+\varepsilon B_\varepsilon(\theta)
		\right\},
		\end{equation*}
		where
		\begin{equation*} 
		\sigma_1=h\sqrt{h^2+R^2},
		\qquad \sigma_2=h^2+R^2.
		\end{equation*}
		Moreover,
		\begin{equation*} 
		B_\varepsilon(\theta)=B^*(\theta)+O(\varepsilon) \  \ \text{in}\ C^{s+1}(\mathbb{T}),
		\qquad
		B^*(\theta)=C^*\cos 3\theta+t^*,
		\end{equation*}
with the constants 
		\begin{equation*} 
		C^*=\frac{R^3\sigma_1^{-1}}{8},
		\qquad
		t^*=\frac{10}{9\sigma_2^2H'(0)}h^2\sqrt{F_R}.
		\end{equation*}
		\item The stream function $\psi$ belongs to $C^{s+1}$ in $\varOmega_{R,\varepsilon}$ up to the boundary. Moreover, $\varepsilon^2F_R+\widehat F(\psi)>0$,   $F (\psi) > 0$ and $H  (\psi)$ are of class $C^s$ in $\varOmega_{R,\varepsilon}$.
	\end{enumerate}
		
\end{theorem}

	\begin{remark}\label{rem:constants}
	In fact, the value of the constants and the structure of the solutions is completely explicit:
	\begin{enumerate}
		\item The boundary $\partial\varOmega_{R,\varepsilon}$ is  approximately an ellipse  defined by an equation of the form $$\frac{(x_1-R)^2}{\sigma_1^2 }+\frac{x_2^2}{\sigma_2^2}-\varepsilon^2=O(\varepsilon^3).$$
		\item The constant $F_R$ is
		\[
		F_R := \frac1{16} \left (h^2+R^2\right )^3H'(0)^2 > 0 ,
		\]
		and the constant $c$ in the Neumann condition is of the form
		\[
		c = \frac{5 }{16}\left (h^2+R^2\right )^2H'(0)^2\varepsilon^2  + O(\varepsilon^3) .
		\]
		\item The function $\psi$ is of the form
		\[
		\psi = \frac14 \left [ \left( h^2+R^2\right)^2 H'(0)  \right ]
		\left [ \frac{(x_1-R)^2}{\sigma_1^2 }+\frac{x_2^2}{\sigma_2^2}-\varepsilon^2 \right ] + O(\varepsilon^3) .
		\]
		\item  The vorticity,
		\begin{equation*} 
		\begin{split}
\vec{w}=&\left( \frac{x_2}{h}\left( \frac{F'(\psi)F(\psi)}{|\vec{\xi}|^2}-H'(\psi) -\frac{F'(\psi)\left( x_1\partial_{x_1}\psi+x_2\partial_{x_2}\psi\right) }{|\vec{\xi}|^2} \right)+\frac{F'(\psi)\partial_{x_2}\psi}{h}\right)  \textbf{e}_1\\
&-\left( \frac{x_1}{h}\left( \frac{F'(\psi)F(\psi)}{|\vec{\xi}|^2}-H'(\psi) -\frac{F'(\psi)\left( x_1\partial_{x_1}\psi+x_2\partial_{x_2}\psi\right) }{|\vec{\xi}|^2} \right)+\frac{F'(\psi)\partial_{x_1}\psi}{h}\right)  \textbf{e}_2\\
&+\left( \frac{F'(\psi)F(\psi)}{|\vec{\xi}|^2}-H'(\psi) -\frac{F'(\psi)\left( x_1\partial_{x_1}\psi+x_2\partial_{x_2}\psi\right) }{|\vec{\xi}|^2} \right) \textbf{e}_3,
		\end{split}
		\end{equation*}
		 is also of class $C^{s-1}$ up to the boundary.
	\end{enumerate}
\end{remark}
Several comments are in order. The theorem gives more than a perturbative existence result. It identifies the leading anisotropic scale of the cross-section and the first nontrivial boundary correction. The term $C^*\cos 3\theta+t^*$ is a genuinely helical effect: it cannot be removed by translating, rescaling, or rotating the leading ellipse. Thus the constructed vortex tubes are not merely small elliptic tubes, but carry an intrinsic threefold modulation generated by the interaction between the helical geometry and the overdetermined Bernoulli-Neumann condition.

Furthermore, our results provide a rigorous realization of a class of helical Kelvin waves with compactly supported cross-sections, thereby establishing a direct connection with the conjecture on the existence of helical Kelvin waves. It has been conjectured that perturbations of a helical vortex tube may give rise to three-dimensional incompressible Euler flows with helical symmetry \cite{LucasDritschel}. This conjecture was  addressed in the absence of swirl, where \cite{CaoFanLiQin2026} constructed helically symmetric solutions generated by such perturbations. In contrast, Theorem \ref{thm:main} establishes the existence of smooth, compactly supported helical Kelvin waves in the presence of swirl, thereby extending the previous theory to a substantially more general setting.


The paper is organized as follows: in Section \ref{sec:Grad} we will derive the Grad--Shafranov formulation \eqref{Grad Shaf eqs} of the three-dimensional helically symmetric Euler equation and prove Lemma \ref{lem:overdet}. The rest of the paper is devoted to the proof of Theorem \ref{thm:main}.
In Section \ref{sec:Dirichlet} we introduce anisotropic scaling transformation and construct solutions to the corresponding  scaled elliptic Equation \eqref{eq:phi} with Dirichlet condition \eqref{eq:Dirichlet_varphi}. 
Asymptotic expansion of the solution in the parameter $\varepsilon$ is computed in Section
\ref{sec:phi-expansion}. The  functional $ \mathcal{F}(\varepsilon,B) $ of the scaled Neumann condition \eqref{eq:Bernoulli-expanded} is estimated in Section \ref{sec:neumann-analysis}. The key is that we need to prove the existence of  the leading deformation $ B^* $ satisfying $ \mathcal{F}(\varepsilon,B^*)=\kappa+O(\varepsilon^2) $. 
The way these solutions change when the domain is perturbed near $ B^* $  is
discussed in Section \ref{sec:variation}.  This result, which we prove in Section \ref{sec: invertibility}, allows us to
complete the proof of Theorem \ref{thm:main}. We briefly discuss the existence of compactly supported solutions defined by functions $ F $ and $ H $ different from those considered in Theorem \ref{thm:main} in Section \ref{sec:different}.

\section{The Grad-Shafranov formulation}\label{sec:Grad}
In this section, we will derive the Grad-Shafranov formulation for stationary solutions of \eqref{eq:euler} with helical symmetry, in terms of the Cartesian coordinates $ (x_1,x_2,x_3) $.  Notice that \cite{GavrilovHelical} deduced a formulation for helically symmetric steady flows in cylindrical coordinates $ (r,\theta,z) $. A similar 2D model for time dependent helically symmetric solutions of \eqref{eq:euler} with swirl was given in \cite{QinWan}. 

\subsection{Some notations and introduction of a stream function}
Let us start by introducing some notations of helically symmetric functions and vector fields. Let $ x=(x_1,x_2,x_3)^t $ in Cartesian coordinates, where $ x^t $ is the transpose of   $ x $. 
Let $ h>0 $. A transformation group is denoted as  $ \mathcal{H}_h:=\{\bar{Q}_{h,\theta}:\mathbb{R}^3\to\mathbb{R}^3\mid  \theta\in\mathbb{R}\} $, where
\begin{equation*}
\bar{Q}_{h,\theta}(x):=Q_\theta(x)+\begin{pmatrix}
0  \\
0  \\ 
h\theta
\end{pmatrix}=\begin{pmatrix}
x_1\cos\theta+x_2\sin\theta  \\
-x_1\sin\theta +x_2\cos\theta  \\ 
x_3+h\theta
\end{pmatrix}.
\end{equation*}	
Helical functions and fields are defined as follows.  A function $f:\mathbb R^3\rightarrow\mathbb R $ is  helical, if $$ f(\bar{Q}_{h,\theta}(x))=f(x),  \ \ \forall\, \theta\in\mathbb{R}, x\in \mathbb R^3, $$
and a vector field $ u=(u_1,u_2,u_3)^t:\mathbb R^3\rightarrow\mathbb R^3$ is   helical, if $$ u(\bar{Q}_{h,\theta}(x))= Q_{\theta} u(x),\ \ \forall\, \theta\in\mathbb{R}, x\in \mathbb{R}^3. $$
Naturally, a   pair ($u, p$) is called a helical solution  of \eqref{eq:euler},  if ($u, p$) solves \eqref{eq:euler} and both the field $ u $ and the function $ p  $ are helical. 
Denote the tangent field of the group $ \mathcal{H}_h $
\begin{equation*}
\vec{\xi}(x):= (x_2, -x_1, h)^t.
\end{equation*}
Then the magnitude  of $ \vec{\xi} $ is $ |\vec{\xi}|=\sqrt{h^2+x_1^2+x_2^2} $. A simple computation yields that   a $ C^1 $   function $f$ and a $ C^1 $ vector field  $u$ are  helical, if and only if $$\vec{\xi}\cdot \nabla f=0,\quad \vec{\xi}\cdot \nabla  u=(u_2,-u_1,0)^t.$$
For a helical function $ f $ and a helical field $ u $, we define their projections on $ \mathbb{R}^2 $
$$u^*(x_1,x_2):=u(x_1,x_2,0), \quad f^*(x_1,x_2):=f(x_1,x_2,0).$$
Then $ (u,f) $ is uniquely determined by $ (u^*,f^*) $ by the formulas
$$ u(x_1,x_2,x_3)=Q_{\frac{x_3}{h}}u^*\left (R_{-\frac{x_3}{h}}(x_1,x_2)\right ),\quad f(x_1,x_2,x_3)= f^*\left (R_{-\frac{x_3}{h}}(x_1,x_2)\right ). $$
The circulation of $ u $  is defined as 
\begin{equation*} 
F :=u \cdot \vec{\xi}.
\end{equation*} 

Now let $  (u, p)  $ be a helical solution pair   of \eqref{eq:euler} and $ \vec{w}:=\nabla\times u $ be its vorticity field. We  define an auxiliary field and its curl field as
\begin{equation}\label{orth field} 
v=(v_1,v_2,v_3):=u-\frac{F\vec{\xi}}{|\vec{\xi}|^2},\quad \vec{\zeta}:=\nabla\times v.
\end{equation}
One computes directly that  
\begin{lemma}\label{lm2-2}
Let $  (u, p)  $ be a helical solution of \eqref{eq:euler}. Then
 \begin{itemize}
 	\item [(a)]  
 	$F$ is  a helical function.
 	\item [(b)]  
 	 $v$ is   helically symmetric, divergence free and $ v\cdot \vec{\xi}=0 $.	
 	\item [(c)]  
 	 $\vec{w}$ and $ \vec{\zeta} $ are  helically symmetric.		
 	\item [(d)]  Set $\vec{w}:=(w_1,w_2,w_3)^t $ and  $\vec{\zeta}:=(\zeta_1,\zeta_2,\zeta_3)^t $ in Cartesian coordinates. 
 	Then 
 	\begin{equation}\label{2-1}
 		\vec{w}=\frac{w_3}{h}\vec{\xi}+\frac{1}{h}(\partial_{x_2}F, -\partial_{x_1}F,0)^t,
 	\end{equation}
 and
 	\begin{equation}\label{2-2}
 		w_3=\zeta_3-\frac{2h^2 F}{|\vec{\xi}|^4}-\frac{x_1 \partial_{x_1}F+x_2\partial_{x_2}F}{|\vec{\xi}|^2}.
 	\end{equation}
 \end{itemize}

\end{lemma}
 
\begin{proof}
The calculation of \eqref{2-1} and \eqref{2-2} is obtained directly  from the definition of $ \vec{\xi}, F, \vec{w}   $ and $ \vec{\zeta} $.
\end{proof}
 
Based on this, we are able to  introduce the  ``stream function'' of   solutions to \eqref{eq:euler} under helical symmetry. 
\begin{lemma}\label{lm2-4}
Let $  (u, p)  $ be a helical solution of \eqref{eq:euler} and $ v^*=(v^*_1,v^*_2,v^*_3)^t, \vec{\zeta}^*=(\zeta_1^*,\zeta_2^*,\zeta_3^*)^t $ in Cartesian coordinates. Here $v^*$ and $\vec{\zeta}^*$ are defined as $u^*$ before. Then there exists a function  $\psi: \mathbb R^2\mapsto \mathbb R$ such that 
	   \begin{equation}\label{2-5}
	   	 \begin{pmatrix}
	   	 	v_1^*(x_1,x_2) \\
	   	 	v_2^*(x_1,x_2)
	   	 \end{pmatrix}=\frac{1}{|\vec{\xi}|^2}\begin{pmatrix}
	   	- x_1x_2 &   h^2+x_1^2  \\
	   	 -h^2-x_2^2  &  x_1x_2
   	 \end{pmatrix}  \begin{pmatrix}
   	 \partial_{x_1}\psi(x_1,x_2) \\
   	 \partial_{x_2}\psi(x_1,x_2)
    \end{pmatrix}. 
	   \end{equation}
Moreover, 
	 \begin{equation}\label{2-6}
	 	 \zeta_3^*(x_1,x_2)=\mathcal{L}\psi:=-\nabla\cdot(K(x_1,x_2)\nabla\psi)(x_1,x_2),  
	 \end{equation}
where  
the coefficient matrix
	\begin{equation*} 
		K(x_1,x_2)=\frac{1}{h^2+x_1^2+x_2^2}
		\begin{pmatrix}
			h^2+x_2^2 & -x_1x_2 \\
			-x_1x_2 &  h^2+x_1^2
		\end{pmatrix}.
	\end{equation*}
\end{lemma}
\begin{proof}
Taking  $v\cdot \vec{\xi}=0$ and   $ \vec{\xi} \cdot\nabla v_3=0$ into $\nabla \cdot v=0$, we have
\begin{equation*} 
\begin{split}
0=&\partial_{x_1} v_1+\partial_{x_2} v_2+\partial_{x_3} v_3
=\frac{1}{h^2}\partial_{x_1}  \left[(h^2+x_2^2)v_1-x_1x_2v_2\right]+\frac{1}{h^2}\partial_{x_2}  \left[(h^2+x_1^2)v_2-x_1x_2v_1\right].
\end{split}
\end{equation*}
Thus we can define a function  $\psi: \mathbb R^2\mapsto \mathbb R$ such that 
\begin{equation*} 
\begin{cases}
\partial_{x_2} \psi(x_1,x_2)&=\frac{1}{h^2}\left[(h^2+x_2^2)v_1(x_1,x_2,0)-x_1x_2v_2(x_1,x_2,0)\right],\\
\partial_{x_1} \psi(x_1,x_2)&=\frac{1}{h^2}  \left[(h^2+x_1^2)v_2(x_1,x_2,0)-x_1x_2v_1(x_1,x_2,0)\right],
\end{cases}
\end{equation*}
which implies \eqref{2-5}. Equation~\eqref{2-6} follows immediately from  \eqref{2-5} and  $\zeta_3=\partial_{x_1} v_2-\partial_{x_2} v_1$.
\end{proof}
Actually the velocity field $ u $ and the vorticity field $ \vec{w} $ can be reconstructed  via the stream function $ \psi. $ 
\begin{lemma} \label{lm2-4-1}
$ u $ and   $ \vec{w} $ can be recovered from $ \psi $ by 
$$u(x_1,x_2,x_3)=Q_{\frac{x_3}{h}}u^*\left (R_{-\frac{x_3}{h}}(x_1,x_2)\right ),\ \ \vec{w}(x_1,x_2,x_3)=Q_{\frac{x_3}{h}}\vec{w}^*\left (R_{-\frac{x_3}{h}}(x_1,x_2)\right ),$$
where
\begin{equation*} 
u^*(x_1,x_2)=\begin{pmatrix}
& \frac{1}{|\vec{\xi}|^2}\left(-x_1x_2\partial_{x_1}\psi+(h^2+x_1^2)\partial_{x_2}\psi+F x_2 \right)\\
&\frac{1}{|\vec{\xi}|^2}\left(-(h^2+x_2^2)\partial_{x_1}\psi+x_1x_2\partial_{x_2}\psi-F x_1 \right)\\
&\frac{1}{|\vec{\xi}|^2}\left(-hx_1 \partial_{x_1}\psi-hx_2\partial_{x_2}\psi+F h \right)  
\end{pmatrix},
\end{equation*}
and 
\begin{equation}\label{2-5-2}
\vec{w}^*(x_1,x_2)=\begin{pmatrix}
& \frac{x_2}{h}\left(\mathcal{L}\psi-\frac{2h^2F}{|\vec{\xi}|^4}-\frac{x_1\partial_{x_1}F+x_2\partial_{x_2}F}{|\vec{\xi}|^2} \right)+\frac{\partial_{x_2}F}{h}\\
&\frac{-x_1}{h}\left(\mathcal{L}\psi-\frac{2h^2F}{|\vec{\xi}|^4}-\frac{x_1\partial_{x_1}F+x_2\partial_{x_2}F}{|\vec{\xi}|^2} \right)-\frac{\partial_{x_1}F}{h}\\
&\mathcal{L}\psi-\frac{2h^2F}{|\vec{\xi}|^4}-\frac{x_1\partial_{x_1}F+x_2\partial_{x_2}F}{|\vec{\xi}|^2} 
\end{pmatrix}.
\end{equation}
\end{lemma}
\begin{proof}
Formula \eqref{2-5-1} comes from $ u=v+\frac{F\vec{\xi}}{|\vec{\xi}|^2} $, \eqref{2-5} and $ v\cdot \vec{\xi}=0 $. Formula \eqref{2-5-2} comes from \eqref{2-1}, \eqref{2-2} and \eqref{2-6}.
\end{proof}
We also have the following property of $ \psi $.
\begin{lemma} \label{lm2-5}
Let $g: \mathbb R^3\mapsto \mathbb R$ be a helical function. Then
	\begin{equation*}
		\left (u \cdot \nabla g\right )^* = \nabla^\perp\psi \cdot\nabla (g^*),
	\end{equation*}
where $ \nabla^\perp:=(\partial_{x_2},-\partial_{x_1}). $
\end{lemma}
\begin{proof}
From \eqref{orth field} and Lemma \ref{lm2-2}, we have
\begin{equation}\label{2-6-1}
u  \cdot \nabla g=\left(v+\frac{F\vec{\xi}}{|\vec{\xi}|^2}\right)  \cdot \nabla g=	v \cdot \nabla g. 
\end{equation}
By $v\cdot \vec{\xi}=0$ and $ \vec{\xi}\cdot \nabla g=0 $, we have
\begin{equation}\label{2-6-2}
v_3=\frac{ -x_2v_1+x_1v_2}{h},\ \ \partial_{x_3} g=\frac{-x_2\partial_{x_1} g+x_1\partial_{x_2} g}{h}.
\end{equation}
Taking  \eqref{2-5} and \eqref{2-6-2} into \eqref{2-6-1}, we get
\begin{equation*} 
\begin{split}
(u \cdot \nabla g)^*(x_1,x_2)=&(v_1(x_1,x_2,0),v_2(x_1,x_2,0))\begin{pmatrix}
1+\frac{x_2^2}{h^2} &   -\frac{x_1x_2}{h^2}  \\
-\frac{x_1x_2}{h^2} & 1+\frac{x_1^2}{h^2}
\end{pmatrix}  \begin{pmatrix}
\partial_{x_1}g(x_1,x_2,0) \\
\partial_{x_2}g(x_1,x_2,0)
\end{pmatrix}\\
=&\frac{1}{|\vec{\xi}|^2}\left (\partial_{x_1}\psi(x_1,x_2), \partial_{x_2}\psi(x_1,x_2)\right )\begin{pmatrix}
- x_1x_2 &  -h^2-x_2^2  \\ 
 h^2+x_1^2  &  x_1x_2
\end{pmatrix}\\  
&\begin{pmatrix}
1+\frac{x_2^2}{h^2} &   -\frac{x_1x_2}{h^2}  \\
-\frac{x_1x_2}{h^2} & 1+\frac{x_1^2}{h^2}
\end{pmatrix}
  \begin{pmatrix}
\partial_{x_1}g(x_1,x_2,0) \\
\partial_{x_2}g(x_1,x_2,0)
\end{pmatrix}\\
=&\left (\partial_{x_1}\psi(x_1,x_2), \partial_{x_2}\psi(x_1,x_2)\right )\begin{pmatrix}
0 &  -1 \\ 
1  & 0
\end{pmatrix}   \begin{pmatrix}
\partial_{x_1}g(x_1,x_2,0) \\
\partial_{x_2}g(x_1,x_2,0)
\end{pmatrix}\\
=&\nabla^\perp\psi(x_1,x_2)\cdot\nabla g(x_1,x_2,0).
\end{split}
\end{equation*}
\end{proof}
\subsection{Helical Grad--Shafranov reduction}
Having made   the above preparations,  we can now derive the Grad--Shafranov formulation  for the stream function defined on the transverse $ x_1 $--$ x_2 $ plane.   
We obtain directly  from \eqref{eq:euler} that $$ u\times \vec{w}=\nabla H, $$
 where $ H:=p+\frac{|u|^2}{2} $ is the Bernoulli function.
 
Note that the circulation $ F $ is a helical function. Since $ u\cdot\nabla F=u\cdot\nabla(u\cdot\vec{\xi})=-\nabla p\cdot\vec{\xi}+(u\cdot\nabla\vec{\xi})\cdot u=0, $
from Lemma \ref{lm2-5}, we have $ \nabla^\perp\psi\cdot \nabla F=0 $. This implies that $ F $ is a function of the stream function~$ \psi $,   $$ F=F(\psi). $$

The Bernoulli function $ H $ is also a helical function. Since $ u\cdot\nabla H=u\cdot(u\times \vec{w})=0, $   $ H $ is also a function of~$ \psi $,  $$ H=H(\psi). $$

It follows from the third equation of $ u\times \vec{w}=\nabla H $ that
\begin{equation}\label{2-20}
u_1w_2-u_2w_1=H'(\psi)\partial_{x_3}\psi.
\end{equation}
Taking \eqref{2-1} and \eqref{2-5-1} into Equation \eqref{2-20}, we get
\begin{equation*} 
w_3^*=-H'(\psi)+\frac{F'(\psi)}{|\vec{\xi}|^2}\left( -x_1\partial_{x_1}\psi-x_2\partial_{x_2}\psi\right) +\frac{F'(\psi)F(\psi)}{|\vec{\xi}|^2},
\end{equation*}
which combined with the third equation of \eqref{2-5-2} yields a equation for $ \psi $
\begin{equation*} 
\nabla\cdot(K(x_1,x_2)\nabla\psi)=-\frac{2h^2F(\psi)}{|\vec{\xi}|^4}-\frac{F'(\psi)F(\psi)}{|\vec{\xi}|^2}+H'(\psi).
\end{equation*}
To conclude, we have
\begin{proposition}\label{pp2-7}
The Grad-Shafranov formulation for stationary helically symmetric solutions of three dimensional incompressible Euler equations  is Equation \eqref{Grad Shaf eqs}. For a stream function $ \psi $ that solves Equation~\eqref{Grad Shaf eqs}, one can reconstruct a stationary helically symmetric solution pair $ (u,p) $ of \eqref{eq:euler} in $ \mathbb{R}^3 $ by Lemma \ref{lm2-4-1} and $ p(x_1,x_2,x_3)=p^*\left (R_{-\frac{x_3}{h}}(x_1,x_2)\right ) $ with
$$p^*=H(\psi)-\frac{|u^*|^2}{2}.$$
\end{proposition}
\subsection{A 2D overdetermined problem}
We prove that if one has a stationary helically symmetric Euler flow on a helical  domain $\varOmega$ which is tangent to the boundary and whose pressure is constant on $\partial\varOmega$, then one can trivially extend it to a stationary helically symmetric Euler flow on the whole space with compact cross-sectional support.

Let us start by recalling the definition of a weak stationary Euler flow.
We say that a pair $(u,p)$ of class  $L^2_{\mathrm{loc}}(\mathbb{R}^3)$ is a weak solution to the stationary Euler equation if
\[
\int_{\mathbb{R}^3} \bigl[ (u\otimes u) \cdot \nabla \eta + p\,\nabla\cdot \eta \bigr] \,dx = 0
\qquad\text{and}\qquad
\int_{\mathbb{R}^3} \mathbf{u} \cdot \nabla \varphi \,dx = 0
\]
for any vector field $\eta\in C^\infty_c(\mathbb{R}^3)$ and any scalar function $\varphi\in C^\infty_c(\mathbb{R}^3)$.
Of course, if $u$ and $p$ are smooth enough, this is equivalent to saying that they satisfy Equations~\eqref{eq:euler} in $\mathbb{R}^3$.

\begin{lemma}\label{lem:general}
	Given a helical domain $\varOmega$ in $\mathbb{R}^3$ whose cross-section $\varOmega_0$ is a bounded domain in $ \mathbb{R}^2 $ with  a $C^2$ boundary, suppose that $\tilde u\in C^1(\varOmega)\cap L^2_{loc}(\varOmega)$ is a helically symmetric solution to the stationary Euler equation in $\varOmega$ with pressure $\tilde p\in C^1(\varOmega)\cap L^1_{loc}(\varOmega)$.
	Assume that $\tilde u\cdot\nu |_{\partial\varOmega}=0$ and $\tilde p|_{\partial\varOmega}=c$, where $c$ is a constant,  and $ \nu $ is a unit normal field on $ \partial\varOmega $.
	Then
	\[
	u(x) :=
	\begin{cases}
	\tilde u(x) & \text{if } x\in\varOmega ,\\
	0 & \text{if } x\notin\varOmega
	\end{cases}
	\]
	is a weak helically symmetric solution of the stationary Euler equation on $\mathbb{R}^3$ with pressure
	\[
	p(x) :=
	\begin{cases}
	\tilde p(x) & \text{if } x\in\varOmega ,\\
	c & \text{if } x\notin\varOmega .
	\end{cases}
	\]
\end{lemma}

\begin{proof}
	We start by noticing that, for all $\varphi\in C^\infty_c(\mathbb{R}^3)$,
	\[
	\int_{\mathbb{R}^3} u \cdot \nabla \varphi \,dx
	= \int_{\varOmega} \tilde u \cdot \nabla \varphi \,dx
	= -\int_{\varOmega} \varphi\,\nabla\cdot \tilde u \,dx
	+ \int_{\partial\varOmega} \varphi\,\nu\cdot \tilde u \,dS = 0,
	\]
	where we have used that $\nabla\cdot \tilde u = 0$ in $\varOmega$ and $\nu\cdot \tilde u = 0$ on $\partial\varOmega$.
	Hence $\nabla\cdot u = 0$ in the sense of distributions.
	
	Let us now take an arbitrary vector field $\eta\in C^\infty_c(\mathbb{R}^3)$.
	As $\int_{\mathbb{R}^3}\nabla\cdot \eta\, dx = 0$, we can write
	\begin{align*}
	\int_{\mathbb{R}^3} \bigl[ (u\otimes u) \cdot \nabla \eta + p\,\nabla\cdot \eta \bigr] \,dx
	&= \int_{\varOmega} (u\otimes u) \cdot \nabla \eta \,dx
	+ \int_{\mathbb{R}^3} (p-c)\,\nabla\cdot \eta \,dx \\
	&= \int_{\varOmega} \bigl[ (\tilde u\otimes \tilde u) \cdot \nabla \eta - \eta \cdot \nabla \tilde p \bigr] \,dx
	+ \int_{\partial\varOmega} (\tilde p - c)\, \eta \cdot \nu \,dS \\
	&= -\int_{\varOmega} \bigl[ \nabla\cdot(\tilde u\otimes \tilde u) + \nabla \tilde p \bigr] \cdot \eta \,dx \\
	&\quad + \int_{\partial\varOmega} \bigl[ (\tilde u\cdot \eta)(\tilde u\cdot \nu) + (\tilde p - c)\, \eta \cdot \nu \bigr] \,dS .
	\end{align*}
	The volume integral is zero because $v$ satisfies the stationary Euler equation in $\varOmega$.
	Since $\tilde u\cdot \nu = \tilde p - c = 0$ on $\partial\varOmega$, the surface integral vanishes too.
	It then follows that $u$ is a weak solution of the Euler equation in $\mathbb{R}^3$, as claimed.
\end{proof}
From formula \eqref{2-5-1}, the three-dimensional velocity field reconstructed from $\psi$ 
  satisfies the identity
\begin{equation}\label{eq:kinetic}
|u^*|^2=\frac1{|\vec{\xi}|^2}
\left[(h^2+x_2^2)(\partial_{x_1}\psi)^2+(h^2+x_1^2)(\partial_{x_2}\psi)^2
-2x_1x_2\partial_{x_1}\psi\,\partial_{x_2}\psi+F(\psi)^2\right].
\end{equation}
 The pressure is
\begin{align*} 
p^*&=H(\psi)-\frac12|u^*|^2 \notag\\
&=H(\psi)-\frac1{2|\vec{\xi}|^2}
\left[(h^2+x_2^2)(\partial_{x_1}\psi)^2+(h^2+x_1^2)(\partial_{x_2}\psi)^2
-2x_1x_2\partial_{x_1}\psi\,\partial_{x_2}\psi+F(\psi)^2\right].
\end{align*}
With $u$ given by~\eqref{2-5-1}, and using the notation in Lemma~\ref{lem:general}, the conditions that $u$ be tangential to the helical domain defined by $\varOmega$ and that the pressure~\eqref{eq:pressure} be constant on the boundary amount to saying that $\psi$ takes a constant value $c_0$ on $\partial\varOmega$ and that $$ \frac1{|\vec{\xi}|^2}
\left[(h^2+x_2^2)(\partial_{x_1}\psi)^2+(h^2+x_1^2)(\partial_{x_2}\psi)^2
-2x_1x_2\partial_{x_1}\psi\,\partial_{x_2}\psi+F(\psi)^2\right] $$ is also constant on $\partial\varOmega$.
Setting $c_0:=0$ without loss of generality, Lemma~\ref{lem:general} results in Lemma~\ref{lem:overdet}.


\section{Local coordinates and the scaled elliptic equation}\label{sec:Dirichlet}

\subsection{Anisotropic scaling near a fixed point}

Let us take any $R>0$  that will be fixed during the whole construction and set
\begin{equation*}
\sigma_1=h\sqrt{h^2+R^2},\qquad \sigma_2=h^2+R^2.
\end{equation*}
We introduce suitably translated and rescaled variables 
\begin{equation*}
x_1=R+\varepsilon\sigma_1 x,
\qquad x_2=\varepsilon\sigma_2 y,
\qquad (x,y)\in\mathbb{R}^2,
\end{equation*}
where $\varepsilon>0$ is a small parameter.  The chart used for the main expansion is
\begin{equation}\label{eq:chart}
\partial_{x_1}=\varepsilon^{-1}\sigma_1^{-1}\partial_x,
\qquad \partial_{x_2}=\varepsilon^{-1}\sigma_2^{-1}\partial_y.
\end{equation}

\begin{remark}\label{rem:scale-choice}
	At $(R,0)$ the scaling is anisotropic.  The values of $\sigma_1$ and $\sigma_2$ are chosen so that, after the change of variables and the rescaling $\psi=\varepsilon^2\phi$, the leading operator is the usual Laplacian.  
\end{remark}
With Equation \eqref{eq:chart}, the Grad-Shafranov equation \eqref{Grad Shaf eqs} becomes
\begin{equation}\label{eq:local}
\begin{split}
&\bigl((R+\varepsilon\sigma_1x)^2+(\varepsilon\sigma_2y)^2+h^2\bigr)\bigl(h^2+(\varepsilon\sigma_2y)^2\bigr)
\varepsilon^{-2}\sigma_1^{-2}\partial_{xx}\psi \\
&-2\bigl((R+\varepsilon\sigma_1x)^2+(\varepsilon\sigma_2y)^2+h^2\bigr)(R+\varepsilon\sigma_1x)(\varepsilon\sigma_2y)
\varepsilon^{-2}\sigma_1^{-1}\sigma_2^{-1}\partial_{xy}\psi\\
&+\bigl((R+\varepsilon\sigma_1x)^2+(\varepsilon\sigma_2y)^2+h^2\bigr)\bigl(h^2+(R+\varepsilon\sigma_1x)^2\bigr)
\varepsilon^{-2}\sigma_2^{-2}\partial_{yy}\psi\\
&\quad=\bigl((R+\varepsilon\sigma_1x)^2+(\varepsilon\sigma_2y)^2+3h^2\bigr)\bigl((R+\varepsilon\sigma_1x)\varepsilon^{-1}\sigma_1^{-1}\partial_{x}\psi+y\partial_{y}\psi\bigr)-2h^2F(\psi)\\
&\qquad-\bigl((R+\varepsilon\sigma_1x)^2+(\varepsilon\sigma_2y)^2+h^2\bigr) \frac{\left(F^2\right)'(\psi)}{2}+\bigl((R+\varepsilon\sigma_1x)^2+(\varepsilon\sigma_2y)^2+h^2\bigr)^2 H'(\psi),
\end{split}
\end{equation} 
where $F$ and $H$ are of the form described in Theorem~\ref{thm:main}.
\subsection{The rescaled unknown and solvability of a Dirichlet problem}
Set
\begin{equation*}
\psi=\varepsilon^2\phi.
\end{equation*}
 Equation   \eqref{eq:local}  can be written in terms of $\phi$ as
\begin{equation}\label{eq:phi-0}
\begin{split}
&\bigl((R+\varepsilon\sigma_1x)^2+(\varepsilon\sigma_2y)^2+h^2\bigr)\bigl(h^2+(\varepsilon\sigma_2y)^2\bigr)
 \sigma_1^{-2}\partial_{xx}\phi \\
&-2\bigl((R+\varepsilon\sigma_1x)^2+(\varepsilon\sigma_2y)^2+h^2\bigr)(R+\varepsilon\sigma_1x)(\varepsilon\sigma_2y)
 \sigma_1^{-1}\sigma_2^{-1}\partial_{xy}\phi\\
&+\bigl((R+\varepsilon\sigma_1x)^2+(\varepsilon\sigma_2y)^2+h^2\bigr)\bigl(h^2+(R+\varepsilon\sigma_1x)^2\bigr)
 \sigma_2^{-2}\partial_{yy}\phi\\
&\quad=\varepsilon\bigl((R+\varepsilon\sigma_1x)^2+(\varepsilon\sigma_2y)^2+3h^2\bigr)\bigl((R+\varepsilon\sigma_1x) \sigma_1^{-1}\partial_{x}\psi+\varepsilon y\partial_{y}\psi\bigr)-2h^2F(\varepsilon^2\phi)\\
&\qquad-\bigl((R+\varepsilon\sigma_1x)^2+(\varepsilon\sigma_2y)^2+h^2\bigr) \frac{\left(F^2\right)'(\varepsilon^2\phi)}{2}+\bigl((R+\varepsilon\sigma_1x)^2+(\varepsilon\sigma_2y)^2+h^2\bigr)^2 H'(\varepsilon^2\phi).
\end{split}
\end{equation} 
We look for solutions to Equation~\eqref{eq:phi-0} in domains of the form
\[
\varOmega_{\varepsilon, B} := \{ (\rho,\theta)\mid\rho < 1 + \varepsilon B(\theta) \}
\]
for some function $B\in C^{s+1}(\mathbb{T})$. Here $ (\rho,\theta) $ is the planar polar coordinates,   $ x=\rho\cos\theta, y=\rho\sin\theta $.  

 Equation~\eqref{eq:phi-0} can be rewritten  as
\begin{equation}\label{eq:phi}
\begin{split}
\Delta\phi={}&\sigma_2^2c+\varepsilon\Bigl(-2R\sigma_1^{-1}xh^2\partial_{xx}\phi+2  R\sigma_1^{-1}\sigma_2y\partial_{xy}\phi-4R\sigma_1\sigma_2^{-1}x\partial_{yy}\phi\\
&
+(R^2+3h^2)R\sigma_1^{-1}\partial_x\phi 
+4R \sigma_1\sigma_2xc-2h^2\sqrt{F_R}\Bigr)
+\mathcal{R}(\varepsilon,\phi),
\end{split}
\end{equation}
where we have defined the positive constant 
$$c:=H'(0),$$
and where the remainder is
\begin{equation}\label{eq:R-structural}
\begin{split}
\mathcal{R}(\varepsilon,\phi):=&\left [-\sigma_1^{-2}\varepsilon^2(\sigma_1^2x^2+\sigma_2^2y^2)(h^2+\varepsilon^2\sigma_2^2y^2)-\sigma_1^{-2}\varepsilon^2\sigma_2^2y^2(\sigma_2+2R\varepsilon\sigma_1x)\right ]\partial_{xx}\phi \\
&+2\varepsilon\sigma_1^{-1}y\left[ (2R\varepsilon\sigma_1x+\varepsilon^2\sigma_1^2x^2+\varepsilon^2\sigma_2^2y^2)(R+\varepsilon\sigma_1x)+\varepsilon\sigma_1\sigma_2x\right] \partial_{xy}\phi  \\
&+\left [-\varepsilon^2(\sigma_1^2x^2+\sigma_2^2y^2)\left(h^2+(R+\varepsilon\sigma_1x)^2\right)-\varepsilon^2\sigma_1^2x^2(\sigma_2+2R\varepsilon\sigma_1x)-4R^2\varepsilon^2\sigma_1^{2}x^2\right ]\sigma_2^{-2}\partial_{yy}\phi \\
&+\varepsilon\left[ (2R\varepsilon\sigma_1x+\varepsilon^2\sigma_1^2x^2+\varepsilon^2\sigma_2^2y^2)
\left ((R+\varepsilon\sigma_1x)\sigma_1^{-1}\partial_x\phi+\varepsilon y\partial_y\phi\right )+(R^2+3h^2)\varepsilon(x\partial_x\phi+y\partial_y\phi)\right]  \\
&-2h^2\bigl(F(\varepsilon^2\phi)-F(0)\bigr)
-\bigl((R+\varepsilon\sigma_1x)^2+\varepsilon^2\sigma_2^2y^2+h^2\bigr) \frac{\widehat{F}'(\varepsilon^2\phi)}{2}  \\
&+\bigl((R+\varepsilon\sigma_1x)^2+\varepsilon^2\sigma_2^2y^2+h^2\bigr)^2\bigl(H'(\varepsilon^2\phi)-c\bigr)\\
&+\left [2(\varepsilon^2\sigma_1^2x^2+\varepsilon^2\sigma_2^2y^2)(\sigma_2+2R\varepsilon\sigma_1x)+(\varepsilon^2\sigma_1^2x^2+\varepsilon^2\sigma_2^2y^2)^2+4R^2\varepsilon^2\sigma_1^2x^2\right ]c.
\end{split}
\end{equation}
The remainder satisfies
\begin{equation*}
\sup_{||\phi||_{C^{s+1}(\Omega)}\le C}
||\mathcal{R}(\varepsilon,\phi)||_{C^{s-1}(\Omega)}\le C\varepsilon^2.
\end{equation*}
The Dirichlet boundary condition $\psi=0$ on $\partial\varOmega_{\varepsilon, B}$ can then be rewritten in terms of $\phi(\rho,\theta)$ as
\begin{equation}\label{eq:Dirichlet_varphi}
\phi(1+\varepsilon B(\theta),\theta) = 0 .
\end{equation}
We will say that a function $f(\rho,\theta)$ is \emph{even} if
\[
f(\rho,\theta) = f(\rho,-\theta) ,
\]
and similarly for a function $g(\theta)$.

Note that the function
\[
\phi_0:=A_0(\rho^2-1),
\qquad
A_0:=\frac{\sigma_2^2c}{4} 
\]
satisfies Equation~\eqref{eq:phi} and the Dirichlet condition~\eqref{eq:Dirichlet_varphi} when $\varepsilon=0$. It is straightforward to show that there are solutions to the problem for small $\varepsilon$ using the implicit function theorem in Banach spaces.

\begin{proposition}\label{prop:Dirichlet}
	For every $B$ with $||B||_{C^{s+1}}\le1$ and for all sufficiently small $\varepsilon $, there is a unique solution
	\begin{equation*}
	\phi=\phi_{\varepsilon,B}\simeq \phi_0\quad\hbox{in }C^{s+1}(\Omega_{\varepsilon,B}) 
	\end{equation*}
	that satisfies Equation~\eqref{eq:phi} and the Dirichlet boundary condition~\eqref{eq:Dirichlet_varphi}.
	Furthermore, $\phi_{\varepsilon,B}<0$ in $\varOmega_{\varepsilon, B}$ and $\phi_{\varepsilon,B}$ is even if $B$ is.
\end{proposition}

\begin{proof}
	Let $ \mathbb{D}$ be the unit disk. Define a map $ \Xi_{\varepsilon,B}:\mathbb{D}\mapsto \varOmega_{\varepsilon, B} $
	\begin{equation*}
	\Xi_{\varepsilon,B}(\rho,\theta):=((1+\varepsilon\chi(\rho)B(\theta))\rho,\theta),
	\end{equation*}
	where $\chi$ is a smooth cut-off function and $\chi=0$ near $\rho=0$ and $\chi=1$ near $\rho=1$.  If $\varepsilon $ is small, this map is a $C^{s+1}$ diffeomorphism and $ \Xi_{\varepsilon,B}^{-1}  $ pulls the domain $\varOmega_{\varepsilon, B}$ back to the unit disk.  

Then one can define a map
\[
	\mathfrak H : (-\varepsilon_0,\varepsilon_0) \times C^{s+1}_{\text{Dir}}(\mathbb{D}) \mapsto C^{s-1}(\mathbb{D})
\]
as
\begin{equation*} 
\begin{split}
	\mathfrak H(\varepsilon,\bar\phi) :=&
\Bigl[ \Delta\bigl( \bar\phi \circ \Xi_{\varepsilon,B}^{-1} \bigr)
- \sigma_2^2c-\varepsilon\Bigl(-2R\sigma_1^{-1}xh^2\partial_{xx}(\bar\phi \circ \Xi_{\varepsilon,B}^{-1})+2  R\sigma_1^{-1}\sigma_2y\partial_{xy}(\bar\phi \circ \Xi_{\varepsilon,B}^{-1})\\
&-4R\sigma_1\sigma_2^{-1}x\partial_{yy}(\bar\phi \circ \Xi_{\varepsilon,B}^{-1})+(R^2+3h^2)R\sigma_1^{-1}\partial_x(\bar\phi \circ \Xi_{\varepsilon,B}^{-1})\\
&
+4R \sigma_1\sigma_2xc-2h^2\sqrt{F_R}\Bigr)
-\mathcal{R}(\varepsilon,\bar\phi \circ \Xi_{\varepsilon,B}^{-1}) \Bigr] \circ \Xi_{\varepsilon B} .
\end{split}
\end{equation*}
Here
\[
C^{s+1}_{\mathrm{Dir}}(\mathbb{D}) :=
\bigl\{ \bar\phi \in C^{s+1}(\mathbb{D}) : \bar\phi|_{\partial \mathbb{D}}=0 \bigr\} .
\]
Notice that the map $\mathfrak H$ is $C^1$.  This follows from the formula for the coefficients, the regularity of $F$ and $H$, and the fact that products and compositions are continuous in H\"older spaces when $s>2$. $ \Xi_{\varepsilon,B}  $ maps the boundary $\partial\mathbb{D} $   to $\partial\varOmega_{\varepsilon, B}$, so the zero Dirichlet condition $ \bar\phi\circ\Xi_{\varepsilon,B}^{-1}|_{\partial\varOmega_{\varepsilon, B}}=0 $ becomes the fixed condition $\bar\phi|_{\partial \mathbb{D} }=0$. 
It is obvious that if $\phi$ solves $\mathfrak H(\varepsilon,\bar\phi)=0$, then $\phi$ is a solution to the Dirichlet problem~\eqref{eq:phi} and \eqref{eq:Dirichlet_varphi}.

Notice that $\mathfrak H(0,\phi_0)=0$ and the Fr\'echet derivative
\[
D_{\bar\phi} \mathfrak H(0,\phi_0) = \Delta
\]
is an invertible mapping $C^{s+1}_{\mathrm{Dir}}(\mathbb{D}) \to C^{s-1}(\mathbb{D})$ by the Schauder estimate. The implicit function theorem then gives a unique solution $\bar\phi_{\varepsilon,B}$ near $\phi_0$ on the fixed disk for any small enough $\varepsilon$ and $B$.  Pushing it forward by $\Xi_{\varepsilon,B}$ gives $\phi_{\varepsilon,B}$ on $\varOmega_{\varepsilon,B}$.   
As $B$ only appears in the problem through the product $\varepsilon B$, this is equivalent to the first part of the statement.
Also, this uniqueness property immediately implies that $\phi_{\varepsilon,B} := \bar\phi_{\varepsilon,B} \circ \Xi_{\varepsilon,B}^{-1}$ is even when $B$ is.
The property that $\phi_{\varepsilon,B}<0$ in $\varOmega_{\varepsilon,B}$ follows from Equation~\eqref{eq:phi}, i.e.,
\[
\Delta\phi_{\varepsilon,B} = \sigma_2^2c+  O(\varepsilon) > 0 ,
\]
via the maximum principle.	
	
\end{proof}

\section{Expansion of the Dirichlet solution with respect to $\varepsilon$}\label{sec:phi-expansion}

In the next proposition we compute an asymptotic expansion for the function $\phi_{\varepsilon,B}$ for small $\varepsilon$.
The constants $A_0$ and $A_1$ appearing in this expansion, which depend on $R$ but not on $\varepsilon$, will play a major role in the rest of the paper.

Throughout, we will denote by $ P_{\varepsilon B}$ the Poisson integral operator of the domain $\varOmega_{\varepsilon, B}$ in the coordinates $(\rho,\theta)$.
That is, $v(\rho,\theta) := P_{\varepsilon B} f$ denotes the only harmonic function in $\varOmega_{\varepsilon, B}$ satisfying the boundary condition
\[
v(1+\varepsilon B(\theta),\theta) = f(\theta) .
\]
Note that $P_{\varepsilon B}$ only depends on the product $\varepsilon B$, not on $\varepsilon$ and $B$ separately.
It is standard that $P_{\varepsilon B}$ defines a map $C^{s+1}(\mathbb{T}) \to C^{s+1}(\varOmega_{\varepsilon, B})$.
A convenient explicit Fourier formula for the Poisson operator in the case of the disk is
\[
P_0 f(\rho,\theta) := \sum_{n\in\mathbb{Z}} f_n \,\rho^{|n|}\, e^{in\theta}
\qquad\text{if}\quad f(\theta) = \sum_{n\in\mathbb{Z}} f_n e^{in\theta} .
\]
The first-order expansion of $ \phi_{\varepsilon,B} $ used below is
 \begin{proposition}\label{prop:first-order-dirichlet-expansion}
 	For $B$ in a bounded set of $C^{s+1}(\mathbb{T})$, the solution of Proposition~\ref{prop:Dirichlet} has the expansion
\begin{equation}\label{eq:phi-expansion}
\phi_{\varepsilon,B}=A_0(\rho^2-1)+\varepsilon\left[A_1(\rho^3-\rho)\cos\theta-\frac{h^2\sqrt{F_R}}{2}(\rho^2-1)-2A_0P_{\varepsilon B}B\right]+O(\varepsilon^2),
\end{equation}
where  $A_0=\frac{\sigma_2^2c}{4}$ and
\begin{align*}\label{eq:A1}
A_1:=\frac18\Big(&-4R\sigma_1^{-1}A_0h^2
-8R\sigma_1\sigma_2^{-1}A_0
+2(R^2+3h^2)R\sigma_1^{-1}A_0 +4R \sigma_1\sigma_2c\Big).
\end{align*} 	
Moreover, its first order derivatives are given by  
\begin{equation}\label{eq:px}
\partial_x\phi_{\varepsilon,B}={} 2A_0x+\varepsilon \left[  A_1(3x^2+y^2-1)-  h^2\sqrt{F_R}x  
-2 A_0\partial_x(P_{\varepsilon B}B)\right] 
+O(\varepsilon^2),\tag{4.1a}
\end{equation}
\begin{equation}\label{eq:py}
\partial_y\phi_{\varepsilon,B}={} 2A_0y+\varepsilon \left [2A_1xy- h^2\sqrt{F_R}y  
-2 A_0\partial_y(P_{\varepsilon B}B)\right ]
+O(\varepsilon^2),\tag{4.1b}
\end{equation}
\begin{equation}\label{eq:prho}
\partial_{\rho}\phi_{\varepsilon,B}={} 2A_0\rho+\varepsilon \left[  A_1(3\rho^2-1)\cos\theta-  h^2\sqrt{F_R}\rho  
-2 A_0\partial_{\rho}(P_{\varepsilon B}B)\right] 
+O(\varepsilon^2),\tag{4.1c}
\end{equation}
and
\begin{equation}\label{eq:Dphi-square}
|\nabla\phi_{\varepsilon,B}|^2=4A_0^2\rho^2+4\varepsilon A_0 
\left[A_1(3\rho^3-\rho)\cos\theta-h^2\sqrt{F_R}\rho^2-2A_0\rho\partial_\rho(P_{\varepsilon B}B)\right]+O(\varepsilon^2).
\tag{4.1d}
\end{equation}
The error terms are uniform in $B$ on bounded sets.
 \end{proposition}

\begin{proof}
Notice that $\partial_{xx}\phi_0=\partial_{yy}\phi_0=2A_0$, $\partial_{xy}\phi_0=0$, $\partial_x\phi_0=2A_0x$, and the equation for $\phi = \phi_{\varepsilon,B}$ is of the form
\begin{equation*} 
\begin{split}
\Delta\phi_{\varepsilon,B}={}&\sigma_2^2c+\varepsilon\Bigl(-2R\sigma_1^{-1}xh^2\partial_{xx}\phi_{\varepsilon,B}+2  R\sigma_1^{-1}\sigma_2y\partial_{xy}\phi_{\varepsilon,B}-4R\sigma_1\sigma_2^{-1}x\partial_{yy}\phi_{\varepsilon,B}\\
&
+(R^2+3h^2)R\sigma_1^{-1}\partial_x\phi_{\varepsilon,B} 
+4R \sigma_1\sigma_2xc-2h^2\sqrt{F_R}\Bigr)
+O(\varepsilon^2),
\end{split}
\end{equation*}
with the boundary condition  $ \phi_{\varepsilon,B}(1+\varepsilon B,\theta)=0 $. One can then set $\phi_1 := (\phi_{\varepsilon,B}-\phi_0)/\varepsilon$ and arrive at the equation
\[
\Delta\phi_1 = 8A_1 x -2h^2\sqrt{F_R}+ O(\varepsilon) , \qquad
\phi_1(1+\varepsilon B(\theta),\theta) = -2A_0 B(\theta) + O(\varepsilon) .
\]
	 A short computation then shows that $h := \phi_1 - \frac43 A_1 x^3+\frac{h^2\sqrt{F_R}}{2}(\rho^2-1)$ satisfies
	\[
	\Delta h = O(\varepsilon) , \qquad
	h(1+\varepsilon B(\theta),\theta) = -2A_0 B(\theta) - \frac43 A_1 \cos^3\theta + O(\varepsilon) .
	\]
	Hence
	\begin{align*}
	h &= -2A_0 P_{\varepsilon B} B - \frac43 A_1 P_{\varepsilon B}(\cos^3\theta) + O(\varepsilon) \\
	&= -2A_0 P_{\varepsilon B} B - \frac43 A_1 P_0(\cos^3\theta) + O(\varepsilon) .
	\end{align*}
	As $\cos^3\theta = \frac14\cos 3\theta + \frac34\cos\theta$, we then have
	\begin{align*}
	\phi_1 &= -2A_0 P_{\varepsilon B} B
	+ \frac{4A_1}{3} \bigl[ \rho^3\cos^3\theta - P_0(\cos^3\theta) \bigr]-\frac{h^2\sqrt{F_R}}{2}(\rho^2-1) + O(\varepsilon) \\
	&= -2A_0 P_{\varepsilon B} B
	+ \frac{A_1}{3}\Bigl[ \rho^3(\cos 3\theta+3\cos\theta)
	- \bigl( \rho^3\cos 3\theta + 3\rho\cos\theta \bigr) \Bigr]-\frac{h^2\sqrt{F_R}}{2}(\rho^2-1) + O(\varepsilon) \\
	&= -2A_0 P_{\varepsilon B} B + A_1(\rho^3-\rho)\cos\theta -\frac{h^2\sqrt{F_R}}{2}(\rho^2-1)+ O(\varepsilon) .
	\end{align*}
	This is the desired expression for $\phi$. 
The derivative formulas are obtained by differentiating the expansion.  For example,
	\[
	\partial_x\bigl((\rho^3-\rho)\cos\theta\bigr)=3x^2+y^2-1,
	\qquad
	\partial_y\bigl((\rho^3-\rho)\cos\theta\bigr)=2xy.
	\]
	Also $\partial_x(\rho^2-1)=2x$ and $\partial_y(\rho^2-1)=2y$.  These identities give \eqref{eq:px}, \eqref{eq:py} and \eqref{eq:prho}.  The formula for $ \nabla\phi_{\varepsilon,B} $ and $|\nabla\phi_{\varepsilon,B}|^2$ follows  from \eqref{eq:px} and \eqref{eq:py}.
\end{proof}

Eventually we will need to evaluate the above formulas on the boundary of the domain, that is, at $\rho = 1+\varepsilon B(\theta)$.
In this direction, recall that the   ``type I'' Dirichlet--Neumann map of the disk, defined as
$$\Lambda_0 f(\theta) : = \partial_\rho  P_0 f(1,\theta) ,$$
is the operator $C^{s+1}(\mathbb{T}) \to C^s(\mathbb{T})$ given by
\[
\Lambda_0 f(\theta)= \sum_{n\in\mathbb{Z}} f_n |n|\, e^{in\theta}
\qquad\text{if}\quad f(\theta) = \sum_{n\in\mathbb{Z}} f_n e^{in\theta} .
\]
Also, we  define  the ``type II'' Dirichlet--Neumann map of the disk  as
\[
M_0 f(\theta) : = \partial_\theta  P_0 f(1,\theta) ,
\]
which is the operator $C^{s+1}(\mathbb{T}) \to C^s(\mathbb{T})$ given by
\[
M_0 f(\theta)= i\sum_{n\in\mathbb{Z}} f_n n\, e^{in\theta}
\qquad\text{if}\quad f(\theta) = \sum_{n\in\mathbb{Z}} f_n e^{in\theta} .
\] 
Both these two Dirichlet--Neumann maps of the domain $\varOmega_{\varepsilon, B}$ are elliptic pseudodifferential operators of first order of the form
\begin{equation*}
\Lambda_{\varepsilon B} f := \partial_\rho P_{\varepsilon B} f \big|_{\rho=1+\varepsilon B}
= \Lambda_0 f + O(\varepsilon) ,
\end{equation*}
and
\begin{equation*}
M_{\varepsilon B} f := \partial_\theta P_{\varepsilon B} f \big|_{\rho=1+\varepsilon B}
= M_0 f + O(\varepsilon) ,
\end{equation*}
where the above notation can be taken to mean that the $C^s$ norm of the error is bounded by $C\varepsilon\|f\|_{C^{s+1}}$. 
Note also that the Poisson operator of the domain $\varOmega_{\varepsilon B}$ satisfies
\begin{equation*}
P_{\varepsilon B} f  =   P_{0} f  + O(\varepsilon) ,
\end{equation*}
where the above notation can be taken to mean that the $C^s$ norm of the error is bounded by $C\varepsilon\|f\|_{C^{s}}$.

Near the boundary, Proposition \ref{prop:first-order-dirichlet-expansion} gives
\begin{equation*} 
\begin{split}
\partial_x&\phi_{\varepsilon,B}\big|_{\rho=1+\varepsilon B}\\
&=2A_0(1+\varepsilon B(\theta))\cos\theta+\varepsilon\left[2A_1\cos^2\theta-h^2\sqrt{F_R}\cos\theta-2A_0\partial_x(P_{\varepsilon B}B)|_{\rho=1+\varepsilon B}\right]+O(\varepsilon^2)\\
&=2A_0 \cos\theta+\varepsilon\left[2A_0  B(\theta)\cos\theta+2A_1\cos^2\theta-h^2\sqrt{F_R}\cos\theta-2A_0\partial_x(P_{0}B)|_{\rho=1 }\right]+O(\varepsilon^2),
\end{split}
\end{equation*}
\begin{equation*} 
\begin{split}
\partial_y&\phi_{\varepsilon,B}\big|_{\rho=1+\varepsilon B}\\
&=2A_0(1+\varepsilon B(\theta))\sin\theta+\varepsilon\left[A_1\sin2\theta-h^2\sqrt{F_R}\sin\theta-2A_0\partial_y(P_{\varepsilon B}B)|_{\rho=1+\varepsilon B}\right]+O(\varepsilon^2)\\
&=2A_0 \sin\theta+\varepsilon\left[2A_0  B(\theta)\sin\theta+A_1\sin2\theta-h^2\sqrt{F_R}\sin\theta-2A_0\partial_y(P_{0}B)|_{\rho=1}\right]+O(\varepsilon^2),
\end{split}
\end{equation*}
and
\begin{equation}\label{eq:prho-phi}
\begin{split}
\partial_{\rho}\phi_{\varepsilon,B}\big|_{\rho=1+\varepsilon B}&=  2A_0(1+\varepsilon B(\theta))+\varepsilon \left[  2A_1 \cos\theta-  h^2\sqrt{F_R} 
-2 A_0\partial_{\rho}(P_{\varepsilon B}B)|_{\rho=1+\varepsilon B}\right] 
+O(\varepsilon^2) \\
&=2A_0 +\varepsilon \left[ 2A_0 B(\theta)+ 2A_1 \cos\theta-  h^2\sqrt{F_R} 
-2 A_0 \Lambda_{0}B  \right] 
+O(\varepsilon^2).
\end{split}
\end{equation}

\section{Analysis of the Bernoulli--Neumann condition}\label{sec:neumann-analysis}

In this section we  consider the Neumann condition \eqref{eq:bernoulli-bc}. 
Notice that this  condition  is equivalent to
\begin{align}\label{eq:Bernoulli-expanded}
&(h^2+x_2^2)(\partial_{x_1}\psi)^2+(h^2+x_1^2)(\partial_{x_2}\psi)^2
-2x_1x_2\partial_{x_1}\psi\,\partial_{x_2}\psi
-c(x_1^2+x_2^2+h^2)=-F(0)^2
\end{align}
on the perturbed boundary $ \partial\varOmega_{\varepsilon, B}=\{(\rho,\theta):\rho=1+\varepsilon B(\theta)\} $.  Since $\psi=\varepsilon^2\phi$, set
\begin{equation*}
c_{\varepsilon,B}:=\varepsilon^{-2}c.
\end{equation*}
Using the linear transformation \eqref{eq:chart}, this becomes the boundary condition for $ \phi $
\begin{align*}
&\sigma_1^{-2}\bigl(h^2+(\varepsilon\sigma_2y)^2\bigr)(\partial_x\phi)^2
+\sigma_2^{-2}\bigl(h^2+(R+\varepsilon\sigma_1x)^2\bigr)(\partial_y\phi)^2
-2(R+\varepsilon\sigma_1x)(\varepsilon\sigma_2y)\sigma_1^{-1}\sigma_2^{-1}\partial_x\phi\,\partial_y\phi\notag\\
&\qquad -c_{\varepsilon,B}\bigl((R+\varepsilon\sigma_1x)^2+(\varepsilon\sigma_2y)^2+h^2\bigr)
=-\varepsilon^{-2}F(0)^2.
\end{align*}
\subsection{Boundary functional and choice of the normalization constant}
Let us now define the boundary functional
\begin{equation}\label{eq:Fboundary}
\mathcal F(\varepsilon,B) (\theta):={}\mathcal F_0(\varepsilon,B) (\theta)- c_{\varepsilon,B}\Big[(R+\varepsilon\sigma_1x)^2+(\varepsilon\sigma_2y)^2+h^2\Big]_{\rho=1+\varepsilon B(\theta)},
\end{equation}
where
\begin{align*}
\mathcal F_0(\varepsilon,B) (\theta):={}&\Big[\sigma_1^{-2}\bigl(h^2+(\varepsilon\sigma_2y)^2\bigr)(\partial_x\phi)^2
+\sigma_2^{-2}\bigl(h^2+(R+\varepsilon\sigma_1x)^2\bigr)(\partial_y\phi)^2\notag\\
&\hspace{4em}-2(R+\varepsilon\sigma_1x)\varepsilon\sigma_1^{-1}y\partial_x\phi\,\partial_y\phi\Big]_{\rho=1+\varepsilon B(\theta)}.
\end{align*}
Then from the choice of the circulation $ F $, the Neumann condition is equivalent to $ \mathcal F (\varepsilon,B)(\theta)+F_R=0 $ as an identity in $\theta$.

Next we pick the constant $c_{\varepsilon,B}$ so that $\mathcal{F}(\varepsilon,B)$ is $L^2$-orthogonal to $\cos\theta$, that is, $ \int_0^{2\pi}\mathcal F_{\varepsilon,B}(\theta)\cos\theta\, d\theta=0 $.
The reason for which we do so will be clear later.
This amounts to setting
\begin{equation}\label{eq:Ceps-expression}
c_{\varepsilon,B}=\frac{c_{\varepsilon,B,1}}{c_{\varepsilon,B,2}},
\end{equation}
where, with $x=(1+\varepsilon B(\theta))\cos\theta$ and $y=(1+\varepsilon B(\theta))\sin\theta$,
\begin{equation*}
\begin{aligned}
c_{\varepsilon,B,1}:={}&\int_0^{2\pi}\Big[\sigma_1^{-2}(h^2+(\varepsilon\sigma_2y)^2)(\partial_x\phi)^2
+\sigma_2^{-2}(h^2+(R+\varepsilon\sigma_1x)^2)(\partial_y\phi)^2\\
&\hspace{7em}-2(R+\varepsilon\sigma_1x)\varepsilon\sigma_1^{-1}y\partial_x\phi\,\partial_y\phi\Big]\cos\theta d\theta,\\
c_{\varepsilon,B,2}:={}&\int_0^{2\pi}\Big[(R+\varepsilon\sigma_1(1+\varepsilon B(\theta))\cos\theta)^2
+(\varepsilon\sigma_2(1+\varepsilon B(\theta))\sin\theta)^2+h^2\Big]\cos\theta d\theta.
\end{aligned}
\end{equation*}
	The next result guarantees that this choice of $c_{\varepsilon,B}$ makes sense for all small enough $\varepsilon$, including $\varepsilon=0$, and shows that $\mathcal{F}(0,B)$ is in fact the constant
\[
\kappa :=-\frac{A_0}{4}\sigma_2c=-F_R,
\]
which depends on $R$ but not on $B$.
In what follows, we employ the notation
\[
\langle f, g\rangle := \int_0^{2\pi} f(\theta)\,g(\theta)\,d\theta
\]
for the $L^2$ product on $\mathbb{T}$.

\begin{proposition}\label{prop:F_expansion}
	For small enough $\varepsilon$ and any $B$,
\begin{align*} 
c_{\varepsilon,B} &=\frac{4A_0A_1\sigma_2^{-1}-R^3A_0^2\sigma_1^{-1}\sigma_2^{-1}}{R\sigma_1}+O(\varepsilon),
\end{align*}
 \begin{align*} 
 \mathcal{F}(\varepsilon,B)& = \kappa + O(\varepsilon).
 \end{align*}
 
\end{proposition}

\begin{proof}
	Let us assume that $\varepsilon\neq0$.	The computation of $c_{\varepsilon,B,2}$ is straightforward:
\begin{equation}\label{eq: c_2}
c_{\varepsilon,B,2}= \int_0^{2\pi} \bigl[ \sigma_2 + 2 R\sigma_1\varepsilon\cos\theta \bigr] \cos\theta\,d\theta
+ O(\varepsilon^2)
= 2\pi R\sigma_1\varepsilon  + O(\varepsilon^2) .
\end{equation}
 
For $ c_{\varepsilon,B,1} $, it stems from the first order expansion for $\partial_x\phi_{\varepsilon,B}$ and $\partial_y\phi_{\varepsilon,B}$ derived in~Proposition \ref{prop:first-order-dirichlet-expansion} that
\begin{align}\label{eq: F_0}
 \mathcal F_0&(\varepsilon,B)(\theta)\notag\\
 =& \biggl[\sigma_1^{-2}\left(h^2+O\left(\varepsilon^2\right)\right)\left(2 A_0 x+\varepsilon\left[A_1\left(3 x^2+y^2-1\right) -h^2\sqrt{F_R}x -2 A_0 \partial_x  (P_0B)\right]+O\left(\varepsilon^2\right)\right)^2  \notag\\
& +\sigma_2^{-2}\left(\sigma_2+2 R \varepsilon \sigma_1 x+O\left(\varepsilon^2\right)\right)\left(2 A_0 y+\varepsilon\left[2 A_1 x y -h^2\sqrt{F_R}y-2 A_0 \partial_y (P_0B)\right]+O\left(\varepsilon^2\right)\right)^2  \notag\\
& -2\left(R+\varepsilon \sigma_1 x\right) \varepsilon \sigma_1^{-1} y\left(2 A_0 x+\varepsilon\left[A_1\left(3 x^2+y^2-1\right) -h^2\sqrt{F_R}x -2 A_0 \partial_x  (P_0B)\right]+O\left(\varepsilon^2\right)\right) \notag\\
&  \cdot\left(2 A_0 y+\varepsilon\left[2 A_1 x y -h^2\sqrt{F_R}y-2 A_0 \partial_y (P_0B)\right]+O\left(\varepsilon^2\right)\right)\biggr]_{\rho=1+\varepsilon B(\theta)}  \notag \\
 =& \biggl[\left(\sigma_2^{-1}+O\left(\varepsilon^2\right)\right)\left(4 A_0^2 x^2+4 A_0 x \varepsilon\left[A_1\left(3 x^2+y^2-1\right)-h^2\sqrt{F_R}x - 2A_0 \partial_x (P_0B)\right]+O\left(\varepsilon^2\right)\right)  \notag\\
& +\left(\sigma_2^{-1}+2 R \varepsilon \sigma_1 \sigma_2^{-2} x  +O\left(\varepsilon^2\right)\right)\left(4 A_0^2 y^2+4 A_0 y \varepsilon\left[2 A_1 x y-h^2\sqrt{F_R}y-2 A_0 \partial_y (P_0 B)\right]+O\left(\varepsilon^2\right)\right) \notag\\
& -2 \varepsilon(R+O(\varepsilon)) \sigma_1^{-1} y \cdot\biggl(4 A_0^2 x y+\varepsilon (4 A_0 A_1 x^2 y-4 A_0^2 x \partial_y (P_0 B)+2 A_0 A_1 y\left(3 x^2+y^2-1\right)  \notag \\
&  -4 A_0^2 y \partial_x (P_0 B)  -4A_0h^2\sqrt{F_R}xy )+O\left(\varepsilon^2\right)\biggr)\bigg]_{\rho=1+\varepsilon B(\theta)} \notag\\
=&4  A_0^2\sigma_2^{-1}\left(1+\varepsilon B\right)^2+4 A_0 \sigma_2^{-1} \varepsilon\left[A_1\left(3 x^3+x y^2-x\right)-h^2\sqrt{F_R}x^2 -2 A_0x \partial_x(P_0 B)\right]_{\rho=1+\varepsilon B(\theta)}  \notag\\	
& +8 R \varepsilon \sigma_1 \sigma_2^{-2} A_0^2 \cos\theta\sin^2\theta+4 A_0 \sigma_2^{-1} \varepsilon\left[2 A_1 x y^2-h^2\sqrt{F_R}y^2-2A_0y \partial_y (P_0 B)\right]_{\rho=1+\varepsilon B(\theta)} \notag \\
& -8 R\varepsilon  \sigma_1^{-1} A_0^2\cos\theta\sin^2\theta    +O\left(\varepsilon^2\right) \notag\\
=& 4 A_0^2\sigma_2^{-1}+ 8 A_0^2 \sigma_2^{-1}B \varepsilon+4 A_0 \sigma_2^{-1} \varepsilon\left(2A_1\cos\theta-h^2\sqrt{F_R}-2A_0\Lambda_0 B \right)  \notag\\
& +8 R \varepsilon \sigma_1 \sigma_2^{-2} A_0^2 \cos\theta\sin^2\theta-8 R \varepsilon \sigma_1^{-1} A_0^2 \cos\theta\sin^2\theta+O\left(\varepsilon^2\right).
\end{align} 
This implies that 
	\begin{align*}
c_{\varepsilon,B,1} =& \int_0^{2\pi} \biggl[4 A_0^2\sigma_2^{-1}+ 8 A_0^2\sigma_2^{-1} B \varepsilon+4 A_0 \sigma_2^{-1} \varepsilon\left(2A_1\cos\theta-h^2\sqrt{F_R}-2A_0\Lambda_0 B \right) \notag\\
&\ \ \ \ \ \ \  +8 R \varepsilon \sigma_1 \sigma_2^{-2} A_0^2 \cos\theta\sin^2\theta-8 R \varepsilon \sigma_1^{-1} A_0^2 \cos\theta\sin^2\theta+O\left(\varepsilon^2\right)\biggr]\cos\theta \,d\theta \notag\\
=	& 8 A_0^2\sigma_2^{-1} \varepsilon\langle B - \Lambda_0 B, \cos\theta\rangle
+ 8A_0 A_1\sigma_2^{-1}\varepsilon  \int_0^{2\pi} \cos^2\theta\,d\theta \notag\\
&+ 8R  (\sigma_1 \sigma_2^{-2}-\sigma_1^{-1}) A_0^2\varepsilon\int_0^{2\pi} \sin^2\theta\cos^2\theta\,d\theta + O(\varepsilon^2) \notag\\
=	& 8\pi A_0 A_1\sigma_2^{-1}\varepsilon-2\pi R^3 A_0^2\sigma_1^{-1}\sigma_2^{-1} \varepsilon+ O(\varepsilon^2) , 
	\end{align*}
	where we have used that $ \int_0^{2\pi} \sin^2\theta\cos^2\theta\,d\theta=\frac{\pi}{4} $ and that
	\[
	\langle B - \Lambda_0 B, \cos\theta\rangle
	= \langle B, (1-\Lambda_0)\cos\theta\rangle = 0
	\]
	for any $B$ because $\Lambda_0$ is self-adjoint and $\Lambda_0(\cos\theta)=\cos\theta$. The $ h^2\sqrt{F_R} $ term vanishes since its coefficient $ 4 A_0 \sigma_2^{-1} \varepsilon\int_0^{2\pi} \cos \theta\,d\theta=0. $ This readily implies that $c_{\varepsilon,B}$ can be defined at $\varepsilon=0$ by continuity and yields the formula for $c_{\varepsilon,B}$ presented in the statement.
	Also, the above formulas and the definition of $ A_0 $ and $ A_1 $ immediately imply that
	\begin{align*}
	\mathcal{F}(\varepsilon,B)
	&= 	\mathcal{F}_0(\varepsilon,B)
	- c_{\varepsilon,B} \left( (R+\varepsilon\sigma_1(1+\varepsilon B(\theta))\cos\theta)^2+(\varepsilon\sigma_2(1+\varepsilon B(\theta))\sin\theta)^2+h^2\right)   \\
	&= 4 A_0^2\sigma_2^{-1}-\frac{4A_0A_1\sigma_2^{-1}-R^3A_0^2\sigma_1^{-1}\sigma_2^{-1}}{R\sigma_1}\cdot \sigma_2 + O(\varepsilon) \\
	&=-\frac{A_0}{4}\sigma_2c + O(\varepsilon),
	\end{align*}
	as claimed.
 
\end{proof}

\subsection{Choice of the leading boundary deformation $ B^* $}
We now choose $ B=B^* $ such that $ \mathcal{F}(\varepsilon,B^*)= \kappa+O(\varepsilon^2) $, from which we can use the implicit theorem to show the existence of $ B $ near $ B^* $ satisfying the Neumann condition. A striking point is that, different from \cite{DominguezVazquezEncisoPeraltaSalas2021}, we observe that $ \mathcal{F}(\varepsilon,0)\neq \kappa+O(\varepsilon^2) $, see  Equation \eqref{eq:Fboundary-6A3}. Actually, if we set
$$B^*:=C^*\cos3\theta+t^*,$$
with $ C^*:=R^3\sigma_1^{-1}/8 $ and $ t^*:=5h^2\sqrt{F_R}/(18A_0) $,
we then have 
\begin{proposition}\label{prop:choice of B_0}
 	For small enough $\varepsilon$,
 
	\[
	\mathcal{F}(\varepsilon,B^*) = \kappa + O(\varepsilon^2).
	\]
	
\end{proposition}
\begin{proof}
  
We define the ($R$-dependent) constant
\begin{equation}\label{eq:C3-def}
C_3:=\lim_{\varepsilon\to0}\frac{c_{\varepsilon,B^*}-
	\frac{4A_0A_1\sigma_2^{-1}-R^3A_0^2\sigma_1^{-1}\sigma_2^{-1}}{R\sigma_1}}{\varepsilon}.
\end{equation}
Then $ c_{\varepsilon,B^*}=	\frac{4A_0A_1\sigma_2^{-1}-R^3A_0^2\sigma_1^{-1}\sigma_2^{-1}}{R\sigma_1}+C_3\varepsilon+O(\varepsilon^2) $. From Equation \eqref{eq: F_0}, we expand the functional to the order $ \varepsilon^2 $ as
\begin{equation}\label{eq:Fboundary-6A3}
\begin{split}
\mathcal F(\varepsilon,B^*)={}&4 A_0^2\sigma_2^{-1}+ 8 A_0^2 \sigma_2^{-1} \varepsilon(B^*-\Lambda_0B^*)+ 8 A_0A_1 \sigma_2^{-1} \varepsilon \cos\theta    \\
& -4 A_0 \sigma_2^{-1} \varepsilon h^2\sqrt{F_R}-8 R^3 A_0^2  \sigma_1^{-1} \sigma_2^{-1} \varepsilon\cos\theta\sin^2\theta\\
&-\left( \frac{4A_0A_1\sigma_2^{-1}-R^3A_0^2\sigma_1^{-1}\sigma_2^{-1}}{R\sigma_1}+C_3\varepsilon \right)\left( \sigma_2+2R\varepsilon\sigma_1x \right)+O\left(\varepsilon^2\right)  \\
={}&\kappa+ 8 A_0^2 \sigma_2^{-1} \varepsilon(B^*-\Lambda_0B^*)+2R^3A_0^2\sigma_1^{-1}\sigma_2^{-1}\varepsilon \cos\theta-4A_0\sigma_2^{-1}\varepsilon h^2\sqrt{F_R}\\
&-8R^3A_0^2\sigma_1^{-1}\sigma_2^{-1}\varepsilon \cos\theta\sin^2\theta
+C_3\sigma_2\varepsilon+O(\varepsilon^2)\\
=&\kappa+ \left( 8 A_0^2 \sigma_2^{-1}  (B^*-\Lambda_0B^*)+2R^3A_0^2\sigma_1^{-1}\sigma_2^{-1} \cos3\theta -4A_0\sigma_2^{-1}  h^2\sqrt{F_R} 
+C_3\sigma_2\right) \varepsilon+O(\varepsilon^2).
\end{split}
\end{equation}
 
We now calculate $ C_3 $. We obtain from Equation \eqref{eq: c_2} that
\begin{equation*} 
c_{\varepsilon,B^*,2}=   2\pi R\sigma_1\varepsilon +2R\sigma_1\varepsilon^2\langle B^*  , \cos^2\theta\rangle
+ O(\varepsilon^3).
\end{equation*}
The coefficient $C_3$ is then determined from   the condition \eqref{eq:Ceps-expression}, that is, 
\begin{equation}\label{eq:C3-integral}
\begin{split}
\int_0^{2\pi}\mathcal F_0(\varepsilon,B^*)(\theta)\cos\theta\, d\theta=&\left (2\pi R\sigma_1\varepsilon +2R\sigma_1\varepsilon^2\langle B^*  , \cos^2\theta\rangle
+ O(\varepsilon^3)\right )c_{\varepsilon,B^*}  
 \\
=&\bigl(4A_0A_1\sigma_2^{-1}-R^3A_0^2\sigma_1^{-1}\sigma_2^{-1}\bigr)2\pi\varepsilon+2\pi R\sigma_1C_3\varepsilon^2 \\
&+\varepsilon^2\bigl(8A_0A_1\sigma_2^{-1}-2R^3A_0^2\sigma_1^{-1}\sigma_2^{-1}\bigr)\langle B^*  , \cos^2\theta\rangle
+O(\varepsilon^3).
\end{split}
\end{equation}
Since $C_3$ appears in the coefficient of the $\varepsilon^2$ term of $ \int_0^{2\pi}\mathcal F_0(\varepsilon,B^*)(\theta)\cos\theta\, d\theta $, we must expand $\phi_{\varepsilon,B^*}$ up to order $O(\varepsilon^3)$.   

\textbf{Step I: Second-order expansion of $ \phi_{\varepsilon,B^*} $.}
From Proposition \ref{prop:first-order-dirichlet-expansion}, it is easy to check that the ($\varepsilon$-dependent) function $ \phi_2 $ defined  as
\begin{equation}\label{eq:phi2-decomp}
\phi_{\varepsilon,B^*}=\phi_0+\varepsilon\phi_1+\varepsilon^2\phi_2,
\end{equation}
where
\begin{equation*}
\phi_0=A_0(\rho^2-1),
\qquad
\phi_1=A_1(\rho^3-\rho)\cos\theta-\frac{h^2\sqrt{F_R}}{2}(\rho^2-1)-2A_0(P_0B^*),
\end{equation*}
satisfies the equation
\begin{align}\label{eq:phi2-eq}
\Delta\phi_2={}&-2R\sigma_1^{-1}xh^2\partial_{xx}\phi_1
+2\sigma_2R\sigma_1^{-1}y\partial_{xy}\phi_1
-4R\sigma_1\sigma_2^{-1}x\partial_{yy}\phi_1\notag\\
& +(R^2+3h^2)R\sigma_1^{-1}\partial_x\phi_1+\varepsilon^{-2}\mathcal{R}(\varepsilon,\phi_{\varepsilon,B^*})+O(\varepsilon),
\end{align}
with the boundary condition
\begin{equation*}
\begin{split}
\phi_2(1+\varepsilon B^*,\theta)=&-\varepsilon^{-2}\left( \phi_0(1+\varepsilon B^*,\theta)+\varepsilon\phi_1(1+\varepsilon B^*,\theta)\right) \\
=& \varepsilon^{-2}\Bigl(-A_0\left (2\varepsilon B^*+\varepsilon^2(B^*)^2\right )-\varepsilon\Bigl[A_1(2\varepsilon B^*+\varepsilon^2(B^*)^2)(1+\varepsilon B^*)\cos\theta\\
&\ \ \ \ \ \ -\frac{h^2\sqrt{F_R}}{2}(2\varepsilon B^*+\varepsilon^2(B^*)^2)-2A_0\left( (1+\varepsilon B^*)^3C^*\cos3\theta+t^* \right) \Bigr] \Bigr)\\
=&-A_0(B^*)^2-2A_1B^*\cos\theta+h^2\sqrt{F_R}B^*+6A_0C^*B^*\cos3\theta+O(\varepsilon)\\
=&5A_0 (C^*\cos3\theta)^2 +4A_0t^* C^*\cos3\theta    -A_0t^{*2}-2A_1 B^*\cos\theta  +h^2\sqrt{F_R}  B^*+O(\varepsilon).
\end{split}
\end{equation*}
Here we have used that $ P_0B^* =C^*\rho^3\cos3\theta+t^*. $ From the order $ \varepsilon^2 $ expansion of $ \mathcal{R}(\varepsilon,\phi_{\varepsilon,B}) $ in Equation \eqref{eq:R-structural}, we observe that the main terms of $ \varepsilon^{-2}\mathcal{R}(\varepsilon,\phi_{\varepsilon,B^*}) $ are linear combinations of $ x^2,y^2 $ and 1. So Equation \eqref{eq:phi2-eq} can be written as
\begin{equation*}
\Delta\phi_2=M_1x^2+M_2y^2+M_3+M_4h^2\sqrt{F_R}x+O(\varepsilon),
\end{equation*}
with $$ M_4=6R\sigma_1^{-1}h^2-(R^2+3h^2)R\sigma_1^{-1} $$ and some explicit constants $M_1, M_2, M_3$ (depending on $R$ but not on $\varepsilon$) that are not relevant for our purposes.  
The solution to this problem is therefore of the form
\begin{align*}
\phi_2={}&\frac{M_1+M_2}{32}(\rho^4-1)
+\frac{M_1-M_2}{24}(\rho^2-1)\rho^2\cos^2\theta
+\frac{M_3}{4}(\rho^2-1)\notag\\
& +\frac{M_4}{8}h^2\sqrt{F_R}(\rho^2-1)\rho\cos\theta+ 5A_0P_0\left( (C^*\cos3\theta)^2\right)+4A_0t^* P_0\left(C^*\cos3\theta\right)\notag\\
&   -A_0t^{*2}-2A_1P_0\left(B^*\cos\theta\right) +h^2\sqrt{F_R}P_0 B^*   +O(\varepsilon).
\end{align*}

\textbf{Step II: Determination of the constant $ C_3 $.}
From Equation \eqref{eq:phi2-decomp}, the integrand of the left-hand side of formula \eqref{eq:C3-integral} satisfies
\begin{equation*}
	\begin{split}
\mathcal F_0&(\varepsilon,B^*)\\
=& \bigg[\left(\sigma_2^{-1}+\varepsilon^2 \sigma_1^{-2} \sigma_2^2 y^2\right)\left(2 A_0 x+\varepsilon\left[A_1\left(3 x^2+y^2-1\right)-h^2 \sqrt{F_R} x-2 A_0 \partial_x  P_0 B^*\right]+\varepsilon^2\partial_x \phi_2\right)^2 \\
	&	+  \left(\sigma_2^{-1}+2 R \varepsilon \sigma_1 \sigma_2^{-2} x+\varepsilon^2 \sigma_1^2 \sigma_2^{-2} x^2\right)\left(2 A_0 y+\varepsilon\left[2 A_1 x y-h^2 \sqrt{F_R} y-2 A_0 \partial_y  P_0 B^*\right]+\varepsilon^2 \partial_y \phi_2\right)^2 \\
	 &	-2\left(R+\varepsilon \sigma_1 x\right) \varepsilon \sigma_1^{-1} y \left(2 A_0 x+\varepsilon\left[A_1\left(3 x^2+y^2-1\right)-h^2 \sqrt{F_R} x-2 A_0 \partial_x  P_0 B^*\right]+\varepsilon^2 \partial_x \phi_2\right) \\
		&  \cdot\left(2 A_0 y+\varepsilon\left[2 A_1 x y-\hbar^2 \sqrt{F_R} y-2 A_0 \partial_y  P_0 B^*\right]+\varepsilon^2 \partial_y \phi_2\right) \bigg]_{\rho=1+ \varepsilon B^*}.
	\end{split}
\end{equation*} 
With  $x=(1+\varepsilon B^*(\theta))\cos\theta$ and $y=(1+\varepsilon B^*(\theta))\sin\theta$, this  yields 
\begin{equation}\label{eq:exp of F_0}
\begin{split}
\mathcal F_0&(\varepsilon,B^*)\\
=& 4 A_0^2 \sigma_2^{-1}\left(1+\varepsilon B^*\right)^2\\
&+4 A_0 \sigma_2^{-1} \varepsilon\left[A_1\left(3 x^3+3 x y^2-x\right)-h^2 \sqrt{F_R}\left(1+\varepsilon B^*\right)^2-2 A_0 x \partial_x P_0 B^*-2 A_0 y \partial_y  P_0 B^*\right] \\
& +\sigma_2^{-1} \varepsilon^2\left[A_1\left(3 x^2+y^2-1\right)-h^2 \sqrt{F_R} x-2 A_0 \partial_x P_0 B^*\right]^2+ 4 A_0 \sigma_2^{-1} \varepsilon^2 x \partial_x \phi_2 \\
&+4 A_0^2\varepsilon^2 \sigma_1^{-2} \sigma_2^2  x^2y^2   + \sigma_2^{-1} \varepsilon^2\left[2 A_1 x y-h^2 \sqrt{F_R} y-2 A_0 \partial_y P_0 B^*\right]^2+ 4 A_0 \sigma_2^{-1} \varepsilon^2 y \partial_y \phi_2  \\
&+ 8 R \varepsilon \sigma_1 \sigma_2^{-2} A_0^2 x y^2+ 8 R \varepsilon^2 \sigma_1 \sigma_2^{-2} A_0\left[2 A_1 x^2 y^2-h^2 \sqrt{F_R} x y^2-2 A_0 x y \partial_y  P_0 B^*\right] \\
& +4 A_0^2 \varepsilon^2 \sigma_1^2 \sigma_2^{-2} x^2 y^2 
  -8 R \varepsilon \sigma_1^{-1} A_0^2 x y^2 -4 R \varepsilon^2 \sigma_1^{-1} A_0\left[2 A_1 x^2 y^2-h^2 \sqrt{F_R} x y^2-2 A_0 x y \partial_y  P_0 B^*\right]  \\
& -4 R \varepsilon^2 \sigma_1^{-1} A_0\left[A_1\left(3 x^2 y^2+y^4-y^2\right)-h^2 \sqrt{F_R} x y^2-2 A_0 y^2 \partial_x  P_0 B^*\right]-8A_0^2\varepsilon^2    x^2y^2+O(\varepsilon^3)\\
=:&\sum_{i=1}^{14}N_i+O(\varepsilon^3).
\end{split}
\end{equation}
From the definition of $ B^* $, we can check that
\begin{align*}
\langle N_1,\cos\theta \rangle=&0,\\
\langle N_2,\cos\theta \rangle=&8\pi A_0A_1 \sigma_2^{-1} \varepsilon+32\pi t^* A_0A_1 \sigma_2^{-1} \varepsilon^2+O(\varepsilon^3),\\
\langle N_3,\cos\theta \rangle=&-4 A_1 \sigma_2^{-1} h^2 \sqrt{F_R}\varepsilon^2 \int \cos ^4 \theta  +12C^* A_0 \sigma_2^{-1} h^2 \sqrt{F_R} \varepsilon^2 \int \cos 2 \theta   \cos ^2 \theta  +O(\varepsilon^3)\\
=&-3\pi A_1 \sigma_2^{-1} h^2 \sqrt{F_R}\varepsilon^2  +6\pi C^* A_0 \sigma_2^{-1} h^2 \sqrt{F_R} \varepsilon^2  +O(\varepsilon^3) ,\\
\langle N_4,\cos\theta \rangle=&8 \sigma_2^{-1} A_0  \frac{M_4}{8}   h^2 \sqrt{F_R}  \varepsilon^2  \int \cos ^4 \theta + 4 \sigma_2^{-1} A_0\varepsilon^2\left(12 A_0 t^*  C^*   \int \cos 2 \theta  \cos ^2 \theta  \right . \\
& \left .-2 A_1 t^*\int \cos ^2 \theta \right)+ 12 C^*  \sigma_2^{-1} A_0  h^2 \sqrt{F_R}  \varepsilon^2\int \cos 2 \theta   \cos ^2 \theta  +O(\varepsilon^3)\\
=& \frac{3\pi}{4}\sigma_2^{-1} A_0 M_4   h^2 \sqrt{F_R} \varepsilon^2 + 24 \pi t^*  C^*\sigma_2^{-1} A_0^2  \varepsilon^2  -8\pi t^*\sigma_2^{-1} A_0   A_1  \varepsilon^2\\
& + 6\pi C^*  \sigma_2^{-1} A_0  h^2 \sqrt{F_R} \varepsilon^2 +O(\varepsilon^3), \\
\langle N_5,\cos\theta \rangle=& O(\varepsilon^3),\\
\langle N_6,\cos\theta \rangle=&-4 A_1 \sigma_2^{-1} h^2 \sqrt{F_R}\varepsilon^2 \int \cos ^2 \theta \sin ^2 \theta -24 A_0 \sigma_2^{-1} C^* h^2 \sqrt{F_R}\varepsilon^2 \int \cos ^2 \theta \sin ^2 \theta+O(\varepsilon^3) \\
=&-\pi A_1 \sigma_2^{-1} h^2 \sqrt{F_R} \varepsilon^2  -6\pi C^* A_0 \sigma_2^{-1} h^2 \sqrt{F_R} \varepsilon^2 +O(\varepsilon^3),\\
\langle N_7,\cos\theta \rangle=&8 A_0 \sigma_2^{-1}   \frac{M_4}{8}  h^2 \sqrt{F_R} \varepsilon^2\int \cos ^2 \theta \sin ^2 \theta- 96t^*   C^* A_0^2 \sigma_2^{-1}   \varepsilon^2   \int  \cos^2\theta \sin^2 \theta\\
&-24  C^* A_0 \sigma_2^{-1}h^2 \sqrt{F_R}\varepsilon^2\int \cos^2 \theta \sin ^2 \theta+O(\varepsilon^3)\\
=& \frac{\pi}{4} A_0 \sigma_2^{-1}    M_4   h^2 \sqrt{F_R} \varepsilon^2 - 24\pi t^*   C^* A_0^2 \sigma_2^{-1}   \varepsilon^2    -6\pi  C^* A_0 \sigma_2^{-1}h^2 \sqrt{F_R}\varepsilon^2+O(\varepsilon^3), \\
\langle N_8,\cos\theta \rangle=&2\pi R \sigma_1 \sigma_2^{-2} A_0^2\varepsilon + 6\pi t^*R \sigma_1 \sigma_2^{-2} A_0^2\varepsilon^2+O(\varepsilon^3) ,\\
\langle N_9,\cos\theta \rangle=&-8 R \sigma_1 \sigma_2^{-2} A_0 h^2 \sqrt{F_R} \varepsilon^2\int \cos ^2 \theta \sin ^2 \theta+O(\varepsilon^3)=-2\pi R \sigma_1 \sigma_2^{-2} A_0 h^2 \sqrt{F_R} \varepsilon^2 +O(\varepsilon^3),\\
\langle N_{10},\cos\theta \rangle=& O(\varepsilon^3),\\
\langle N_{11},\cos\theta \rangle=&-2\pi R  \sigma_1^{-1} A_0^2\varepsilon - 6\pi t^*R   \sigma_1^{-1} A_0^2\varepsilon^2  +O(\varepsilon^3),\\
\langle N_{12},\cos\theta \rangle=&4 R \sigma_1^{-1}   A_0 h^2 \sqrt{F_R} \varepsilon^2\int \cos ^2 \theta \sin ^2 \theta+O(\varepsilon^3)=\pi R \sigma_1^{-1}   A_0 h^2 \sqrt{F_R} \varepsilon^2  +O(\varepsilon^3),\\
\langle N_{13},\cos\theta \rangle=&4 R \sigma_1^{-1}   A_0 h^2 \sqrt{F_R} \varepsilon^2\int \cos ^2 \theta \sin ^2 \theta+O(\varepsilon^3)=\pi R \sigma_1^{-1}   A_0 h^2 \sqrt{F_R} \varepsilon^2  +O(\varepsilon^3),\\
\langle N_{14},\cos\theta \rangle=&O(\varepsilon^3) ,
\end{align*} 
where we have used that $ \int \cos^4 \theta \, d\theta=\frac{3\pi}{4} $, $ \int \cos^2 \theta \sin ^2 \theta\, d\theta=\frac{\pi}{4} $, $ \int \cos^2 \theta \cos2\theta\, d\theta =\frac{\pi}{2} $ and $ \int \cos n \theta \cos m\theta\, d\theta=\pi\delta_{n,m} $ for any $ n,m\in\mathbb{Z} $. Here $ \delta_{n,m}=1 $ for $ n=m $ and $ \delta_{n,m}=0 $ otherwise. 

Taking the above formulas into Equation \eqref{eq:exp of F_0} and using the definition of $ A_0,A_1 $ and $ M_4 $,  the  order $ \varepsilon^3  $ expansion  of $ c_{\varepsilon,B^*,1} $ is given by
\begin{align*}
\begin{split}
c_{\varepsilon,B^*,1}=&\left( 8\pi A_0 A_1\sigma_2^{-1} -2\pi R^3 A_0^2\sigma_1^{-1}\sigma_2^{-1}\right)  \varepsilon + \biggl(32\pi t^* A_0A_1 \sigma_2^{-1}-3\pi A_1 \sigma_2^{-1} h^2 \sqrt{F_R}  \\
&+6\pi C^* A_0 \sigma_2^{-1} h^2 \sqrt{F_R}+\frac{3\pi}{4}\sigma_2^{-1} A_0 M_4   h^2 \sqrt{F_R}   + 24 \pi t^*  C^*\sigma_2^{-1} A_0^2    -8\pi t^*\sigma_2^{-1} A_0   A_1  \\
&+ 6\pi C^*  \sigma_2^{-1} A_0  h^2 \sqrt{F_R}-\pi A_1 \sigma_2^{-1} h^2 \sqrt{F_R}   -6\pi C^* A_0 \sigma_2^{-1} h^2 \sqrt{F_R}+\frac{\pi}{4} A_0 \sigma_2^{-1}    M_4   h^2 \sqrt{F_R}\\
&  - 24\pi t^*   C^* A_0^2 \sigma_2^{-1}       -6\pi  C^* A_0 \sigma_2^{-1}h^2 \sqrt{F_R}+ 6\pi t^*R \sigma_1 \sigma_2^{-2} A_0^2-2\pi R \sigma_1 \sigma_2^{-2} A_0 h^2 \sqrt{F_R}\\
&- 6\pi t^*R   \sigma_1^{-1} A_0^2+2\pi R \sigma_1^{-1}   A_0 h^2 \sqrt{F_R}\biggr)\varepsilon^2+O(\varepsilon^3)\\
=&\left( 8\pi A_0 A_1\sigma_2^{-1} -2\pi R^3 A_0^2\sigma_1^{-1}\sigma_2^{-1}\right)  \varepsilon+\biggl( 6\pi t^* A_0\left( 4A_1 \sigma_2^{-1}+    R \sigma_1 \sigma_2^{-2} A_0  - R   \sigma_1^{-1} A_0 \right)    \\
&  +\pi h^2 \sqrt{F_R}\left (-4 A_1 \sigma_2^{-1}+  \sigma_2^{-1} A_0 M_4-2  R \sigma_1 \sigma_2^{-2} A_0+   2  R \sigma_1^{-1}   A_0\right ) \biggr)\varepsilon^2+O(\varepsilon^3)\\
=&\left( 8\pi A_0 A_1\sigma_2^{-1} -2\pi R^3 A_0^2\sigma_1^{-1}\sigma_2^{-1}\right)  \varepsilon+\left( -\frac{\pi }{2}R\sigma_1ch^2\sqrt{F_R}+\frac{15}{2}\pi RA_0\sigma_1ct^*\right) \varepsilon^2+O(\varepsilon^3).
\end{split}
\end{align*}
By comparing the order $ \varepsilon^2 $-term of Equation \eqref{eq:C3-integral}, it is then immediate   that  
\begin{equation*}
\bigl(8A_0A_1\sigma_2^{-1}-2R^3A_0^2\sigma_1^{-1}\sigma_2^{-1}\bigr)\pi t^*+2\pi R\sigma_1C_3
=-\frac{\pi }{2}R\sigma_1ch^2\sqrt{F_R}+\frac{15}{2}\pi RA_0\sigma_1ct^*.
\end{equation*}
Hence
\begin{equation}\label{eq:C3-final}
C_3=\frac{5}{2}A_0ct^*-\frac c4h^2\sqrt{F_R}.
\end{equation}
Recall $ C^*=\frac{R^3\sigma_1^{-1}}8 $ and $ t^*=\frac{5}{18A_0}h^2\sqrt{F_R} $. Substituting \eqref{eq:C3-final} into \eqref{eq:Fboundary-6A3} gives
\begin{equation}\label{eq:Fe-final-before-t}
\begin{split}
\mathcal F& (\varepsilon,B^*)\\
&=\kappa+ \left( 8 A_0^2 \sigma_2^{-1}  (B^*-\Lambda_0B^*)+2R^3A_0^2\sigma_1^{-1}\sigma_2^{-1} \cos3\theta -4A_0\sigma_2^{-1}  h^2\sqrt{F_R} 
+C_3\sigma_2\right) \varepsilon+O(\varepsilon^2)\\
&=\kappa+\left( 8A_0^2\sigma_2^{-1}  t^*-4A_0\sigma_2^{-1}  h^2\sqrt{F_R}
+\left(\frac{5}{2}A_0ct^*-\frac c4h^2\sqrt{F_R}\right)\sigma_2\right) \varepsilon+O(\varepsilon^2)\notag\\
&=\kappa+\varepsilon\left(\frac92A_0\sigma_2ct^*-\frac54\sigma_2ch^2\sqrt{F_R}\right)+O(\varepsilon^2)\\
&=\kappa +O(\varepsilon^2),
\end{split}
\end{equation}
as claimed.

\end{proof}

\begin{remark}\label{rem:third-mode-choice}
	The third Fourier mode appears because the geometric error contains $\cos3\theta$.  The term $B-\Lambda_0B$ multiplies the mode $\cos n\theta$ by $1-n$.  For $n=3$ this multiplier is $-2$, so a single $\cos3\theta$ correction cancels the geometric third mode.
\end{remark}
 
 \section{Computing the variations with respect to the domain}\label{sec:variation}
Our next objective is to compute how $\phi$ changes as we change the domain by perturbing the function $B$.
More precisely, we aim to compute the derivative of $\phi$ with respect to $B$, which we will denote as
\[
\Phi_{\varepsilon,B,\mathbf{B}} :=
\frac{d}{d t}\bigg|_{t=0} \phi_{\varepsilon,B+t\mathbf{B}} ,
\]
where $\mathbf{B}(\theta)$ is a function defined on $\mathbb{T}$.
In this section, we are mainly interested in the derivative at $B=B^*$,   $\Phi_{\varepsilon,B^*,\mathbf{B}}$.

Let us first list some Fourier--Poisson calculus which is frequently used in this section. Here and in what follows, $\operatorname{sgn}(n) := n/|n|$ is the sign of the nonzero integer $n$.

\begin{lemma} \label{lem:poisson-calculus}
Let $\textbf{B}(\theta)=\sum_{n\in\mathbb Z}B_ne^{in\theta}\in C^{s+1}(\mathbb T)$.  The Cartesian derivatives of $ P_0\textbf{B} $ up to second order are given by 
\begin{equation}\label{eq:Poisson-derivatives1}
\begin{split}
\partial_xP_0\textbf{B}&=\sum_{n\in\mathbb Z\setminus\{0\}}|n|B_n\rho^{|n|-1}e^{i(n-\text{sgn}(n))\theta},\\
\partial_yP_0\textbf{B}&=\sum_{n\in\mathbb Z\setminus\{0\}}\text{sgn}(n)i|n|B_n\rho^{|n|-1}e^{i(n-\text{sgn}(n))\theta},\\
\partial_{xx}P_0\textbf{B}&=\sum_{n\in\mathbb Z\setminus\{0,\pm1\}}|n|(|n|-1)B_n\rho^{|n|-2}e^{i(n-2\text{sgn}(n))\theta},\\
\partial_{xy}P_0\textbf{B}&=\sum_{n\in\mathbb Z\setminus\{0,\pm1\}}\text{sgn}(n)i|n|(|n|-1)B_n\rho^{|n|-2}e^{i(n-2\text{sgn}(n))\theta},\\
\partial_{yy}P_0\textbf{B}&=-\sum_{n\in\mathbb Z\setminus\{0,\pm1\}}|n|(|n|-1)B_n\rho^{|n|-2}e^{i(n-2\text{sgn}(n))\theta}.
\end{split}
\end{equation}

\end{lemma}
\begin{proof}
For $ \textbf{B}(\theta)=\sum_{n\in\mathbb Z}B_ne^{in\theta}\in C^{s+1}(\mathbb T), $
the Poisson extension in the unit disk is $ P_0\textbf{B}(\rho,\theta)=\sum_{n\in\mathbb Z}B_n\rho^{|n|}e^{in\theta}. $
In polar coordinates $x=\rho\cos\theta$, $y=\rho\sin\theta$,
\begin{equation*}\label{eq:px-polar}
\begin{split}
\partial_x&=\cos\theta\partial_\rho-\rho^{-1}\sin\theta\partial_\theta
=\frac12e^{i\theta}\left(\partial_\rho+\frac i\rho\partial_\theta\right)
+\frac12e^{-i\theta}\left(\partial_\rho-\frac i\rho\partial_\theta\right),\\
\partial_y&=\sin\theta\partial_\rho+\rho^{-1}\cos\theta\partial_\theta
=\frac1{2i}e^{i\theta}\left(\partial_\rho+\frac i\rho\partial_\theta\right)
-\frac1{2i}e^{-i\theta}\left(\partial_\rho-\frac i\rho\partial_\theta\right). 
\end{split}
\end{equation*}
Also
\begin{equation*}\label{eq:x-y-Fourier}
x=\frac\rho2(e^{i\theta}+e^{-i\theta}),
\qquad y=\frac\rho{2i}(e^{i\theta}-e^{-i\theta}).
\end{equation*}
For $ \partial_xP_0\textbf{B} $, the mode-by-mode derivation is the following.  If $n>0$,
\begin{align*}\label{eq:px-mode-positive-full}
\partial_x(B_n\rho^ne^{in\theta})
&=\frac{e^{i\theta}+e^{-i\theta}}{2}nB_n\rho^{n-1}e^{in\theta}
-\rho^{-1}\frac{e^{i\theta}-e^{-i\theta}}{2i}inB_n\rho^ne^{in\theta}\notag\\
&=nB_n\rho^{n-1}e^{i(n-1)\theta},
\end{align*}
and if $n<0$,
\begin{equation*}\label{eq:px-mode-negative-full}
\partial_x(B_n\rho^{|n|}e^{in\theta})=|n|B_n\rho^{|n|-1}e^{i(n+1)\theta}.
\end{equation*}
So the first formula of \eqref{eq:Poisson-derivatives1} holds term by term for trigonometric polynomials and can be extended to $C^{s+1}$ data by density and  Schauder estimates. Similarly, for $ \partial_yP_0\textbf{B} $ we have
\begin{equation*}\label{eq:py-mode-positive-full}
\begin{split}
\partial_y(B_n\rho^ne^{in\theta})&=\frac{e^{i\theta}-e^{-i\theta}}{2i}nB_n\rho^{n-1}e^{in\theta}
+\rho^{-1}\frac{e^{i\theta}+e^{-i\theta}}{2}inB_n\rho^ne^{in\theta} \\
&=inB_n\rho^{n-1}e^{i(n-1)\theta},\qquad n>0,\\
\partial_y(B_n\rho^{|n|}e^{in\theta})&=-i|n|B_n\rho^{|n|-1}e^{i(n+1)\theta},\qquad n<0. 
\end{split}
\end{equation*}
So the second formula of \eqref{eq:Poisson-derivatives1} holds. For $ \partial_{xx}P_0\textbf{B},\partial_{xy}P_0\textbf{B}$ and $ \partial_{yy}P_0\textbf{B} $, applying $\partial_x$ or $\partial_y$ once more gives the result.  

\end{proof}
In the statement of the next proposition, we will need the operator
 $\mathcal{T} : C^{s+1}(\mathbb{T}) \to C^{s}(\mathbb{D})$ defined as
\begin{align*}
\mathcal{T}f(\rho,\theta):=&\sum_{|n|\ge3}\frac{(|n|-1)(I_0-I_2+I_1)+2I_3}{4}f_n
(\rho^{|n|-1}-\rho^{|n|+1})e^{i(n-\text{sgn} (n))\theta}\notag\\
&+\sum_{|n|\ge3}\frac{|n|(I_0-I_2-I_1)}{8}f_n
(\rho^{|n|-3}-\rho^{|n|+1})e^{i(n-3\text{sgn} (n))\theta}\notag\\
&+\frac{I_0-I_2+I_1+2I_3}{4}(f_2e^{i\theta}+f_{-2}e^{-i\theta})(\rho-\rho^3)\notag\\
&+\frac{I_0-I_2-I_1}{4}(f_2e^{-i\theta}+f_{-2}e^{i\theta})(\rho-\rho^3)
+\frac12I_3(f_1+f_{-1})(1-\rho^2),\notag
\end{align*}
 for $f(\theta) = \sum_{n\in\mathbb{Z}} f_n e^{in\theta}$. Here $ I_0,I_1,I_2 $ and $ I_3 $ are constants defined as 
\begin{equation*} 
I_0:=-2R\sigma_1^{-1}h^2,
\quad I_1:=2R\sigma_1^{-1}\sigma_2,
\quad I_2:=-4R\sigma_1\sigma_2^{-1},
\quad I_3:=(R^2+3h^2)R\sigma_1^{-1}.
\end{equation*}
 
\begin{proposition}\label{prop:variation}
	For $ \textbf{B} \in C^{s+1}(\mathbb{T})$ and any small enough $\varepsilon$,
\begin{equation*} 
\begin{split}
\Phi_{\varepsilon,B^*,\mathbf{B}} =& -2\varepsilon A_0 P_{\varepsilon B^*} \textbf{B}+ \varepsilon^2 \Bigl[ A_0\mathcal{T} \textbf{B}
-2A_1P_{\varepsilon B^*}(\cos\theta\, \textbf{B}) +h^2\sqrt{F_R}P_{\varepsilon B^*}  \textbf{B}\\
&\hspace{8em}
 +4A_0C^*P_{\varepsilon B^*}(\cos3\theta\, \textbf{B})-2A_0t^*P_{\varepsilon B^*} \textbf{B} \Bigr]
+ O(\varepsilon^3) .
\end{split}
\end{equation*}
Moreover, its first-order derivative   satisfies
\begin{align*} 
\partial_x\Phi_{\varepsilon,B^*, \textbf{B}}
=&-2\varepsilon A_0\partial_xP_{\varepsilon B^*}  \textbf{B}
+\varepsilon^2\Big[A_0\partial_x\mathcal{T} \textbf{B}-2A_1\partial_xP_{\varepsilon B^*}(  \cos\theta\, \textbf{B}) +h^2\sqrt{F_R}\partial_xP_{\varepsilon B^*}  \textbf{B}\\
&\hspace{9em}
+4A_0C^*\partial_xP_{\varepsilon B^*}(\cos3\theta\, \textbf{B})-2A_0t^*\partial_xP_{\varepsilon B^*} \textbf{B} \Bigr]
+ O(\varepsilon^3),\notag\\
\partial_y\Phi_{\varepsilon,B^*, \textbf{B}}
=&-2\varepsilon A_0\partial_yP_{\varepsilon B^*}  \textbf{B}
+\varepsilon^2\Big[A_0\partial_y\mathcal{T} \textbf{B}-2A_1\partial_yP_{\varepsilon B^*}(  \cos\theta\, \textbf{B}) +h^2\sqrt{F_R}\partial_yP_{\varepsilon B^*}  \textbf{B}\\
&\hspace{9em}
+4A_0C^*\partial_yP_{\varepsilon B^*}(\cos3\theta\, \textbf{B})-2A_0t^*\partial_yP_{\varepsilon B^*} \textbf{B} \Bigr]
+ O(\varepsilon^3).\notag 
\end{align*}

\end{proposition} 
\begin{proof}

Differentiating Equation \eqref{eq:phi-0}  with respect to $ B $ at $ B = B^* $, we obtain that
$ \Phi=\Phi_{\varepsilon,B^*,\mathbf{B}} $ satisfies the equation 
 \begin{align}\label{eq:Phi}
 &\bigl((R+\varepsilon\sigma_1x)^2+(\varepsilon\sigma_2y)^2+h^2\bigr)\bigl(h^2+(\varepsilon\sigma_2y)^2\bigr)\sigma_1^{-2}\partial_{xx}\Phi \notag\\
 &  -2\bigl((R+\varepsilon\sigma_1x)^2+(\varepsilon\sigma_2y)^2+h^2\bigr)(R+\varepsilon\sigma_1x)(\varepsilon\sigma_2y)
 \sigma_1^{-1}\sigma_2^{-1}\partial_{xy}\Phi \notag\\
 & +\bigl((R+\varepsilon\sigma_1x)^2+(\varepsilon\sigma_2y)^2+h^2\bigr)\bigl(h^2+(R+\varepsilon\sigma_1x)^2\bigr)\sigma_2^{-2}\partial_{yy}\Phi \notag\\
 =&\varepsilon\bigl((R+\varepsilon\sigma_1x)^2+(\varepsilon\sigma_2y)^2+3h^2\bigr)\bigl((R+\varepsilon\sigma_1x)\sigma_1^{-1}\partial_x\Phi+\varepsilon y\partial_y\Phi\bigr)\notag\\
 &
 - 2\varepsilon^2h^2
 F'(\varepsilon^2\phi_{\varepsilon,B^*})\Phi-\varepsilon^2\bigl((R+\varepsilon\sigma_1x)^2+(\varepsilon\sigma_2y)^2+h^2\bigr)
 \left(\frac{F^2}{2}\right)''(\varepsilon^2\phi_{\varepsilon,B^*})\Phi \notag\\
 & +\varepsilon^2\bigl((R+\varepsilon\sigma_1x)^2+(\varepsilon\sigma_2y)^2+h^2\bigr)^2
 H''(\varepsilon^2\phi_{\varepsilon,B^*})\Phi.
 \end{align}
Using the fact that
\begin{equation*}\label{eq:low-regularity-FH}
F'(\varepsilon^2\phi_{\varepsilon,B^*})=O(\varepsilon),
\qquad \left(\frac{F^2}{2}\right)''(\varepsilon^2\phi_{\varepsilon,B^*})=O(1),
\qquad H''(\varepsilon^2\phi_{\varepsilon,B^*})=O(1),
\end{equation*}
Equation \eqref{eq:Phi} can be rewritten as 
\begin{equation*}\label{eq:Phi-expanded-coefficients}
\begin{split}
&\bigl(1+2R\varepsilon\sigma_1^{-1}xh^2+O(\varepsilon^2)\bigr)\partial_{xx}\Phi
-2\bigl(\sigma_2R\varepsilon\sigma_1^{-1}y+O(\varepsilon^2)\bigr)\partial_{xy}\Phi
+\bigl(1+4R\varepsilon\sigma_1\sigma_2^{-1}x+O(\varepsilon^2)\bigr)\partial_{yy}\Phi\\
&\qquad=\varepsilon\bigl(R^2+3h^2+O(\varepsilon)\bigr)\left( R\sigma_1^{-1}\partial_x\Phi
+O(\varepsilon)(\partial_x\Phi+\partial_y\Phi)\right) +O(\varepsilon^2)\Phi,
\end{split}
\end{equation*}
which implies that
\begin{align}\label{eq:Phi-expanded}
\Delta\Phi={}&\varepsilon\Big[-2R\sigma_1^{-1}h^2x\partial_{xx}\Phi
+2 R\sigma_1^{-1}\sigma_2y\partial_{xy}\Phi
-4R\sigma_1\sigma_2^{-1}x\partial_{yy}\Phi +(R^2+3h^2)R\sigma_1^{-1}\partial_x\Phi\Big] \notag\\
&
+O(\varepsilon^2)\left (\Phi+\partial_{x}\Phi+\partial_{y}\Phi+\partial_{xx}\Phi+\partial_{xy}\Phi+\partial_{yy}\Phi\right )
\end{align} 
in $ \varOmega_{\varepsilon,B^*} $. 	Likewise, differentiating the boundary condition we obtain that 
\begin{equation*} 
 \Phi(1+\varepsilon B^*(\theta),\theta)
=-\varepsilon\partial_\rho\phi_{\varepsilon,B^*}(1+\varepsilon B^*(\theta),\theta)\mathbf{B}(\theta) .
\end{equation*}
In view of the asymptotics for $ \phi_{\varepsilon,B^*}$ computed in Proposition~\ref{prop:first-order-dirichlet-expansion} and formula \eqref{eq:prho-phi}, this boundary condition can be written as 
 \begin{align*}
 &\Phi(1+\varepsilon B^*(\theta),\theta)\notag\\
 &\quad=-\varepsilon \textbf{B}(\theta)\bigl(2A_0+2\varepsilon A_1\cos\theta-\varepsilon h^2\sqrt{F_R}-6C^*A_0\varepsilon \cos3\theta+2A_0\varepsilon B^*(\theta)+O(\varepsilon^2)\bigr)\notag\\
 &\quad=-2\varepsilon A_0  \textbf{B}-2\varepsilon^2A_1\cos\theta\, \textbf{B}
 +\varepsilon^2h^2\sqrt{F_R} \textbf{B}+4\varepsilon^2C^*A_0\cos3\theta\, \textbf{B}-2\varepsilon^2A_0 t^*\textbf{B}+O(\varepsilon^3),
 \end{align*}
and  $ \Phi $ has the expression
 \begin{equation*}\label{eq:Phi-leading}
 \Phi=-2\varepsilon A_0P_{\varepsilon B^*} \textbf{B}+O(\varepsilon^2).
 \end{equation*}
 Equation~\eqref{eq:Phi-expanded} is then of the form
\begin{align*} 
\Delta\Phi&=\varepsilon\left [-2R\sigma_1^{-1}h^2x\partial_{xx}\Phi
+2 R\sigma_1^{-1}\sigma_2y\partial_{xy}\Phi
-4R\sigma_1\sigma_2^{-1}x\partial_{yy}\Phi +(R^2+3h^2)R\sigma_1^{-1}\partial_x\Phi\right ]  
+O(\varepsilon^3)\\
&=\varepsilon\left [I_0x\partial_{xx}\Phi
+I_1y\partial_{xy}\Phi
+I_2x\partial_{yy}\Phi +I_3\partial_x\Phi\right ]  
+O(\varepsilon^3).
\end{align*} 
Let us set  $ \Phi_1:=(\Phi-\varepsilon\Phi_0)/ \varepsilon^2 $, with $ \Phi_0:=-2A_0P_{\varepsilon B^*} \textbf{B}. $ A short calculation shows that $\Phi_1$ must solve the equation
 \begin{equation}\label{eq:Phi1-equation}
 \begin{split}
 \Delta\Phi_1=&I_0x\partial_{xx}\Phi_0+I_1y\partial_{xy}\Phi_0+I_2x\partial_{yy}\Phi_0+I_3\partial_x\Phi_0+O(\varepsilon)\\
=&-2A_0I_0x\partial_{xx}P_0\textbf{B}-2A_0I_1y\partial_{xy}P_0\textbf{B}-2A_0I_2x\partial_{yy}P_0\textbf{B}-2A_0I_3\partial_xP_0\textbf{B}+O(\varepsilon) 
 \end{split}
 \end{equation}
with the boundary condition  
 \begin{equation*} 
 \Phi_1(1+\varepsilon B^*,\theta)
 =-2A_1 \cos\theta\, \textbf{B}+h^2\sqrt{F_R} \textbf{B}+4 A_0C^*\cos3\theta\, \textbf{B}-2 A_0 t^*\textbf{B}+O(\varepsilon).
 \end{equation*}
From Lemma \ref{lem:poisson-calculus}, the Fourier   terms in  Equation \eqref{eq:Phi1-equation} are obtained from
 \begin{align*} 
 x\partial_{xx}P_0\textbf{B}
 &=\sum_{n\in\mathbb{Z}\setminus\{0,\pm1\}}|n|(|n|-1)B_n\rho^{|n|-1}
 \frac{e^{i(n+1-2\text{sgn}(n))\theta}+e^{i(n-1-2\text{sgn}(n))\theta}}{2},
\notag\\
 y\partial_{xy}P_0\textbf{B}
 &=\sum_{n\in\mathbb{Z}\setminus\{0,\pm1\}}\text{sgn}(n)|n|(|n|-1)B_n\rho^{|n|-1}
 \frac{e^{i(n+1-2\text{sgn}(n))\theta}-e^{i(n-1-2\text{sgn}(n))\theta}}{2},
\notag\\
 x\partial_{yy}P_0\textbf{B}
 &=-\sum_{n\in\mathbb{Z}\setminus\{0,\pm1\}}|n|(|n|-1)B_n\rho^{|n|-1}
 \frac{e^{i(n+1-2\text{sgn}(n))\theta}+e^{i(n-1-2\text{sgn}(n))\theta}}{2},
\notag\\
 \partial_xP_0\textbf{B}
 &=\sum_{n\in\mathbb{Z}\setminus\{0\}}|n|B_n\rho^{|n|-1}e^{i(n-\text{sgn}(n))\theta},
\notag
 \end{align*}
for $ \textbf{B}=\sum_{n\in\mathbb Z}B_ne^{in\theta} $.  Taking these into \eqref{eq:Phi1-equation}, we get
\begin{equation*} 
 \Delta\Phi_1=D_0+O(\varepsilon), 
\end{equation*}
where 
\begin{equation*} 
\begin{split}
D_0:=&-2A_0(I_0-I_2)\sum_{n\in\mathbb{Z}\setminus\{0,\pm1\}}|n|(|n|-1)B_n\rho^{|n|-1}
\frac{e^{i(n+1-2\text{sgn}(n))\theta}+e^{i(n-1-2\text{sgn}(n))\theta}}{2}\\
&-2A_0I_1\sum_{n\in\mathbb{Z}\setminus\{0,\pm1\}}\text{sgn}(n)|n|(|n|-1)B_n\rho^{|n|-1}
\frac{e^{i(n+1-2\text{sgn}(n))\theta}-e^{i(n-1-2\text{sgn}(n))\theta}}{2}\\
&-2A_0I_3\sum_{n\in\mathbb{Z}\setminus\{0\}}|n|B_n\rho^{|n|-1}e^{i(n-\text{sgn}(n))\theta}.
\end{split}
\end{equation*}
Let us note that  $ \Delta \frac{B_n}{(|n|+2)^2-m^2}\rho^{|n|+2}e^{im\theta}=B_n\rho^{|n|}e^{im\theta}  $, provided $(|n|+2)^2\ne m^2$. 
If we set  
 \begin{align}\label{eq:Phi2}
 \Phi_2:=&\sum_{|n|\ge3}\frac{(|n|-1)(I_0-I_2+I_1)+2I_3}{4}A_0B_n
 (\rho^{|n|-1}-\rho^{|n|+1})e^{i(n-\text{sgn} (n))\theta}\notag\\
 &+\sum_{|n|\ge3}\frac{|n|(I_0-I_2-I_1)}{8}A_0B_n
 (\rho^{|n|-3}-\rho^{|n|+1})e^{i(n-3\text{sgn} (n))\theta}\notag\\
 &+\frac{I_0-I_2+I_1+2I_3}{4}A_0(B_2e^{i\theta}+B_{-2}e^{-i\theta})(\rho-\rho^3)\notag\\
 &+\frac{I_0-I_2-I_1}{4}A_0(B_2e^{-i\theta}+B_{-2}e^{i\theta})(\rho-\rho^3)
 +\frac12I_3A_0(B_1+B_{-1})(1-\rho^2), 
 \end{align}
then direct computation shows that $  \Delta\Phi_2=D_0. $ 
 The boundary value of $\Phi_2$ is obtained as follows. From the Taylor formula $ \rho^m\big|_{\rho=1+\varepsilon B^*}=1+m\varepsilon B^*+O(\varepsilon^2) $, we have 
 \begin{equation*} 
 (\rho^{m_1}-\rho^{m_2})\big|_{\rho=1+\varepsilon B^*}
 =(m_1-m_2)\varepsilon B^*+O(\varepsilon^2).
 \end{equation*}
 Applying this formula to each term in Equation \eqref{eq:Phi2}, we get
 \begin{align*}\label{eq:Phi2-boundary-expanded}
 \Phi_2(1+\varepsilon B^*,\theta)={}&\varepsilon A_0B^*\sum_{|n|\ge3}
 \frac{(|n|-1)(I_0-I_2+I_1)+2I_3}{2}B_n e^{i(n-\text{sgn} (n))\theta}\notag\\
 &+\varepsilon A_0B^*\sum_{|n|\ge3}
 \frac{|n|(I_0-I_2-I_1)}{2}B_n e^{i(n-3\text{sgn} (n))\theta}\notag\\
 &+\varepsilon A_0B^*\frac{I_0-I_2+I_1+2I_3}{2}(B_2e^{i\theta}+B_{-2}e^{-i\theta}) \notag\\
 &+\varepsilon A_0B^*\frac{I_0-I_2-I_1}{2}(B_2e^{-i\theta}+B_{-2}e^{i\theta})
 +\varepsilon I_3A_0B^*(B_1+B_{-1})+O(\varepsilon^2)\\
 =&O(\varepsilon),
 \end{align*}
where we have used   $ ||\textbf{B}||_{C^s(\mathbb{T})}\leq 1 $ for $ s>2 $ and Schauder estimates.
Consequently, the function $ \Phi_3:=\Phi_1-\Phi_2 $ satisfies the equation
\begin{equation*} 
\Delta\Phi_3=O(\varepsilon)\quad\hbox{in }\varOmega_{\varepsilon,B^*}
\end{equation*} 
and the boundary condition
 \begin{align*} 
 \Phi_3(1+\varepsilon B^*,\theta)=&=-2A_1 \cos\theta\, \textbf{B}+h^2\sqrt{F_R} \textbf{B}+4 A_0C^*\cos3\theta\, \textbf{B}-2 A_0 t^*\textbf{B}+O(\varepsilon).
 \end{align*}
 This shows that
 \begin{align*}
 \Phi_3=-2A_1P_{\varepsilon B^*}(\cos\theta\, \textbf{B})
 +h^2\sqrt{F_R}P_{\varepsilon B^*}  \textbf{B}
 +4A_0C^*P_{\varepsilon B^*}(\cos3\theta\, \textbf{B})-2A_0t^*P_{\varepsilon B^*} \textbf{B}+O(\varepsilon),
 \end{align*}
and therefore
\begin{equation*} 
\begin{split}
\Phi_1=&\Phi_2+\Phi_3\\
=&-2A_1P_{\varepsilon B^*}(\cos\theta\, \textbf{B})
+h^2\sqrt{F_R}P_{\varepsilon B^*}  \textbf{B}
+4A_0C^*P_{\varepsilon B^*}(\cos3\theta\, \textbf{B})-2A_0t^*P_{\varepsilon B^*} \textbf{B}\\
&+\sum_{|n|\ge3}\frac{(|n|-1)(I_0-I_2+I_1)+2I_3}{4}A_0B_n
(\rho^{|n|-1}-\rho^{|n|+1})e^{i(n-\text{sgn} (n))\theta}\\
&+\sum_{|n|\ge3}\frac{|n|(I_0-I_2-I_1)}{8}A_0B_n
(\rho^{|n|-3}-\rho^{|n|+1})e^{i(n-3\text{sgn} (n))\theta}\\
&+\frac{I_0-I_2+I_1+2I_3}{4}A_0(B_2e^{i\theta}+B_{-2}e^{-i\theta})(\rho-\rho^3)\\
&+\frac{I_0-I_2-I_1}{4}A_0(B_2e^{-i\theta}+B_{-2}e^{i\theta})(\rho-\rho^3)
+\frac12I_3A_0(B_1+B_{-1})(1-\rho^2)+O(\varepsilon), 
\end{split}
\end{equation*} 
and $ \Phi=\varepsilon\Phi_0+\varepsilon^2\Phi_1 $. The derivative formulas are obtained by differentiating the expansion, as claimed.
\end{proof} 

As a consequence of Proposition~\ref{prop:variation}, we record that
\begin{align*}
\partial_\rho\Phi_{\varepsilon,B^*, \textbf{B}}|_{\rho=1+\varepsilon B^*}
=&-2\varepsilon A_0\Lambda_{\varepsilon,B^*} \textbf{B}
+\varepsilon^2\Big[-A_0\mathcal{T}' \textbf{B}-2A_1\Lambda_{\varepsilon,B^*}(\cos\theta\,  \textbf{B})
+h^2\sqrt{F_R}\Lambda_{\varepsilon,B^*} \textbf{B} \notag\\
&\hspace{8em}+4A_0C^*\Lambda_{\varepsilon,B^*}(\cos3\theta\, \textbf{B} )-2A_0t^*\Lambda_{\varepsilon,B^*} \textbf{B}\Big]
+O(\varepsilon^3),
\end{align*}
where $\mathcal{T}' : C^{s+1}(\mathbb{T}) \to C^{s}(\mathbb{T})$ is the operator defined as
\begin{align*}
\mathcal{T}' f:=&\sum_{|n|\ge3}\frac{(|n|-1)(I_0-I_2+I_1)+2I_3}{2}f_ne^{i(n-\text{sgn} (n))\theta} +\sum_{|n|\ge3}\frac{|n|(I_0-I_2-I_1)}{2}f_ne^{i(n-3\text{sgn} (n))\theta}\notag\\
&+\frac{I_0-I_2+I_1+2I_3}{2}(f_2e^{i\theta}+f_{-2}e^{-i\theta}) +\frac{I_0-I_2-I_1}{2}(f_2e^{-i\theta}+f_{-2}e^{i\theta})+I_3f_1+I_3f_{-1},
\end{align*}
for   $f(\theta) = \sum_{n\in\mathbb{Z}} f_n e^{in\theta}$. Notice that $ \mathcal{T}' f(\theta)=-\partial_\rho (\mathcal{T} f)(1,\theta). $

%


\section{Solvability of an overdetermined problem} \label{sec: invertibility}

It stems from Proposition~\ref{prop:F_expansion} that the function
\begin{equation}\label{eq:G_def}
\mathcal{G}(\varepsilon,B) := \frac1\varepsilon \bigl[ \mathcal{F}(\varepsilon,B) - \kappa \bigr]
\end{equation}
can be defined at $\varepsilon=0$ by continuity, so that $\mathcal{G}(0,B^*)=0$, resulting in a map defined for all $|\varepsilon|<\varepsilon_0$, where $\varepsilon_0$ is some positive constant.
Let us now define, for each non-integer $s>2$, the space
\[
X^s := \bigl\{ f \in C^s(\mathbb{T}) : f(\theta)=f(-\theta),\;
\langle f,\cos\theta\rangle = 0 \bigr\} .
\]
Near $B^*$ we define its ball of radius $1$
\begin{equation*}
X'_{s}(B^*)=\{B\in X^{s}:||{B-B^*}||_{C^{s}}\le 1\}.
\end{equation*}
As $\langle\mathcal{F}(\varepsilon,B),\cos\theta\rangle = 0$ by the definition of $c_{\varepsilon,B}$, and $\phi_{\varepsilon,B}$ is an even function if $B$ is (cf.\ Proposition~\ref{prop:Dirichlet}), our previous results then immediately imply
\begin{proposition}\label{prop:G}
	Given any $R>0$, there is some $\varepsilon_0>0$ such that the formula~\eqref{eq:G_def} defines a map
	\[
	\mathcal{G} : (-\varepsilon_0,\varepsilon_0) \times X'_{s+1}(B^*) \to X^s .
	\]
\end{proposition} 

The core of this section is  to show that the Fr\'echet derivative  
\begin{equation*}
D_B\mathcal{G}(0,B^*):X^{s+1}\to X^s
\end{equation*}
is invertible.  Since  $ \textbf{B}  $ is even, thoughout this section, we always denote the Fourier expansion of $ \textbf{B} $ as $$ \textbf{B}(\theta) =  B_0+2\sum_{n\geq 1} B_n \cos n\theta. $$ 

\subsection{Estimates of $  \nabla\phi_{\varepsilon,B^*}\big|_{\rho=1+\varepsilon B^*} $  and  $\nabla\Phi_{\varepsilon, B^*,\textbf{B}}\big|_{\rho=1+\varepsilon B^*} $}

Let us first give some asymptotics of the first-order derivative of $  \phi_{\varepsilon,B^*} $  and  $\Phi_{\varepsilon, B^*,\textbf{B}} $ on the boundary $ \rho=1+\varepsilon B^*. $

Recall that $  \partial_{x}\phi_{\varepsilon,B^*}\big|_{\rho=1+\varepsilon B^*} $ and $  \partial_{y}\phi_{\varepsilon,B^*}\big|_{\rho=1+\varepsilon B^*} $ satisfy
\begin{equation}\label{eq:pxpy-phi}
\begin{split}
\partial_x\phi_{\varepsilon,B^*}\big|_{\rho=1+\varepsilon B^*}&=2A_0 \cos\theta\\
&\quad+\varepsilon\left[2A_0  B^*(\theta)\cos\theta+2A_1\cos^2\theta-h^2\sqrt{F_R}\cos\theta-2A_0\partial_x(P_{0}B^*)|_{\rho=1 }\right]+O(\varepsilon^2),\\
\partial_y\phi_{\varepsilon,B^*}\big|_{\rho=1+\varepsilon B^*}&=2A_0 \sin\theta\\
&\quad+\varepsilon\left[2A_0  B^*(\theta)\sin\theta+A_1\sin2\theta-h^2\sqrt{F_R}\sin\theta-2A_0\partial_y(P_{0}B^*)|_{\rho=1}\right]+O(\varepsilon^2).
\end{split}
\end{equation}
For asymptotics of $\partial_{x}\Phi_{\varepsilon, B^*,\textbf{B}}\big|_{\rho=1+\varepsilon B^*} $ and $\partial_{y}\Phi_{\varepsilon, B^*,\textbf{B}}\big|_{\rho=1+\varepsilon B^*} $, let us   define three operators $ \mathcal{T}_1:X^{s+1}\to C^{s}(\mathbb{D}), \mathcal{T}_2:X^{s+1}\to X^{s-1}  $ and $ \mathcal{T}_3:X^{s+1}\to X^{s-1} $ as 
\begin{align*}
\mathcal{T}_1 \textbf{B}
&:=2t^*\sum_{n\ge1}B_nn\rho^{n}\cos n\theta
+ C^* \sum_{n\ge1}B_nn\left(\rho^{|n+3|}\cos(n+3)\theta+\rho^{|n-3|}\cos(n-3)\theta\right),\\
\mathcal{T}_2 \textbf{B}
&:=2B^*(\theta)\cos\theta\sum_{n\ge1}B_nn(n-1) \cos n\theta +2B^*(\theta) \sin\theta\sum_{n\ge1}B_nn(n-1) \sin n\theta,\\
\mathcal{T}_3 \textbf{B}
&:= 
2B^*(\theta)\sin\theta\sum_{n\ge1}B_nn(n-1)  \cos n\theta  -2B^*(\theta)\cos\theta\sum_{n\ge1}B_nn(n-1)  \sin n\theta.  
\end{align*}

\begin{lemma}\label{lemma:expan of px-PB}
	For $ \textbf{B} \in X^{s+1} $ and any small enough $\varepsilon$,	
\begin{align*}
\partial_xP_{\varepsilon B^*}  \textbf{B}\big|_{\rho=1+\varepsilon B^*}
&=\partial_x(P_0  \textbf{B})\big|_{\rho=1}
+\varepsilon\left(\mathcal{T}_2  \textbf{B}-\partial_x  \mathcal{T}_1  \textbf{B}\big|_{\rho=1}\right)+O(\varepsilon^2), \\
\partial_yP_{\varepsilon B^*}  \textbf{B}\big|_{\rho=1+\varepsilon B^*}
&=\partial_y(P_0  \textbf{B})\big|_{\rho=1}
+\varepsilon\left(\mathcal{T}_3  \textbf{B}-\partial_y  \mathcal{T}_1  \textbf{B}\big|_{\rho=1}\right)+O(\varepsilon^2),
\end{align*}

\end{lemma}
\begin{proof}

Note that $ P_0  \textbf{B}=B_0+ \sum_{|n|\ge1}B_n\rho^{|n|}\cos n\theta. $  
Set $\varphi_1:=P_{\varepsilon B^*}  \textbf{B}-P_0  \textbf{B}$. Then $ \varphi_1 $ satisfies $ \Delta\varphi_1 =0    $ with boundary condition
\begin{equation*}
\begin{split}
\varphi_1(1+\varepsilon B^*)
&=  \textbf{B}-\left(B_0+\sum_{|n|\ge1}B_n(1+\varepsilon B^*)^{|n|}\cos n\theta\right)  =-\varepsilon B^*(\theta)\sum_{|n|\ge1}B_n|n|\cos n\theta+O(\varepsilon^2).
\end{split}
\end{equation*}
Hence
\begin{align*}
\varphi_1
&=-\varepsilon P_{\varepsilon B^*}\left( B^*(\theta)\sum_{|n|\ge1}B_n|n|\cos n\theta\right)+O(\varepsilon^2) \\
&=-\varepsilon  P_{0}\left( B^*(\theta)\sum_{|n|\ge1}B_n|n|\cos n\theta \right)  +O(\varepsilon^2) \\
&=-\varepsilon t^*\sum_{|n|\ge1}B_n|n|\rho^{|n|}\cos n\theta   -\frac{\varepsilon C^*}{2}\sum_{|n|\ge1}B_n|n|\left(\rho^{|n+3|}\cos(n+3)\theta+\rho^{|n-3|}\cos(n-3)\theta\right)+O(\varepsilon^2),
\end{align*}
which implies that  $ P_{\varepsilon B^*}  \textbf{B}=P_0  \textbf{B}-\varepsilon \mathcal{T}_1  \textbf{B}+O(\varepsilon^2)
 $  in $ \varOmega_{\varepsilon,B^*}. $
Differentiating this expansion yields that
\begin{align*}
\partial_xP_{\varepsilon B^*}  \textbf{B}
&=\partial_xP_0  \textbf{B}-\varepsilon\partial_x \mathcal{T}_1  \textbf{B}+O(\varepsilon^2), \quad
\partial_yP_{\varepsilon B^*}  \textbf{B}
=\partial_yP_0  \textbf{B}-\varepsilon\partial_y \mathcal{T}_1  \textbf{B}+O(\varepsilon^2).
\end{align*}
So on the boundary $ \rho=1+\varepsilon B^* $,  we have
\begin{equation}\label{eq: px-py-expan}
\begin{split}
\partial_xP_{\varepsilon B^*}  \textbf{B}\big|_{\rho=1+\varepsilon B^*}
&=\partial_xP_0  \textbf{B}\big|_{\rho=1+\varepsilon B^*}
-\varepsilon\left(\partial_x  \mathcal{T}_1  \textbf{B}\right)\big|_{\rho=1}+O(\varepsilon^2), \\
\partial_yP_{\varepsilon B^*}  \textbf{B}\big|_{\rho=1+\varepsilon B^*}
&=\partial_yP_0  \textbf{B}\big|_{\rho=1+\varepsilon B^*}
-\varepsilon\left(\partial_y  \mathcal{T}_1  \textbf{B}\right)\big|_{\rho=1}+O(\varepsilon^2).
\end{split}
\end{equation}
Note that 
\begin{align*}
\partial_xP_0  \textbf{B}\big|_{\rho=1+\varepsilon B^*}&=
\partial_xP_0  \textbf{B}\big|_{\rho=1}+\varepsilon B^*(\theta)\partial_{\rho}\partial_xP_0 \textbf{B} \big|_{\rho=1}+O(\varepsilon^2)\\
&=\partial_xP_0  \textbf{B}\big|_{\rho=1}
+2\varepsilon B^*(\theta)\cos\theta \sum_{n\ge1}B_nn(n-1)\cos n\theta \\
&\quad +2\varepsilon B^*(\theta)\sin\theta\sum_{n\ge1}B_nn(n-1)\sin n\theta+O(\varepsilon^2),\\
\partial_yP_0  \textbf{B}\big|_{\rho=1+\varepsilon B^*}&=
\partial_yP_0  \textbf{B}\big|_{\rho=1}+\varepsilon B^*(\theta)\partial_{\rho}\partial_yP_0 \textbf{B} \big|_{\rho=1}+O(\varepsilon^2)\\
&=\partial_yP_0  \textbf{B}\big|_{\rho=1}
+2\varepsilon B^*(\theta)\sin\theta \sum_{n\ge1}B_nn(n-1)\cos n\theta \\
&\quad -2\varepsilon B^*(\theta)\cos\theta\sum_{n\ge1}B_nn(n-1)\sin n\theta+O(\varepsilon^2).
\end{align*}
Taking these into \eqref{eq: px-py-expan} and using the definition of $ \mathcal{T}_2 $ and $ \mathcal{T}_3 $ yield the result.

\end{proof}
From the definition of $ \mathcal{T}_1 $, we record that
\begin{align*}
\partial_\rho(  \mathcal{T}_1  \textbf{B})\big|_{\rho=1}
&=2t^*  \sum_{n\ge1}B_nn^2\cos n\theta
+C^*\sum_{n\ge1}B_nn(n+3)\cos(n+3)\theta \\
&\quad +C^*\sum_{n\ge3}B_nn(n-3)\cos(n-3)\theta+2C^*B_2 \cos\theta.
\end{align*}
A direct consequence of Lemma \ref{lemma:expan of px-PB} is the asymptotics of $ \nabla\Phi_{\varepsilon, B^*,\textbf{B}} $ on the boundary $ \rho=1+\varepsilon B^*. $
\begin{lemma}\label{lemma: expan of px-Phi}
		For $ \textbf{B} \in X^{s+1} $ and any small enough $\varepsilon$,	
\begin{align*}
\partial_x\Phi_{\varepsilon,B^*, \textbf{B}}\big|_{\rho=1+\varepsilon B^*}
=&-2\varepsilon A_0  \partial_xP_0  \textbf{B}\big|_{\rho=1}+\varepsilon^2\Bigl[-2A_0 \left(\mathcal{T}_2  \textbf{B}-\partial_x  \mathcal{T}_1  \textbf{B}\big|_{\rho=1}\right)
   \\
& +A_0\partial_x\mathcal{T} \textbf{B}\big|_{\rho=1}-2A_1\partial_xP_0(  \cos\theta\, \textbf{B})\big|_{\rho=1}+h^2\sqrt{F_R}\partial_xP_0  \textbf{B}\big|_{\rho=1}
 \\
&  +4A_0C^*\partial_xP_0(\cos3\theta\,  \textbf{B} )\big|_{\rho=1}-2A_0t^*\partial_xP_0  \textbf{B}\big|_{\rho=1}\Bigr]+O(\varepsilon^3), \\
\partial_y\Phi_{\varepsilon,B^*,  \textbf{B}}\big|_{\rho=1+\varepsilon B^*}
=&-2\varepsilon A_0   \partial_y P_0  \textbf{B}\big|_{\rho=1}+\varepsilon^2\Bigl[-2  A_0 \left(\mathcal{T}_3  \textbf{B}-\partial_y  \mathcal{T}_1  \textbf{B}\big|_{\rho=1}\right)
  \\
&  +A_0\partial_y\mathcal{T}  \textbf{B}\big|_{\rho=1}-2A_1\partial_yP_0(  \cos\theta\, \textbf{B})\big|_{\rho=1}+h^2\sqrt{F_R}\partial_yP_0  \textbf{B}\big|_{\rho=1}
 \\
&  +4A_0C^*\partial_yP_0(  \cos3\theta\, \textbf{B}  )\big|_{\rho=1}-2A_0t^*\partial_yP_0  \textbf{B}\big|_{\rho=1}\Bigr]+O(\varepsilon^3).
\end{align*}
\end{lemma}
\begin{proof}
It follows from Proposition \ref{prop:variation} that
\begin{align*}
\partial_x\Phi_{\varepsilon,B^*, \textbf{B}}\big|_{\rho=1+\varepsilon B^*}
=&-2\varepsilon A_0\partial_xP_{\varepsilon B^*}  \textbf{B}\big|_{\rho=1+\varepsilon B^*} \\
& +\varepsilon^2\Bigl[A_0\partial_x\mathcal{T} \textbf{B}\big|_{\rho=1}-2A_1\partial_xP_0(  \cos\theta\, \textbf{B})\big|_{\rho=1}
+h^2\sqrt{F_R}\partial_xP_0  \textbf{B}\big|_{\rho=1} \\
&  +4A_0C^*\partial_xP_0(\cos3\theta\,  \textbf{B} )\big|_{\rho=1}-2A_0t^*\partial_xP_0  \textbf{B}\big|_{\rho=1}\Bigr]+O(\varepsilon^3), \\
\partial_y\Phi_{\varepsilon,B^*,  \textbf{B}}\big|_{\rho=1+\varepsilon B^*}
=&-2\varepsilon A_0\partial_yP_{\varepsilon B^*}  \textbf{B}\big|_{\rho=1+\varepsilon B^*} \\
& +\varepsilon^2\Bigl[A_0\partial_y\mathcal{T}  \textbf{B}\big|_{\rho=1}-2A_1\partial_yP_0(  \cos\theta\, \textbf{B})\big|_{\rho=1}
+h^2\sqrt{F_R}\partial_yP_0  \textbf{B}\big|_{\rho=1} \\
&  +4A_0C^*\partial_yP_0(  \cos3\theta\, \textbf{B}  )\big|_{\rho=1}-2A_0t^*\partial_yP_0  \textbf{B}\big|_{\rho=1}\Bigr]+O(\varepsilon^3).
\end{align*}
Taking 
Lemma \ref{lemma:expan of px-PB} into these formulas, we get the expansion of $ \partial_i\Phi_{\varepsilon,B^*, \textbf{B}}\big|_{\rho=1+\varepsilon B^*} $ for $ i=x,y $ up to order $  \varepsilon^3  $, as claimed.
\end{proof}

\subsection{Invertiblity of $ D_B\mathcal{G}(0,B^*) $}
In the next theorem we derive the key property of the map $\mathcal{G}$: as its domain consists of the even functions orthogonal to $\cos\theta$, we can show that its derivative with respect to $B$ at certain points is an invertible map.

\begin{theorem}\label{thm:invert}
	For any $R>0$, the Fr\'echet derivative
	\[
	D_B\mathcal{G}(0,B^*) : X^{s+1} \to X^s
	\]
	is one-to-one.
\end{theorem}


\begin{proof}

It follows from the definition of $\mathcal{F}$ (Equation~\eqref{eq:Fboundary}) and of $\Phi_{\varepsilon,B^*,\textbf{B}}$ that
\begin{align*}
D_B\mathcal F(\varepsilon,B^*)  \textbf{B} 
&=I_1+I_2+I_3 \\
&\quad -c_{\varepsilon,B^*}\Bigl(
2(R+\varepsilon\sigma_1(1+\varepsilon B^*)\cos\theta)\varepsilon^2\sigma_1 \cos\theta\, \textbf{B}
+2\varepsilon^3\sigma_2^2(1+\varepsilon B^*)\sin^2\theta\,  \textbf{B}\Bigr) \\
&\quad -\widetilde C_{\varepsilon,B^*,\textbf{B}}
\Bigl((R+\varepsilon\sigma_1(1+\varepsilon B^*)\cos\theta)^2+(\varepsilon\sigma_2(1+\varepsilon B^*)\sin\theta)^2+h^2\Bigr), 
\end{align*}
where
\begin{align*}
I_1:=&\lim_{t\to0}\frac{
	\sigma_1^{-2}\bigl(h^2+(\varepsilon\sigma_2y)^2\bigr)(\partial_x\phi_{\varepsilon,B^*+t  \textbf{B}})^2\big|_{\rho=1+\varepsilon(B^*+t  \textbf{B})}-
	\sigma_1^{-2}\bigl(h^2+(\varepsilon\sigma_2y)^2\bigr)(\partial_x\phi_{\varepsilon,B^*})^2\big|_{\rho=1+\varepsilon B^*}}{t}  \\
=&2\sigma_1^{-2}\varepsilon^3 \sigma_2^2\left(1+\varepsilon B^*\right) \sin ^2 \theta\,  \textbf{B}\left(\partial_x \phi_{\varepsilon, B^*}\right)^2\big|_{\rho=1+\varepsilon B^*}\\
&+2\sigma_1^{-2}\left(h^2+\left(\varepsilon \sigma_2 y\right)^2\right)    \partial_x \phi_{\varepsilon, B^*} \partial_x \Phi_{\varepsilon, B^*, \textbf{B}}\big|_{\rho=1+\varepsilon B^*} \\
& +\sigma_1^{-2}\left(h^2+\varepsilon^2 \sigma_2^2 y^2\right) \partial_\rho\left(\partial_x \phi_{\varepsilon, B^*}\right)^2\big|_{\rho=1+\varepsilon B^*} \cdot \varepsilon \textbf{B},\\
I_2:=&\lim_{t\to0}\frac{
	\sigma_2^{-2}\bigl(h^2+(R+\varepsilon\sigma_1x)^2\bigr)(\partial_y\phi_{\varepsilon,B^*+t  \textbf{B}})^2\big|_{\rho=1+\varepsilon(B^*+t  \textbf{B})}-
	\sigma_2^{-2}\bigl(h^2+(R+\varepsilon\sigma_1x)^2\bigr)(\partial_y\phi_{\varepsilon,B^*})^2\big|_{\rho=1+\varepsilon B^*}}{t}  \\
= &  2\sigma_2^{-2}\left(R+\varepsilon \sigma_1\left(1+\varepsilon B^*\right) \cos \theta\right)  \varepsilon^2 \sigma_1 \cos \theta\, \textbf{B}\left(\partial_y \phi_{\varepsilon, B^*}\right)^2 \big|_{\rho=1+ \varepsilon B^*}\\
&+ 2\sigma_2^{-2}\left(h^2+\left(R+\varepsilon \sigma_1 x\right)^2\right)    \partial_y \phi_{\varepsilon,   B^*}   \partial_y \Phi_{\varepsilon, B^*, \textbf{B}}\big|_{\rho=1+\varepsilon B^*} \\
& +\sigma_2^{-2}\left(h^2+\left(R+\varepsilon \sigma_1 x\right)^2\right)   \partial_\rho\left(\partial_y \phi_{\varepsilon, B^*}\right)^2\big|_{\rho=1+\varepsilon B^*}\cdot \varepsilon \textbf{B},\\
I_3:=&\lim_{t\to0}\frac{-2(R+\varepsilon\sigma_1x)\varepsilon\sigma_1^{-1}y
	(\partial_x\phi_{\varepsilon,B^*+t  \textbf{B}}\partial_y\phi_{\varepsilon,B^*+t  \textbf{B}})\big|_{\rho=1+\varepsilon(B^*+t  \textbf{B})}  + 2(R+\varepsilon\sigma_1x)\varepsilon\sigma_1^{-1}y
	(\partial_x\phi_{\varepsilon,B^*}\partial_y\phi_{\varepsilon,B^*})\big|_{\rho=1+\varepsilon B^*}}{t}\\
=&- 2 \varepsilon^2 R \sigma_1^{-1} \sin \theta\, \textbf{B}  \left(\partial_x \phi_{\varepsilon, B^*}  \partial_y \phi_{\varepsilon, B^*}\right)\big|_{\rho=1+\varepsilon B^*}\\
&-2 \varepsilon \sigma_1^{-1}\left(R+\varepsilon \sigma_1\left(1+\varepsilon B^*\right) \cos \theta\right)\left(1+\varepsilon B^*\right) \sin \theta\left(\partial_x \phi_{\varepsilon, B^*} \partial_y \Phi_{\varepsilon, B^*, \textbf{B}}+\partial_x \Phi_{\varepsilon, B^*, \textbf{B}} \partial_y \phi_{\varepsilon, B^*}\right)\big|_{\rho=1+\varepsilon B^*}  \\
&-2 \varepsilon \sigma_1^{-1}\left(R+\varepsilon \sigma_1\left(1+\varepsilon B^*\right) \cos \theta\right)\left(1+\varepsilon B^*\right) \sin \theta   \partial_\rho\left(\partial_x \phi_{\varepsilon, B^*} \partial_y \phi_{\varepsilon, B^*}\right)\big|_{\rho=1+\varepsilon B^*} \cdot \varepsilon \textbf{B}+O\left(\varepsilon^3\right), 
\end{align*}
and where the constant $\widetilde C_{\varepsilon,B^*,\textbf{B}}$ is given by the derivative
\begin{equation*} 
\widetilde C_{\varepsilon,B^*,\textbf{B}}:=\left.\frac{d}{dt}\right|_{t=0}c_{\varepsilon,B^*+t \textbf{B}}.
\end{equation*}
Combining the above three difference quotients gives
\begin{equation}\label{eq:DBF}
\begin{split}
I_1+I_2+I_3 
&=2 \sigma_2^{-1} 
\partial_x\phi_{\varepsilon,B^*} \partial_x\Phi_{\varepsilon,B^*,  \textbf{B}}\big|_{\rho=1+\varepsilon B^*}  + \sigma_2^{-1} 
\partial_\rho\left ((\partial_x\phi_{\varepsilon,B^*})^2\right )\big|_{\rho=1+\varepsilon B^*}\cdot \varepsilon  \textbf{B} \\
&\quad +2R\sigma_2^{-2}\varepsilon^2\sigma_1\cos\theta\, \textbf{B}\left (\partial_y\phi_{\varepsilon,B^*}\right )^2\big|_{\rho=1+\varepsilon B^*} \\
&\quad +2\left (\sigma_2^{-1}+2R\varepsilon\sigma_1\sigma_2^{-2}(1+\varepsilon B^*)\cos\theta\right )
\partial_y\phi_{\varepsilon,B^*}\partial_y\Phi_{\varepsilon,B^*,  \textbf{B}}\big|_{\rho=1+\varepsilon B^*} \\
&\quad +\left (\sigma_2^{-1}+2R\varepsilon\sigma_1\sigma_2^{-2}(1+\varepsilon B^*)\cos\theta\right )
\partial_\rho\left ((\partial_y\phi_{\varepsilon,B^*})^2\right )\big|_{\rho=1+\varepsilon B^*}\cdot\varepsilon  \textbf{B} \\
&\quad -2\varepsilon^2\sigma_1^{-1}R\sin\theta\, \textbf{B}\left (\partial_x\phi_{\varepsilon,B^*}\partial_y\phi_{\varepsilon,B^*}\right )\big|_{\rho=1+\varepsilon B^*} \\
&\quad -2\varepsilon\sigma_1^{-1}R\sin\theta
\left(\partial_x\phi_{\varepsilon,B^*}\partial_y\Phi_{\varepsilon,B^*,  \textbf{B}}
+\partial_y\phi_{\varepsilon,B^*}\partial_x\Phi_{\varepsilon,B^*,  \textbf{B}}\right)\big|_{\rho=1+\varepsilon B^*} \\
&\quad -2\varepsilon^2\sigma_1^{-1}R\sin\theta\, \textbf{B}
\partial_\rho(\partial_x\phi_{\varepsilon,B^*}\partial_y\phi_{\varepsilon,B^*})\big|_{\rho=1+\varepsilon B^*} +O\left (\varepsilon^3\right ). 
\end{split}
\end{equation}
We readily obtain from formula~\eqref{eq:pxpy-phi} and~Lemma \ref{lemma: expan of px-Phi} that
\begin{align*}
I_1+I_2+I_3 
&=2\sigma_2^{-1}\nabla\phi_{\varepsilon,B^*}\cdot\nabla\Phi_{\varepsilon,B^*,
	  \textbf{B}}\big|_{\rho=1+\varepsilon B^*}
+\sigma_2^{-1}\partial_\rho\left(|\nabla\phi_{\varepsilon,B^*}|^2\right)\big|_{\rho=1+\varepsilon B^*}\cdot\varepsilon  \textbf{B}+O(\varepsilon^2) \\
&=8A_0^2\sigma_2^{-1}\varepsilon(  \textbf{B}-\Lambda_0  \textbf{B})+O(\varepsilon^2). 
\end{align*}
Since $\widetilde C_{\varepsilon,B^*,\textbf{B}}$ is obviously of order $O(\varepsilon)$, cf.\ Proposition~\ref{prop:F_expansion}, it suffices to employ the leading order terms of this expression to arrive at
\[
D_B\mathcal{F}(\varepsilon,B^*) \textbf{B}
= 8\varepsilon A_0^2\sigma_2^{-1} (  \textbf{B} - \Lambda_0  \textbf{B})
- \widetilde C_{\varepsilon,B^*,\textbf{B}} \sigma_2 + O(\varepsilon^2) .
\]
Hence, in order to compute this derivative modulo an error of order $\varepsilon^2$ we only need to derive asymptotics for $ \widetilde C_{\varepsilon,B^*,\textbf{B}}$.
To do this, we write
$$\widetilde C_i:=\left.\frac{d}{dt}\right|_{t=0}c_{\varepsilon,B^*+t \textbf{B}, i}, \qquad i=1,2.$$  Then
\begin{align*}
\widetilde C_{\varepsilon,B^*,\textbf{B}}=\frac{\widetilde C_1}{c_{\varepsilon,B^*,2}}-c_{\varepsilon,B^*}\frac{\widetilde C_2}{c_{\varepsilon,B^*,2}}.
\end{align*}
Let us start with $\widetilde C_2$. Since
\begin{equation*} 
c_{\varepsilon,B^*+t \textbf{B},2}
=\int_0^{2\pi}\left [(R+\varepsilon\sigma_1(1+\varepsilon(B^*+t  \textbf{B}))\cos\theta)^2 + (\varepsilon\sigma_2(1+\varepsilon(B^*+t  \textbf{B}))\sin\theta)^2+h^2\right ]\cos\theta\, d\theta,
\end{equation*}
we get
\begin{equation*}
\begin{split}
\widetilde C_2=&\int_0^{2\pi}2\left[\left (R+\varepsilon\sigma_1(1+\varepsilon B^*)\cos\theta\right )\varepsilon^2\sigma_1  \textbf{B}\cos\theta
+2\varepsilon^3\sigma_2^2(1+\varepsilon B^*)  \textbf{B}\sin^2\theta\right]
\cos\theta\, d\theta \\
=&2R\varepsilon^2\sigma_1\langle  \textbf{B}, \cos^2\theta\rangle+O(\varepsilon^3)\\
=&2\pi R\sigma_1(B_0+B_2)\varepsilon^2+O(\varepsilon^3).
\end{split}
\end{equation*}
To compute $\widetilde C_1$, we again employ the formula~\eqref{eq:DBF}, now to third order:
\begin{align}\label{eq: C_1}
\widetilde C_1 =	&\left\langle I_1+I_2+I_3, \cos \theta\right\rangle\notag\\
		=&\Big\langle 2\sigma_2^{-1}\nabla\phi_{\varepsilon,B^*}\cdot\nabla\Phi_{\varepsilon,B^*,
			\textbf{B}}\big|_{\rho=1+\varepsilon B^*}
		+\sigma_2^{-1}\partial_\rho\left(|\nabla\phi_{\varepsilon,B^*}|^2\right)\big|_{\rho=1+\varepsilon B^*}\cdot\varepsilon  \textbf{B} \notag\\
		& +2 R \sigma_1\sigma_2^{-2}  \varepsilon^2 \textbf{B} \cos \theta\left(\partial_y \phi_{\varepsilon, B^*}\right)^2\big|_{\rho=1+ \varepsilon B^*} \notag\\
		& +\left.4 R \varepsilon \sigma_1 \sigma_2^{-2}\left(1+\varepsilon B^*\right) \cos \theta\left(\partial_y \phi_{\varepsilon, B^*}   \partial_y \Phi_{\varepsilon,   B^*, \textbf{B}}\right)\right|_{\rho=1+\varepsilon B^*} \notag\\
		& +\left.2 R \varepsilon^2 \sigma_1 \sigma_2^{-2}\left(1+\varepsilon B^*\right) \cos \theta\, \textbf{B}   \partial_\rho\left(\partial_y \phi_{\varepsilon, B^*}\right)^2\right|_{\rho=1+\varepsilon B^*} \notag\\
		& -2 \varepsilon^2 \sigma_1^{-1} R \sin \theta\, \textbf{B}\left(\partial_x \phi_{\varepsilon, B^*} \partial_y \phi_{\varepsilon, B^*}\right)\big|_{\rho=1+\varepsilon B^*} \notag\\
		& -2 \varepsilon \sigma_1^{-1} R \sin \theta\left(\partial_x \phi_{\varepsilon, B^*} \partial_y \Phi_{\varepsilon, B^*, \textbf{B}}+\partial_x \Phi_{\varepsilon,B^*, \textbf{B}} \partial_y \phi_{\varepsilon, B^*}\right)\big|_{\rho=1+ \varepsilon B^*} \notag\\
		& -\left.2 \varepsilon^2 \sigma_1^{-1} R \sin \theta\, \textbf{B} \partial_\rho\left(\partial_x \phi_{\varepsilon, B^*} \partial_y \phi_{\varepsilon, B^*}\right)\right|_{\rho=1+\varepsilon B^*} +O\left(\varepsilon^3\right), \cos \theta\Big\rangle \notag\\
		=&: \left\langle P_1+P_2+P_3+P_4+P_5+P_6+P_7+P_8, \cos \theta\right\rangle+O\left(\varepsilon^3\right) . 
\end{align}

We need to compute  $ \sum_{i=1}^8\left\langle P_i, \cos \theta\right\rangle $ appearing in formula \eqref{eq: C_1} in terms of the Fourier coefficients $ B_n$ (note that $  B_{-n} =  B_n$ because $\textbf{B}$ is even). For $ i=3,\cdots,8 $, we can compute directly the terms $ P_i $  to the order $ \varepsilon^3 $ one by one:
\begin{align*}
P_3&=2R\sigma_1\sigma_2^{-2}\varepsilon^2  \textbf{B}\cos\theta\,(4A_0^2\sin^2\theta+O(\varepsilon)) =8R\sigma_1\sigma_2^{-2}A_0^2\varepsilon^2  \textbf{B}\cos\theta\sin^2\theta+O\left (\varepsilon^3\right ), \\
P_4&=4R\varepsilon\sigma_1\sigma_2^{-2}\cos\theta
\left(2A_0\sin\theta\,(-2)\varepsilon A_0\partial_yP_0\textbf{B}|_{\rho=1}+O(\varepsilon^2)\right)+O\left (\varepsilon^3\right ) \\
&=-16RA_0^2\sigma_1\sigma_2^{-2}\varepsilon^2\cos\theta\sin\theta
\left (\sin\theta\Lambda_0\textbf{B}+\cos\theta M_0\textbf{B}\right )+O\left (\varepsilon^3\right ), \\
P_5&=2R\varepsilon^2\sigma_1\sigma_2^{-2}\cos\theta\,\textbf{B}\left (8A_0^2\sin^2\theta+O(\varepsilon)\right )  =16RA_0^2\sigma_1\sigma_2^{-2}\varepsilon^2\textbf{B}\cos\theta\sin^2\theta+O\left (\varepsilon^3\right ), \\
P_6&=-2\varepsilon^2\sigma_1^{-1}R\sin\theta\,\textbf{B}\left (4A_0^2\cos\theta\sin\theta+O(\varepsilon)\right )  =-8RA_0^2\sigma_1^{-1}\varepsilon^2\textbf{B}\cos\theta\sin^2\theta+O\left (\varepsilon^3\right ), \\
P_7&=-2\varepsilon\sigma_1^{-1}R\sin\theta
\left(2A_0\cos\theta(-2)\varepsilon A_0\partial_yP_0\textbf{B}|_{\rho=1}+2A_0\sin\theta(-2)\varepsilon A_0\partial_xP_0\textbf{B}|_{\rho=1}\right)+O\left (\varepsilon^3\right ) \\
&=8RA_0^2\sigma_1^{-1}\varepsilon^2\sin\theta
\left (\Lambda_0 \textbf{B}\, \sin2\theta+
 M_0\textbf{B}\, \cos2\theta\right )+O\left (\varepsilon^3\right ),\\
P_8&=-2\varepsilon^2\sigma_1^{-1}R\sin\theta\,\textbf{B}\left (8A_0^2\cos\theta\sin\theta+O(\varepsilon)\right )  =-16RA_0^2\sigma_1^{-1}\varepsilon^2\textbf{B}\cos\theta\sin^2\theta+O\left (\varepsilon^3\right ).
\end{align*}
For $P_1$, we obtain from formula~\eqref{eq:pxpy-phi} and~Lemma \ref{lemma: expan of px-Phi} that
\begin{align*}
P_1
=&2\sigma_2^{-1}\left (2A_0\cos\theta+\varepsilon[2A_1\cos^2\theta-2A_0\partial_xP_0B^*|_{\rho=1}
-h^2\sqrt{F_R}\cos\theta+2A_0B^*(\theta)\cos\theta]+O(\varepsilon^2)\right ) \\
& \times\Big(-2\varepsilon A_0  \partial_xP_0  \textbf{B}\big|_{\rho=1}
+\varepsilon^2\Bigl[-2A_0 \left(\mathcal{T}_2  \textbf{B}-\partial_x  \mathcal{T}_1  \textbf{B}\big|_{\rho=1}\right)+A_0\partial_x\mathcal{T} \textbf{B}\big|_{\rho=1}-2A_1\partial_xP_0(  \cos\theta\, \textbf{B})\big|_{\rho=1}
\\
&\quad +h^2\sqrt{F_R}\partial_xP_0  \textbf{B}\big|_{\rho=1} +4A_0C^*\partial_xP_0(\cos3\theta\,  \textbf{B} )\big|_{\rho=1}-2A_0t^*\partial_xP_0  \textbf{B}\big|_{\rho=1}\Bigr]+O(\varepsilon^3)\Bigr) \\
& +2\sigma_2^{-1}
\left (2A_0\sin\theta+\varepsilon[ A_1\sin2\theta-2A_0\partial_yP_0B^*|_{\rho=1}
-h^2\sqrt{F_R}\sin\theta+2A_0B^*(\theta)\sin\theta]+O(\varepsilon^2)\right ) \\
& \times\Bigl(-2\varepsilon A_0   \partial_y P_0  \textbf{B}\big|_{\rho=1}
  +\varepsilon^2\Bigl[-2  A_0 \left(\mathcal{T}_3  \textbf{B}-\partial_y  \mathcal{T}_1  \textbf{B}\big|_{\rho=1}\right)+A_0\partial_y\mathcal{T}  \textbf{B}\big|_{\rho=1}-2A_1\partial_yP_0(  \cos\theta\, \textbf{B})\big|_{\rho=1}
\\
& \quad+h^2\sqrt{F_R}\partial_yP_0  \textbf{B}\big|_{\rho=1} +4A_0C^*\partial_yP_0(  \cos3\theta\, \textbf{B}  )\big|_{\rho=1}-2A_0t^*\partial_yP_0  \textbf{B}\big|_{\rho=1}\Bigr]+O(\varepsilon^3)\Bigr).
\end{align*}
Note that the $\varepsilon$ coefficient of $P_1$ is $ -8\sigma_2^{-1}A_0^2 \Lambda_0  \textbf{B} $. The $\varepsilon^2$ coefficient of $P_1$   is organized  as
\begin{align*}
&2\sigma_2^{-1}\left( -4A_0^2\left (\cos\theta\,\mathcal{T}_2  \textbf{B}+
\sin\theta\,\mathcal{T}_3  \textbf{B}-\partial_\rho(\mathcal{T}_1  \textbf{B})|_{\rho=1}\right )
+2A_0^2\partial_\rho(\mathcal{T}  \textbf{B})|_{\rho=1}
-4A_0A_1\Lambda_0( \cos\theta\, \textbf{B})\right . \\
&
  +2A_0h^2\sqrt{F_R}\Lambda_0  \textbf{B}
+8C^*A_0^2\Lambda_0( \cos3\theta\,  \textbf{B})-4A_0^2t^*\Lambda_0  \textbf{B}+4A_0^2\partial_xP_0  \textbf{B}|_{\rho=1}
\cdot 3C^*\cos2\theta\\
& 
+4A_0^2\partial_yP_0  \textbf{B}|_{\rho=1}
\cdot(-3)C^*\sin2\theta  \left .-2A_0\Lambda_0  \textbf{B}\left (2A_1\cos\theta-h^2\sqrt{F_R}+2A_0B^*(\theta)\right )\right)\\
&\quad=  -8\sigma_2^{-1}A_0^2\left (\cos\theta\,\mathcal{T}_2  \textbf{B}+
\sin\theta\,\mathcal{T}_3  \textbf{B}-\partial_\rho(\mathcal{T}_1  \textbf{B})|_{\rho=1}\right )
+4\sigma_2^{-1}A_0^2\partial_\rho(\mathcal{T}  \textbf{B})|_{\rho=1}
 \\
&\qquad -8\sigma_2^{-1}A_0A_1\Lambda_0( \cos\theta\, \textbf{B})+8\sigma_2^{-1}A_0h^2\sqrt{F_R}\Lambda_0  \textbf{B}
+16\sigma_2^{-1}C^*A_0^2\Lambda_0( \cos3\theta\, \textbf{B}) \\
&\qquad -8\sigma_2^{-1}A_0^2t^*\Lambda_0  \textbf{B}+24\sigma_2^{-1}C^*A_0^2\partial_xP_0  \textbf{B}|_{\rho=1}\cos2\theta
-24\sigma_2^{-1}C^*A_0^2\partial_yP_0  \textbf{B}|_{\rho=1}\sin2\theta \\
&\qquad  -8\sigma_2^{-1}A_0A_1\cos\theta\Lambda_0  \textbf{B}-8\sigma_2^{-1}A_0^2B^*(\theta)\Lambda_0  \textbf{B},
\end{align*}
where we have used that $ P_0B^*=C^*\rho^3\cos3\theta+t^*=C^*(x^3-3xy^2)+t^* $, $ \partial_xP_0B^*\big|_{\rho=1}=3C^*\cos2\theta,  $ and $ \partial_yP_0B^*\big|_{\rho=1}=-3C^* \sin2\theta. $ Therefore
\begin{align*}
P_1&=-8\sigma_2^{-1}A_0^2 \varepsilon\Lambda_0  \textbf{B} \\
&\quad + \left( -8\sigma_2^{-1}A_0^2\left (\cos\theta\,\mathcal{T}_2  \textbf{B}+
\sin\theta\,\mathcal{T}_3  \textbf{B}-\partial_\rho(\mathcal{T}_1  \textbf{B})|_{\rho=1}\right )
+4\sigma_2^{-1}A_0^2\partial_\rho(\mathcal{T}  \textbf{B})|_{\rho=1}
 \right .\\
&\quad -8\sigma_2^{-1}A_0A_1\Lambda_0( \cos\theta\, \textbf{B})+8\sigma_2^{-1}A_0h^2\sqrt{F_R}\Lambda_0  \textbf{B}
+16\sigma_2^{-1}C^*A_0^2\Lambda_0( \cos3\theta\, \textbf{B}) \\
&\quad -8\sigma_2^{-1}A_0^2t^*\Lambda_0  \textbf{B}+24\sigma_2^{-1}C^*A_0^2\partial_xP_0  \textbf{B}|_{\rho=1}\cos2\theta
-24\sigma_2^{-1}C^*A_0^2\partial_yP_0  \textbf{B}|_{\rho=1}\sin2\theta \\
&\quad  \left .-8\sigma_2^{-1}A_0A_1\cos\theta\Lambda_0  \textbf{B}-8\sigma_2^{-1}A_0^2B^*(\theta)\Lambda_0  \textbf{B}\right)\varepsilon^2 +O\left (\varepsilon^3\right ).
\end{align*}
For $ P_2, $ we obtain from Proposition \ref{prop:first-order-dirichlet-expansion} that
\begin{align*}
P_2&=\sigma_2^{-1}\partial_\rho\left((\partial_x\phi_{\varepsilon,B^*})^2+(\partial_y\phi_{\varepsilon,B^*})^2\right)\big|_{\rho=1+\varepsilon B^*}\cdot\varepsilon  \textbf{B} \\
&=\left .\sigma_2^{-1}\varepsilon \textbf{B}\partial_\rho\left (4A_0^2\rho^2+4A_0\varepsilon\left [A_1(3\rho^3-\rho)\cos\theta-2A_0\rho\partial_\rho P_{0}B^*-h^2\sqrt{F_R}\rho^2\right ]+O(\varepsilon^2)\right )\right |_{\rho=1+\varepsilon B^*}\\
&=8\sigma_2^{-1}A_0^2\varepsilon  \textbf{B} \\
&\quad +\left(32\sigma_2^{-1}A_0A_1 \cos\theta\, \textbf{B}-72C^*\sigma_2^{-1}A_0^2\cos3\theta\,  \textbf{B}
-8\sigma_2^{-1}A_0h^2\sqrt{F_R}  \textbf{B}+8\sigma_2^{-1}A_0^2B^* \textbf{B}\right)\varepsilon^2+O\left (\varepsilon^3\right ).
\end{align*}
Substituting the above formulas of $P_1$--$P_8$  into the expression \eqref{eq: C_1} for $\widetilde C_1$, the coefficient of the $O(\varepsilon)$ term is first
\begin{align*}
8\sigma_2^{-2}A_0^2 \langle  \textbf{B}-\Lambda_0  \textbf{B}, \cos\theta\rangle=0,
\end{align*}
where we have used that $\Lambda_0$ is self-adjoint and that $\Lambda_0(\cos\theta)=\cos\theta$. The coefficient of the $O(\varepsilon^2)$ term gives
\begin{align}\label{eq:tilde C_1 expansion}
&-8\sigma_2^{-1}A_0^2\left(
\langle \mathcal{T}_2\textbf{B}, \cos^2\theta\rangle
+\langle \mathcal{T}_3\textbf{B}, \sin\theta\cos\theta\rangle
-\langle \partial_\rho(\mathcal{T}_1\textbf{B})|_{\rho=1}, \cos\theta \rangle\right) \notag\\
&\quad -4\sigma_2^{-1}A_0^2 \langle   \mathcal{T}'  \textbf{B}, \cos\theta\rangle
-8\sigma_2^{-1}A_0A_1\langle \textbf{B}, \cos^2\theta\rangle+8\sigma_2^{-1}A_0h^2\sqrt{F_R}\langle  \Lambda_0  \textbf{B}-\textbf{B}  ,\cos\theta\rangle \notag\\
&\quad+16\sigma_2^{-1}C^*A_0^2\langle\Lambda_0( \cos3\theta\, \textbf{B}),\cos\theta\rangle-8\sigma_2^{-1}A_0^2t^*\langle\Lambda_0  \textbf{B}-\textbf{B}, \cos\theta\rangle \notag\\
&\quad +24\sigma_2^{-1}C^*A_0^2\langle \partial_xP_0\textbf{B}|_{\rho=1}\cos2\theta-
	\partial_yP_0\textbf{B}|_{\rho=1}\sin2\theta, \cos\theta\rangle \notag\\
&\quad -8\sigma_2^{-1}A_0A_1\langle \Lambda_0\textbf{B},\cos^2\theta\rangle
-8\sigma_2^{-1}C^*A_0^2\langle \Lambda_0\textbf{B}, \cos3\theta\cos\theta\rangle-8\sigma_2^{-1}A_0^2t^*\langle \Lambda_0\textbf{B}, \cos\theta\rangle \notag\\
&\quad +32\sigma_2^{-1}A_0A_1\langle \textbf{B}, \cos^2\theta\rangle
-64\sigma_2^{-1}C^*A_0^2\langle\textbf{B},\cos3\theta\cos\theta\rangle +8R\sigma_1\sigma_2^{-2}A_0^2\langle \textbf{B}, \cos^2\theta\sin^2\theta\rangle \notag\\
&\quad -16R\sigma_1\sigma_2^{-2}A_0^2
\left(\langle \Lambda_0\textbf{B}, \cos^2\theta\sin^2\theta\rangle
+\langle M_0\textbf{B}, \cos^3\theta\sin\theta\rangle\right) \notag\\
&\quad +16R\sigma_1\sigma_2^{-2}A_0^2\langle \textbf{B}, \cos^2\theta\sin^2\theta\rangle
-8R\sigma_1^{-1}A_0^2\langle \textbf{B}, \cos^2\theta\sin^2\theta\rangle \notag\\
&\quad +8R\sigma_1^{-1}A_0^2\left(\langle \Lambda_0\textbf{B}, 2\sin^2\theta\cos^2\theta\rangle
+\langle M_0\textbf{B}, \cos2\theta\sin\theta\cos\theta\rangle\right) \notag\\
&\quad -16R\sigma_1^{-1}A_0^2\langle \textbf{B}, \cos^2\theta\sin^2\theta\rangle. 
\end{align}

To compute this, we need the following scalar products 
\begin{align*}
\langle  \textbf{B}, \cos^2\theta\rangle
&= \frac14 \langle  \textbf{B}, e^{2i\theta} + e^{-2i\theta} + 2 \rangle
= \pi (  B_0 +   B_2) , \\
\langle \Lambda_0 \textbf{B}, \cos^2\theta\rangle
&= \frac14 \left \langle \textbf{B}, \Lambda_0\left ( e^{2i\theta} + e^{-2i\theta} + 2 \right ) \right \rangle
= 2\pi   B_2 , \\
 \langle   \textbf{B}, \cos^2\theta\sin^2\theta\rangle
&=\frac18 \langle \textbf{B}, 1-\cos4\theta\rangle=\frac\pi4(B_0-B_4), \\
 \langle   \textbf{B}, \cos3\theta\cos\theta\rangle
&=\frac12 \langle \textbf{B}, \cos2\theta+\cos4\theta\rangle= \pi(B_2+B_4), \\
\langle \Lambda_0 \textbf{B}, \cos^2\theta\sin^2\theta\rangle
&=\frac1{16} \left \langle\textbf{B}, \Lambda_0\left (2-e^{4i\theta}-e^{-4i\theta}\right )\right \rangle =-\frac1{4} \langle\textbf{B}, e^{4i\theta}+e^{-4i\theta} \rangle  
=-\pi B_4,\\
\langle \Lambda_0 \textbf{B}, \cos3\theta\cos\theta\rangle
&=\frac12\langle \Lambda_0\textbf{B}, \cos2\theta+
	\cos4\theta\rangle =\langle \textbf{B}, \cos2\theta\rangle+2\langle \textbf{B}, \cos4\theta\rangle
=2\pi B_2+4\pi B_4, \\
\langle M_0 \textbf{B}, \sin3\theta\cos\theta\rangle
&=\frac12\left\langle i\sum_{n\in \mathbb{Z}}B_nn e^{in\theta},
\frac{e^{4i\theta}-e^{-4i\theta}}{2i}+\frac{e^{2i\theta}-e^{-2i\theta}}{2i}\right\rangle \\
&=\frac\pi4(-16B_4 -8B_2)=-4\pi B_4-2\pi B_2,\\
\langle M_0 \textbf{B}, \cos^3\theta\sin\theta\rangle
&=\frac14\left\langle i\sum_{n\in\mathbb{Z}} B_nn e^{in\theta},
\frac{e^{2i\theta}-e^{-2i\theta}}{2i}\right\rangle
+\frac18\left\langle i\sum_{n\in\mathbb{Z}} B_nn e^{in\theta},
\frac{e^{4i\theta}-e^{-4i\theta}}{2i}\right\rangle \\
&=\frac {2\pi}8\left ((-2)B_{-2} -2B_2 \right )
+\frac{2\pi}{16}\left ((-4)B_{-4} -4B_4\right ) \\
&=-\pi B_2-\pi B_4, \\
\langle M_0\textbf{B}, \cos2\theta\sin\theta\cos\theta\rangle
&=\frac14\left\langle i\sum_{n\in\mathbb{Z}} B_nn e^{in\theta},
\frac{e^{4i\theta}-e^{-4i\theta}}{2i}\right\rangle= \frac{2\pi}{8}((-4)B_{-4} -4B_4)=-2\pi B_4,\\
\langle \mathcal{T}' \textbf{B}, \cos\theta\rangle
&=\left\langle 4(I_0-I_2-I_1)B_4 \cos\theta
+(I_0-I_2+I_1+2I_3)B_2\cos\theta
\right.\\
&\qquad\left.+(I_0-I_2-I_1)B_2\cos\theta,\cos\theta\right\rangle \\
&=4(I_0-I_2-I_1)\pi B_4+2(I_0-I_2+I_3)\pi B_2.
\end{align*}
Here we have used that $ \langle \cos n\theta, \cos m\theta\rangle=\pi\delta_{mn}. $ Note also that
\begin{align*}
& \left \langle \partial_xP_0 \textbf{B}\big|_{\rho=1}\cos2\theta-
\partial_yP_0  \textbf{B}\big|_{\rho=1}\sin2\theta, \cos\theta \right \rangle\\
&\quad = \left \langle(\cos\theta\partial_\rho(P_0\textbf{B})-\sin\theta\partial_\theta(P_0\textbf{B}))\big|_{\rho=1}\cos2\theta  -(\sin\theta\partial_\rho P_0\textbf{B}+\cos\theta\partial_\theta P_0\textbf{B})\big|_{\rho=1}\sin2\theta, \cos\theta \right \rangle \\
&\quad=\left \langle \Lambda_0\textbf{B}\cos3\theta-M_0\textbf{B}\sin3\theta, \cos\theta \right \rangle\\
&\quad =4\pi B_2+8\pi B_4.
\end{align*}
Also, it follows from the definition of $ \mathcal{T}_1, \mathcal{T}_2 $ and $ \mathcal{T}_3 $ that
\begin{align*}
\left \langle \partial_\rho(  \mathcal{T}_1  \textbf{B})\big|_{\rho=1}, \cos\theta \right \rangle
&=\left \langle 2t^*  \sum_{n\ge1}B_nn^2\cos n\theta
+C^*\sum_{n\ge1}B_nn(n+3)\cos(n+3)\theta \right . \\
&\quad \left.+C^*\sum_{n\ge3}B_nn(n-3)\cos(n-3)\theta+2C^*B_2 \cos\theta, \cos\theta \right \rangle\\
&=4\pi C^*B_4+2\pi C^*B_2,
\end{align*}
and
\begin{align*}
\langle \mathcal{T}_2 \textbf{B}, \cos^2\theta\rangle+\langle \mathcal{T}_3\textbf{B}, \cos\theta\sin\theta\rangle
&=2\left\langle\sum_{n\ge2}B_nn(n-1)\cos n\theta,
(C^*\cos3\theta+t^*)\cos\theta\right\rangle \\
&=C^*\left\langle\sum_{n\ge2}B_nn(n-1)\cos n\theta,
\cos2\theta+
\cos4\theta\right\rangle \\
&= 2\pi C^*B_2+12\pi C^*B_4.
\end{align*}

Taking these identities into formula \eqref{eq:tilde C_1 expansion},   the coefficient of the $O(\varepsilon^2)$ term of $\widetilde C_1$ is given by
\begin{align*}
	& -8 \sigma_2^{-1} A_0^2\cdot 8 \pi C^* B_4 -4 \sigma_2^{-1} A_0^2\left(4\left(I_0-I_2-I_1\right) \pi B_4+2\left(I_0-I_2+I_3\right) \pi B_2\right)  \\
	& +24\sigma_2^{-1} A_0 A_1 \pi\left(B_0+B_2\right) -48 \sigma_2^{-1} C^* A_0^2 \pi\left(B_2+B_4\right)+24 \sigma_2^{-1} C^* A_0^2\left(4 \pi B_2+8 \pi B_4\right)   \\
	& -8 \sigma_2^{-1} A_0 A_1 \cdot 2 \pi B_2-8 \sigma_2^{-1} C^* A_0^2\left(2 \pi B_2+4 \pi B_4\right)  +6R \sigma_1 \sigma_2^{-2} A_0^2 \pi\left(B_0-B_4\right) \\
	& -16 R \sigma_1 \sigma_2^{-2} A_0^2\left(-2 \pi B_4-\pi B_2\right)   -6 R \sigma_1^{-1} A_0^2 \pi\left(B_0-B_4\right) +8 R \sigma_1^{-1} A_0^2\left(-4 \pi B_4\right)  \\
	&\quad =\left(24 \pi \sigma_2^{-1} A_0 A_1  -6 \pi R^3 \sigma_1^{-1} \sigma_2^{-1} A_0^2\right)B_0\\
	&\qquad+\left(8 \pi \sigma_2^{-1} A_0 A_1-8 \pi R \sigma_1^{-1} \sigma_2^{-1} A_0^2\left(R^2+3 h^2\right)+32 \pi \sigma_2^{-1} C^* A_0^2\right)B_2   \\
	&\qquad +\left(6 \pi R^3 \sigma_1^{-1} \sigma_2^{-1} A_0^2+48\pi \sigma_2^{-1} C^* A_0^2\right) B_4, 
\end{align*}
where we have used the definition of $ I_0,I_1,I_2$  and  $ I_3 $. Note that  $ t^* $ and $ h^2\sqrt{F_R} $ terms vanish since $ \langle \textbf{B}, \cos\theta\rangle=0 $ and $ \langle\Lambda_0  \textbf{B}, \cos\theta\rangle=0 $. This implies that
\begin{align*} 
\widetilde C_1
=&\left[ \left(24 \pi \sigma_2^{-1} A_0 A_1  -6 \pi R^3 \sigma_1^{-1} \sigma_2^{-1} A_0^2\right)B_0\right .\\
& +\left(8 \pi \sigma_2^{-1} A_0 A_1-8 \pi R \sigma_1^{-1} \sigma_2^{-1} A_0^2\left(R^2+3 h^2\right)+32 \pi \sigma_2^{-1} C^* A_0^2\right)B_2   \\
& \left .+\left(6 \pi R^3 \sigma_1^{-1} \sigma_2^{-1} A_0^2+48\pi \sigma_2^{-1} C^* A_0^2\right) B_4\right] \varepsilon^2+O\left (\varepsilon^3\right ).
\end{align*} 
Using the formulas for $c_{\varepsilon,B^*,1}$ and $c_{\varepsilon,B^*,2}$ derived in the proof of Proposition~\ref{prop:F_expansion}, this immediately yields
\begin{align*}
\widetilde C_{\varepsilon,B^*,\textbf{B}}=&\frac{\widetilde C_1}{c_{\varepsilon,B^*,2}}-c_{\varepsilon,B^*}\frac{\widetilde C_2}{c_{\varepsilon,B^*,2}}  \\
=& \frac{1}{2\pi R\sigma_1}\left[ \left (16\pi\sigma_2^{-1}A_0A_1-4\pi R^3\sigma_1^{-1}\sigma_2^{-1}A_0^2\right )B_0
	-2\pi R\sigma_1^{-1}\sigma_2^{-1}A_0^2(R^2+12h^2)B_2\right .\\
&\,\,\,\,\,\,\,\,\,\,\quad\quad	\left .+{12\pi R^3\sigma_1^{-1}\sigma_2^{-1}A_0^2B_4}\right]  
\varepsilon+O\left (\varepsilon^2\right ). 
\end{align*}
which results in
\begin{align*}
D_B\mathcal{F}(\varepsilon,B^*) \textbf{B}
=&\left[ 8A_0^2\sigma_2^{-1}\left (  \textbf{B}-\Lambda_0  \textbf{B}\right ) -\left (8R^{-1}\sigma_1^{-1}A_0A_1-2R^2\sigma_1^{-2}A_0^2\right )B_0\right . \notag\\
& 
\left . - \sigma_1^{-2}A_0^2(R^2+12h^2)B_2+6R^2\sigma_1^{-2}A_0^2B_4
\right]
\varepsilon+O(\varepsilon^2).
\end{align*}
We are now ready to analyze the differential of $\mathcal{G}$, which we have shown to be given by the formula
\begin{align*}
D_B\mathcal{G}(0,B^*)  \textbf{B}
= & 8A_0^2\sigma_2^{-1}\left (  \textbf{B}-\Lambda_0  \textbf{B}\right ) -\left (8R^{-1}\sigma_1^{-1}A_0A_1-2R^2\sigma_1^{-2}A_0^2\right )B_0 \\
& - \sigma_1^{-2}A_0^2(R^2+12h^2)B_2+6R^2\sigma_1^{-2}A_0^2B_4, 
\end{align*}
understood as a map $X^{s+1} \to X^s$.
We recall that, as $  \textbf{B}$ is an even function orthogonal to $\cos\theta$, the Fourier series $  \textbf{B} = \sum_{n\in\mathbb{Z}}   B_n e^{in\theta}$ can be equivalently written as
\[
  \textbf{B}(\theta) =   B_0 + 2\sum_{n=2}^\infty   B_n \cos n\theta .
\]
	Therefore, the action of the linear elliptic operator $D_B\mathcal{G}(0,B^*)$ is given by
\begin{align*}
D_B\mathcal{G}(0,B^*)  \textbf{B}
= & \left( 8A_0^2\sigma_2^{-1}  - 8R^{-1}\sigma_1^{-1}A_0A_1+2R^2\sigma_1^{-2}A_0^2\right )B_0 - \sigma_1^{-2}A_0^2(R^2+12h^2)B_2\\
& +6R^2\sigma_1^{-2}A_0^2B_4-16A_0^2\sigma_2^{-1}\sum_{n=2}^\infty (n-1)   B_n \cos n\theta .
\end{align*}	
 Note that $ 8A_0^2\sigma_2^{-1}  - 8R^{-1}\sigma_1^{-1}A_0A_1+2R^2\sigma_1^{-2}A_0^2\neq 0 $ for all $ c,R>0 $ because
\[
8A_0^2\sigma_2^{-1}  - 8R^{-1}\sigma_1^{-1}A_0A_1+2R^2\sigma_1^{-2}A_0^2=-\frac{A_0\sigma_2c}{2}<0 
\]
This implies that the kernel of the map $D_B\mathcal{G}(0,B^*): X^{s+1} \to X^s$ is trivial, and that its range is the whole space $X^s$, as claimed.

\end{proof}

\subsection{Proof of the main theorem} 

Now we are able to prove our main result (Theorem~\ref{thm:main}). In the next corollary we show that, by the implicit function theorem for Banach spaces, Theorem~\ref{thm:invert} yields the existence of solutions to the overdetermined boundary value problem~\eqref{Grad Shaf eqs}, \eqref{eq:dirichlet} and \eqref{eq:bernoulli-bc} for all small enough $\varepsilon$ and all $R>0$.
In turn, these define piecewise smooth stationary Euler flows of compact support via Lemma~\ref{lem:overdet}, thereby completing the proof of the main result of the paper.
Recall that the constant $F_R$ appears in the definition of the function $F$, cf.\ Equation~\eqref{eq:main-F}.

\begin{corollary}\label{cor:existence}
	Fix any $R>0$. Let $ \sigma_1=h\sqrt{h^2+R^2}, \sigma_2=h^2+R^2,  C^*:=R^3\sigma_1^{-1}/8 $ and $ t^*:=5h^2\sqrt{F_R}/(18A_0) $ with $ A_0 = \sigma_2^2c/4. $
	Then, for any small enough $\varepsilon$ there is a unique $B \in X^{s+1}$ in a $C^{s+1}$ neighborhood of $B^*=C^*\cos3\theta+t^*$ such that $\psi := \varepsilon^2 \phi_{\varepsilon,B}$ satisfies Equation~\eqref{Grad Shaf eqs} in $\varOmega_{R,\varepsilon} := \varOmega_{\varepsilon, B}$ and the overdetermined boundary conditions~\eqref{eq:dirichlet}--\eqref{eq:bernoulli-bc} with $F_R := -\kappa > 0$ and $c = \varepsilon^2 c_{\varepsilon,B}$.
\end{corollary}

\begin{proof}
	Since $\mathcal{G}(0,B^*)=0$, in view of Theorem~\ref{thm:invert}, the implicit function theorem guarantees that if $|\varepsilon|$ is small enough, there is a unique function $B$ in a small neighborhood of $B^*$ in $X^{s+1}_1$ such that
	\[
	\mathcal{G}(\varepsilon,B) = 0 .
	\]
	This is equivalent to saying that
	\[
(h^2+x_2^2)(\partial_{x_1}\psi)^2+(h^2+x_1^2)(\partial_{x_2}\psi)^2
-2x_1x_2\partial_{x_1}\psi\,\partial_{x_2}\psi
-c(x_1^2+x_2^2+h^2)- \varepsilon^2 \kappa = 0.
	\]
	on $\partial\varOmega_{\varepsilon, B}$, with $\psi := \varepsilon^2\phi_{\varepsilon,B}$.
	The assumption that $F^2(0) = \varepsilon^2 F_R = -\varepsilon^2\kappa$ then ensures that we have a solution to the overdetermined boundary problem~\eqref{Grad Shaf eqs}, \eqref{eq:dirichlet} and \eqref{eq:bernoulli-bc}, as claimed.
	Observe that for $ R>0 $
	\[
	\kappa = -\frac{\sigma_2^3c^2}{16} < 0 ,
	\]
	and  $F_R=-\kappa > 0$.
	Accordingly, the function $F(\psi)$ is well defined as
	\[
	F (\psi) = \left ( \varepsilon^2 F_R   + O(\psi^2) \right )^{1/2} ,
	\]
	because $\psi = O(\varepsilon^2)$ and $\psi < 0$ in $\varOmega_{\varepsilon, B}$ (cf.\ Proposition~\ref{prop:Dirichlet}).
\end{proof}

To conclude, we finish the proof of Theorem \ref{thm:main}.

	\section{Different choices for the functions $F$ and $H$}\label{sec:different}

As we mentioned in the Introduction, for the sake of concreteness we have chosen the functions $H$ and $F$ as described in Theorem~\ref{thm:main}.
However, the method introduced in this paper is flexible enough to construct stationary Euler flow supported in a neighborhood of a helix with other choices for the functions $H$ and $F$.
To illustrate this additional flexibility, in this section we show how a straightforward modification of the previous computations allows us to prove the following:

\begin{theorem}\label{thm:different}
	Take any non-integer $s>2$ and any functions $\tilde F, H \in C^s((-1,0])$ with
	\[
	\tilde F(0) = \tilde F'(0) = 0 , \qquad H'(0) > 0 .
	\]
	Then the following statements hold:
	\begin{enumerate} 
		\item For each small enough $\varepsilon>0$ and any $R>0$, there exists a nontrivial, piecewise $C^s$, helically symmetric stationary Euler flow of compact support $u$ of the form described in Lemma~\ref{lem:overdet} for a suitable $C^{s+1}$ planar domain $\varOmega_{R,\varepsilon}$.
		\item The boundary of $\varOmega_{R,\varepsilon}$ is a small deformation of an ellipse, given by an equation of the form $\frac{(x_1-R)^2}{\sigma_1^2 }+\frac{x_2^2}{\sigma_2^2}-\varepsilon^2=O(\varepsilon^3).$
		\item The functions that define the solution are
		\[
		F(\psi) := \varepsilon \sqrt{F_R} + \tilde F(\psi)
		\]
		and $H(\psi)$, where $F_R$ is the positive constant
		\begin{equation}\label{eq:FR2}
		F_R := \frac1{16} \left (h^2+R^2\right )^3H'(0)^2.
		\end{equation}
		\item The function $\psi$ is of class $C^{s+1}$ in $\varOmega_{R,\varepsilon}$ up to the boundary, and has the form
		\[
	\psi = \frac14 \left [ \left( h^2+R^2\right)^2 H'(0)  \right ]
	\left [ \frac{(x_1-R)^2}{\sigma_1^2 }+\frac{x_2^2}{\sigma_2^2}-\varepsilon^2 \right ] + O(\varepsilon^3) .
	\]
		Moreover, $F\circ \psi > 0$ and $H\circ \psi$ are of class $C^s$ in $\varOmega_{R,\varepsilon}$.
		In particular, the vorticity is of class $C^{s-1}$ up the boundary.
	\end{enumerate}
\end{theorem}

\begin{proof}
	Indeed, using the same notation as in Section~\ref{sec:Dirichlet}, and noticing that
	\[
 F (\psi)=\varepsilon \sqrt{F_R}+O(\psi^2),\quad 	(F^2)'(\psi) = \varepsilon\,O(\psi),\quad H'(\psi)=H'(0)+O(\psi),
	\]
	Equation~\eqref{eq:phi} still takes the form
\begin{equation*} 
\begin{split}
\Delta\phi={}&\sigma_2^2c+\varepsilon\Bigl(-2R\sigma_1^{-1}xh^2\partial_{xx}\phi+2  R\sigma_1^{-1}\sigma_2y\partial_{xy}\phi-4R\sigma_1\sigma_2^{-1}x\partial_{yy}\phi\\
&
+(R^2+3h^2)R\sigma_1^{-1}\partial_x\phi 
+4R \sigma_1\sigma_2xc-2h^2 \sqrt{F_R}\Bigr)
+O(\varepsilon^2),
\end{split}
\end{equation*}
where we have defined the constant $c = H'(0)$. Repeating all the arguments in Sections~\ref{sec:Dirichlet}--\ref{sec: invertibility}, we obtain the same equations and results.
	In particular, $\mathcal{F}(\varepsilon,B^*) = \kappa + O(\varepsilon^2)$, where the constant $\kappa$  is given by
	\[
	\kappa =-\frac1{16} \left (h^2+R^2\right )^3H'(0)^2 < 0 ,
	\]
and the invertibility condition in Theorem~\ref{thm:invert} holds for $R>0$.
	The Neumann boundary condition is then satisfied taking $F(0) = \varepsilon \sqrt{F_R}$, with $F_R$ as in Equation~\eqref{eq:FR2}.
	Notice that $F(\psi) = \varepsilon \sqrt{F_R} + O(\psi^2) = \varepsilon \sqrt{F_R} + O(\varepsilon^4) > 0$ in $\varOmega_{R,\varepsilon}$.
\end{proof}

\begin{remark}
Actually, it is possible to relax the condition of $ \tilde F $ in Theorem \ref{thm:different} to 
\begin{equation*}
\tilde F(0)=0, \quad \tilde F'(0)\leq 0.
\end{equation*}
We omit the proof here.
\end{remark}

\subsection*{Acknowledgments}

\par
This work is supported by the grants CEX2023-001347-S and PID2022-136795NB-I00 (D.P.-S.) funded by MCIN/AEI /10.13039/501100011033. J.W. was partially supported by NNSF of China (grant  No.  12471190).  

\subsection*{Conflict of interest statement} On behalf of all authors, the corresponding author states that there is no conflict of interest.

\subsection*{Data availability statement} All data generated or analysed during this study are included in this published article  and its supplementary information files.

     \phantom{s}
     \thispagestyle{empty}

\end{document}